\documentclass[a4paper,reqno]{amsart}

\usepackage{amsthm,amsmath,amssymb,amsfonts,epsfig}
\usepackage{hyperref}
\usepackage{stmaryrd}
\usepackage{epstopdf}
\usepackage{eucal}

\numberwithin{equation}{section}

\theoremstyle{plain}
\newtheorem{theorem}{Theorem}[section]

\newtheorem{lemma}[theorem]{Lemma}
\newtheorem{proposition}[theorem]{Proposition}
\newtheorem{corollary}[theorem]{Corollary}
\theoremstyle{definition}
\newtheorem{definition}[theorem]{Definition}
\theoremstyle{remark}
\newtheorem{remark}[theorem]{Remark}
\newtheorem{example}[theorem]{Example}
\newcounter{counter_a}

\newcounter{counter_b}

\newcommand{\cL}{\mathcal{L}}

\newcommand{\Dom}{{\rm Dom}}

\newcommand{\Spec}{{\rm Spec}}
\newcommand{\Span}{{\rm span}}

\renewcommand{\Re}{{\rm Re}\;}
\renewcommand{\Im}{{\rm Im}\;}

\newcommand{\efrac}[2]{\genfrac{}{}{0ex}{}{#1}{#2}}

\newcommand{\dist}{{\rm dist}}

\DeclareMathOperator\spn{span}
\newcommand{\ess}{\mathrm{ess}}
\newcommand{\rank}{\mathrm{Rank}}
\newcommand{\range}{\mathrm{Range}}
\newcommand{\nul}{\mathrm{nul}}

\newcommand{\dis}{\mathrm{dis}}

\newcommand\void[1]{}
\begin{document}
\title[Spectral enclosure and superconvergence for eigenvalues in gaps]{Spectral enclosure and superconvergence for eigenvalues in gaps}
\author[J. Hinchcliffe \& M. Strauss]{James Hinchcliffe$^{\dagger}$ \& Michael Strauss$^{\ddagger}$}
\begin{abstract}
We consider the problem of how to compute eigenvalues of a self-adjoint operator when a direct application of the Galerkin (finite-section) method is unreliable. The last two decades have seen the development of the so-called quadratic methods for addressing this problem. Recently a new perturbation approach has emerged, the idea being to perturb eigenvalues off the real line and, consequently, away from regions where the Galerkin method fails. We propose a simplified perturbation method which requires no \'a priori information and for which we provide a rigorous convergence analysis. The latter shows that, in general, our approach will significantly outperform the quadratic methods. We also present a new spectral enclosure for operators of the form $A+iB$ where $A$ is self-adjoint, $B$ is self-adjoint and bounded. This enables us to control, very precisely, how eigenvalues are perturbed from the real line. The main results are demonstrated with examples including magnetohydrodynamics, Schr\"odinger and Dirac operators.

\vspace{4pt}\noindent\emph{Keywords:} Spectral enclosure, eigenvalue problem, perturbation of eigenvalues, spectral pollution, Galerkin method, finite-section method, superconvergence.

\vspace{4pt}\noindent\emph{2010 Mathematics Subject Classification:} 47A10, 47A55, 47A58, 47A75.

\vspace{4pt}\noindent$\dagger$  Email:  james.hinchcliffe@gmail.com.

\vspace{4pt}\noindent$\ddagger$ Department of Mathematics, University of Sussex, Falmer Campus, Brighton BN1 9QH, UK. Email: m.strauss@sussex.ac.uk.
\end{abstract}

\maketitle
\thispagestyle{empty}

\section{Introduction}

Computational spectral theory for operators which act on infinite dimensional Hilbert spaces has advanced significantly in recent years. For self-adjoint operators, the introduction of \emph{quadratic} methods has enabled the approximation of those eigenvalues which are not reliably located by a direct application of the Galerkin method. The latter is due to \emph{spectral pollution}; see examples \ref{uber2}, \ref{magneto} and \cite{boff,boff2,bost,dasu,DP,lesh,rapp,me4}. Notable amongst these quadratic techniques are the Davies \& Plum method \cite{DP}, the Zimmermann \& Mertins method \cite{zim}, and the second order relative spectra \cite{bo,bb,bole,bost,dav,lesh,shar,shar2,me2,me3}. For spectral approximation of arbitrary operators see, for example, \cite{chat,ah1,ah2,ah3} and references therein.

The present manuscript is concerned with a technique for self-adjoint operators which is pollution-free and non-quadratic. The idea is to perturb eigenvalues into $\mathbb{C}^+$ and then approximate them with the Galerkin method. This idea was initially proposed for a particular  class of differential operators; see \cite{mar1,mar3,mar2}. An abstract version of this approach for bounded self-adjoint operators was formulated in \cite{me}. The latter requires \'a priori information about the location of gaps in the essential spectrum. Our main aims are to remove the requirement of \'a priori information, to present a rigorous convergence analysis, and demonstrate the effectiveness of our method including a comparison with the quadratic methods. Along the way, we also prove new spectral enclosure results for any operator of the form $A+iB$ where $A$ is self-adjoint and $B$ is bounded and self-adjoint; we note that any bounded operator can be expressed in this form. We now give a brief outline of our main results.

In Section 3, we consider the spectra of operators of the form $A+iB$. 
The main result is Theorem \ref{cor1a} where we give our new spectral enclosure results. We will define a region in terms of the spectra of $A$ and $B$, then show that it contains the spectrum of $A+iB$. Corollary \ref{2eigs} shows that the enclosure is, in a sense, sharp.

Section 4 is primarily concerned with the perturbation of an eigenvalue, $\lambda$, of a self-adjoint operator, $A$. We consider $A + iP_n$ where $(P_n)_{n\in\mathbb{N}}$ is a sequence of orthogonal projections. The main results are Theorem \ref{est} and Theorem \ref{QQ} where we prove extremely rapid convergence properties of the eigenspaces and eigenvalues associated to the perturbed eigenvalue.

In Section 5, we present our new perturbation method. The idea is based on applying the Galerkin method to $A+iP_n$ for a fixed $n\in\mathbb{N}$. The preceding results enable us to lift an eigenvalue, $\lambda$, off the real line, away from the essential spectrum, and extremely close to $\lambda+i$ where it can be approximated by a direct application of the Galerkin method. The main results are Theorem \ref{limlem1} and Theorem \ref{eigconv}, where we prove the rapid convergence of Galerkin eigenspaces and eigenvalues.

In Section 6, we apply our method to several operators arising in magnetohydrodynamics, non-relativistic and relativistic quantum mechanics. Most of our examples involve calculations using trial spaces belonging to the form domain and not the operator domain. In particular, we use the FEM spaces of piecewise linear trial functions. However, the quadratic methods require trial spaces from the operator domain. In our last example we use the operator domain which allows a comparison with the quadratic methods.

Let us now fix some notation. Unless stated otherwise, $A$ will denote a semi-bounded (from below) self-adjoint operator acting on a Hilbert space $\mathcal{H}$. The quadratic form, spectrum, resolvent set, discrete spectrum, essential spectrum and spectral measure we denote by $\frak{a}$, $\sigma(A)$, $\rho(A)$, $\sigma_{\dis}(A)$, $\sigma_{\ess}(A)$ and $E$, respectively. For $\Delta\subset\mathbb{R}$ we denote the range of $E(\Delta)$ by $\mathcal{L}(\Delta)$. Associated to the form $\frak{a}$ is the Hilbert space $\mathcal{H}_\frak{a}$ which has inner-product
\[
\langle u,v\rangle_{\frak{a}}:=\frak{a}(u,v) - (m-1)\langle u,v\rangle\quad\forall u,v\in\Dom(\frak{a})\quad\textrm{where}\quad m=\min\sigma(A)
\]
and norm
\begin{equation}\label{anorm}
\Vert u\Vert_{\frak{a}} =\big(\frak{a}(u,u) - (m-1)\langle u,u\rangle\big)^{\frac{1}{2}}=\Vert(A-m+1)^{\frac{1}{2}}u\Vert.
\end{equation}
The gap or distance between two subspaces $\mathcal{M}$ and $\mathcal{N}$ of $\mathcal{H}$, is defined as
\[
\hat{\delta}(\mathcal{M},\mathcal{N}) = \max\big[\delta(\mathcal{M},\mathcal{N}),\delta(\mathcal{N},\mathcal{M})\big]
\textrm{~where~}
\delta(\mathcal{M},\mathcal{N}) = \sup_{u\in\mathcal{M},\Vert u\Vert=1}\dist(u,\mathcal{N});
\]
see \cite[Section IV.2.1]{katopert} for further details. We shall write $\delta_{\frak{a}}$ to indicate the gap between subspaces with respect to the norm \eqref{anorm}.

\section{Auxiliary geometric results}

Throughout this section, we assume that $\alpha,\beta,\gamma,\delta\in\mathbb{R}$ with $-\infty<\alpha<\beta<\infty$ and $-\infty<\gamma<\delta<\infty$.

\begin{definition}\label{gamdef}
The functions $f,g:[0,1]\to\mathbb{C}$ and the region $\mathcal{U}_{\alpha,\beta}^{\gamma,\delta}$, we define as:
\begin{enumerate}
\item if $\beta-\alpha\le \delta-\gamma$
\begin{align*}
\Re f(t) &= \Re g(t) = \alpha(1-t )+ \beta t,\\
\Im  f(t) &= \frac{\gamma+\delta}{2} - \sqrt{\left(\frac{\delta-\gamma}{2}\right)^2+\big(\Re f(t) - \alpha\big)\big(\Re f(t) - \beta\big)},\\
\Im  g(t) &= \frac{\gamma+\delta}{2} + \sqrt{\left(\frac{\delta-\gamma}{2}\right)^2+\big(\Re g(t) - \alpha\big)\big(\Re g(t) - \beta\big)},\\
\mathcal{U}_{\alpha,\beta}^{\gamma,\delta}&:=\Bigg\{z\in\mathbb{C}:~\alpha<\Re z<\beta,~\text{with either}~\\
&\gamma\le\Im z<\Im f\left(\frac{\Re z-\alpha}{\beta-\alpha}\right)
~\textrm{or}~\Im g\left(\frac{\Re z-\alpha}{\beta-\alpha}\right)<\Im z\le \delta\Bigg\},
 \end{align*}
\item if $\beta-\alpha>\delta-\gamma$
\begin{align*}
\Im f(t) &= \Im g(t) = (1-t)\gamma + t\delta,\\
\Re  f(t)&= \frac{\alpha+\beta}{2}-\sqrt{\left(\frac{\beta-\alpha}{2}\right)^2 + \big(\Im f(t) - \gamma\big)\big(\Im f(t) - \delta\big)},\\
\Re  g(t)&= \frac{\alpha+\beta}{2}+\sqrt{\left(\frac{\beta-\alpha}{2}\right)^2 + \big(\Im g(t) - \gamma\big)\big(\Im g(t) - \delta\big)},\\
\mathcal{U}_{\alpha,\beta}^{\gamma,\delta}&:=\Bigg\{z\in\mathbb{C}:~\gamma\le\Im z\le \delta~\text{and}\\
&\hspace{75pt}\Re f\left(\frac{\Im z-\gamma}{\delta-\gamma}\right)<\Re z<\Re g\left(\frac{\Im z-\gamma}{\delta-\gamma}\right)\Bigg\}.
\end{align*}
\end{enumerate}
We also define
\[
\Gamma_{\alpha,\beta}^{\gamma,\delta}:=\big\{z\in\mathbb{C}:\exists t\in[0,1]\textrm{ with }z=f(t)\text{ or }z=g(t)\big\}.
\]
\end{definition}

The curves and regions defined in Definition \ref{gamdef} are demonstrated in Figure 1.

\begin{figure}[h!]
\centering
\includegraphics[scale=.3]{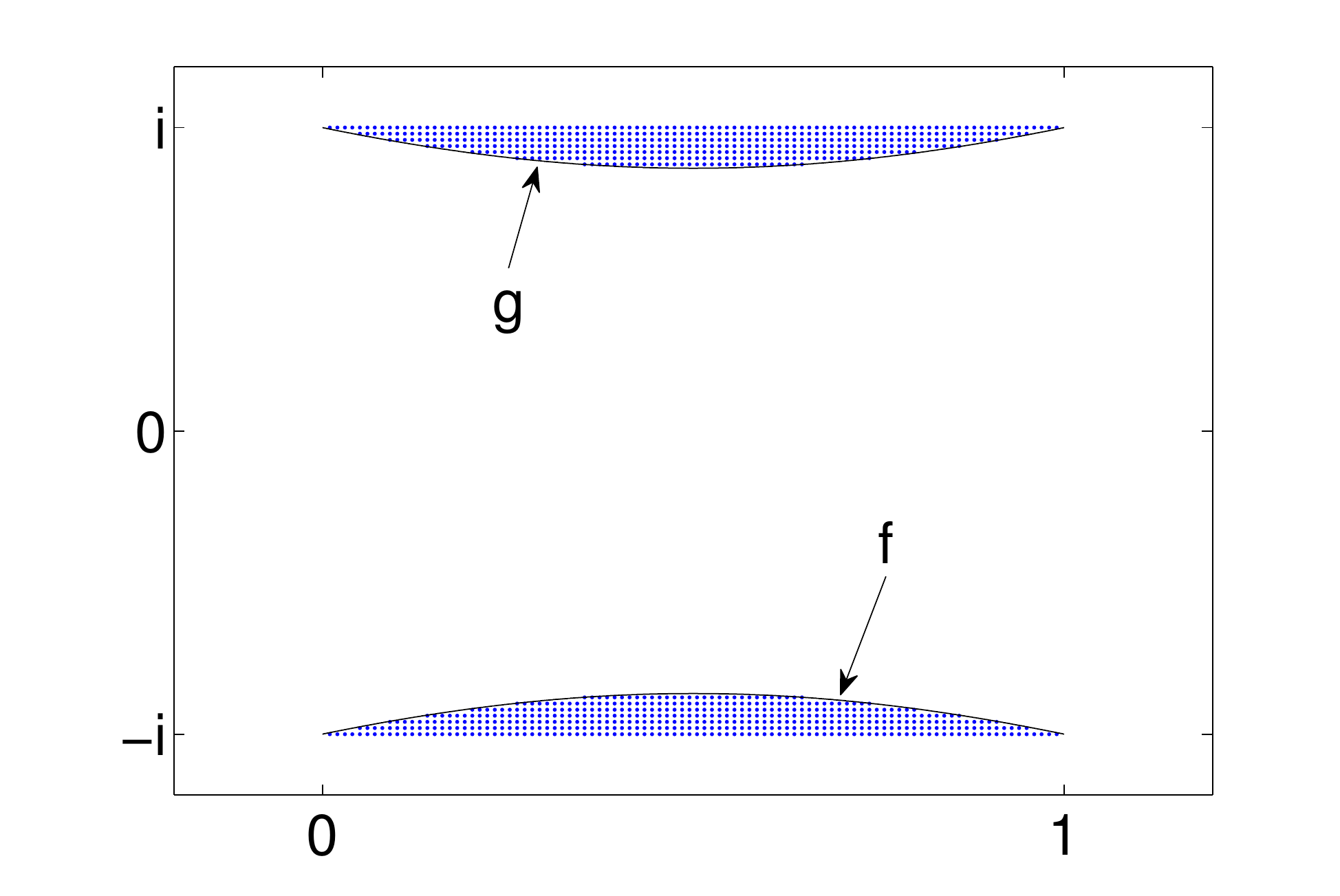}\includegraphics[scale=.3]{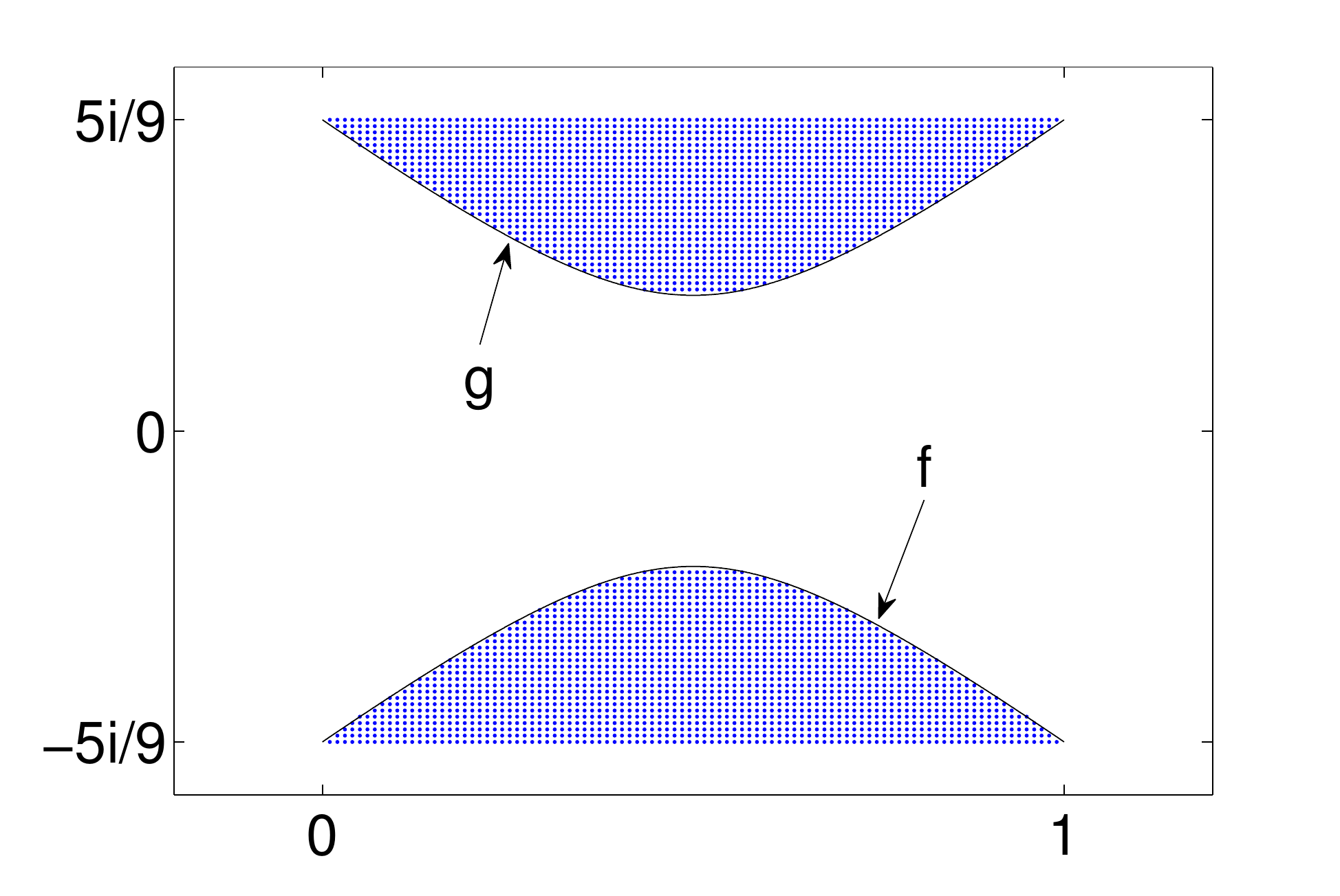}
\includegraphics[scale=.3]{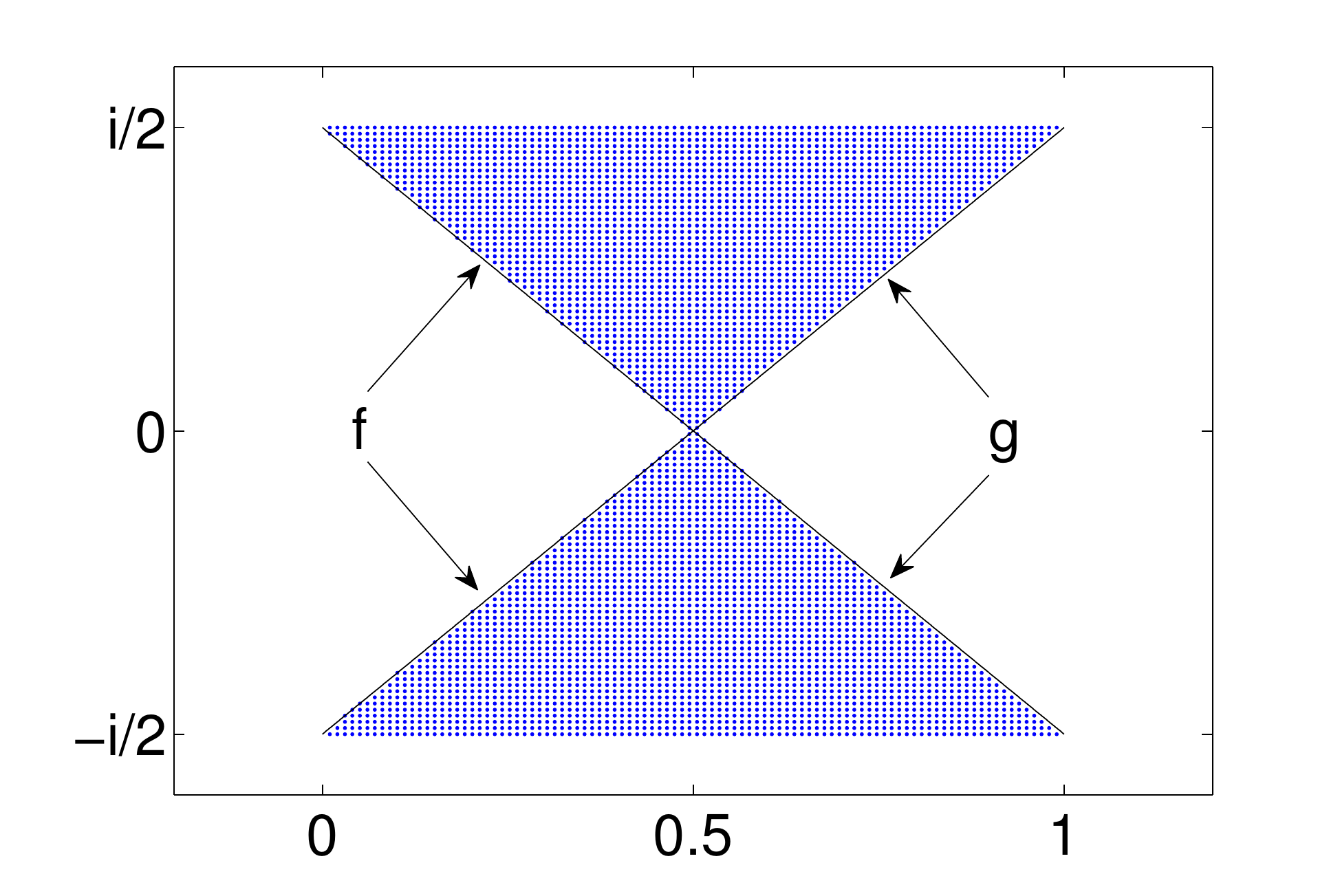}\includegraphics[scale=.3]{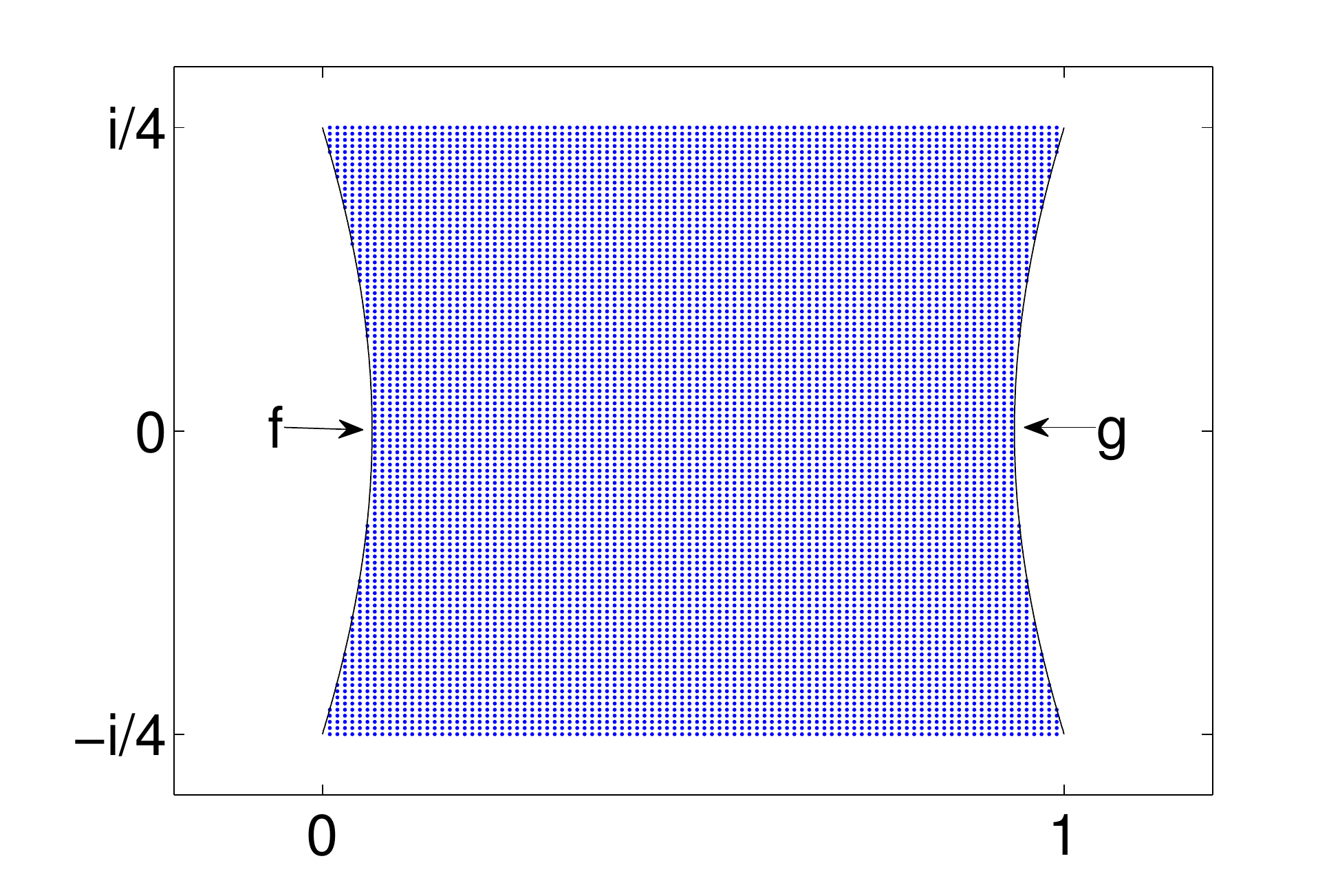}
\caption{The figures show the curves $f$ and $g$ (which together form $\Gamma_{\alpha,\beta}^{\gamma,\delta}$). The shaded regions are $\mathcal{U}_{\alpha,\beta}^{\gamma,\delta}$. Clockwise from top left:  $\alpha=0$, $\beta = 1$, $\gamma=-1$, $\delta=1$; $\alpha=0$, $\beta = 1$, $\gamma=-5/9$, $\delta=5/9$; $\alpha=0$, $\beta = 1$, $\gamma=-1/2$, $\delta=1/2$; $\alpha=0$, $\beta = 1$, $\gamma=-1/4$, $\delta=1/4$.}
\end{figure}
The assertions of the following two lemmata are immediate consequences of the above definition.

\begin{lemma}\label{gamprops}
If $\beta-\alpha\le \delta-\gamma$, then 
\begin{align*}
\gamma\le\Im f(t) &\le \frac{\gamma+\delta}{2} - \sqrt{\left(\frac{\delta-\gamma}{2}\right)^2-\left(\frac{\beta-\alpha}{2}\right)^2}\le\frac{\gamma+\delta}{2}\quad\textrm{and}\\
\delta\ge\Im g(t) &\ge \frac{\gamma+\delta}{2} + \sqrt{\left(\frac{\delta-\gamma}{2}\right)^2-\left(\frac{\beta-\alpha}{2}\right)^2}\ge\frac{\gamma+\delta}{2}\quad\forall t\in[0,1].
\end{align*}
If $\beta-\alpha> \delta-\gamma$, then 
\begin{align*}
\alpha\le\Re f(t) &\le \frac{\alpha+\beta}{2}-\sqrt{\left(\frac{\beta-\alpha}{2}\right)^2-\left(\frac{\delta-\gamma}{2}\right)^2}<\frac{\alpha+\beta}{2}\quad\textrm{and}\\
\beta\ge\Re g(t) &\ge \frac{\alpha+\beta}{2}+\sqrt{\left(\frac{\beta-\alpha}{2}\right)^2-\left(\frac{\delta-\gamma}{2}\right)^2}>\frac{\alpha+\beta}{2}\quad\forall t\in[0,1].
\end{align*}
\end{lemma}

\begin{lemma}\label{olderlem}
If $\alpha<\Re z\le\beta$, $\gamma\le\Im z\le \delta$ and $z\notin\mathcal{U}_{\alpha,\beta}^{\gamma,\delta}\cup\Gamma_{\alpha,\beta}^{\gamma,\delta}$, then
\[
\Re z  -  \frac{(\Im z-\delta)(\Im z-\gamma)}{\Re z- \alpha} >\beta.
\]
If $z\in\Gamma_{\alpha,\beta}^{\gamma,\delta}$ and $\Re z\ne \alpha$, then
\[
\Re z -  \frac{(\Im z-\delta)(\Im z-\gamma)}{\Re z- \alpha} 
= \beta.
\]
If $z\in\mathcal{U}_{\alpha,\beta}^{\gamma,\delta}$, then
\[
\Re z  -  \frac{(\Im z-\delta)(\Im z-\gamma)}{\Re z- \alpha} < \beta.
\]
\end{lemma}

\begin{lemma}\label{oldlem}
If $z\in\mathcal{U}_{\alpha,\beta}^{\gamma,\delta}$, then
\[
\beta - \Re z\ge\frac{\dist(z,\Gamma_{\alpha,\beta}^{\gamma,\delta})^2}{\Re z-\alpha}-  \frac{(\Im z-\delta)(\Im z-\gamma)}{\Re z- \alpha}.
\]
\end{lemma}
\begin{proof} Let $z\in\mathcal{U}_{\alpha,\beta}^{\gamma,\delta}$.
First, consider the case where $\beta-\alpha\le \delta-\gamma$. By Lemma \ref{gamprops}, we have $\Im z\ne (\gamma+\delta)/2$. We assume that $\Im z>(\gamma+\delta)/2$, the case where $\Im z<(\gamma+\delta)/2$ may be treated similarly. For some $r\ge\dist(z,\Gamma_{\alpha,\beta}^{\gamma,\delta})>0$ and $t\in[0,1]$, we have 
\[\Re z + (\Im z-r)i=g(t).
\]
Then, using Lemma \ref{olderlem},
\[
\Re z -  \frac{(\Im z-r-\delta)(\Im z-r-\gamma)}{\Re z- \alpha} 
= \beta
\]
and hence
\begin{align*}
\beta-  \Re z + \frac{(\Im z-\delta)(\Im z-\gamma)}{\Re z- \alpha} &= -\frac{r(\gamma+\delta+ r-2\Im z)}{\Re z-\alpha}>\frac{r^2}{\Re z-\alpha}.
\end{align*}
Now consider the case where $\beta -\alpha>\gamma+\delta$. Let $\Re z\ge(\alpha+\beta)/2$, the case where $\Re z<(\alpha+\beta)/2$ may be treated similarly.  For some $r\ge\dist(z,\Gamma_{\alpha,\beta}^{\gamma,\delta})>0$ we have 
\begin{align*}
\Re z + r -  \frac{(\Im z-\delta)(\Im z-\gamma)}{\Re z+ r- \alpha} =\beta\void{\quad\textrm{and}\quad
\Re z + \frac{\Im z(1-\Im z)}{\Re z - \alpha} <\alpha}
\end{align*}
and hence
\begin{align*}
\beta-  \Re z +  \frac{(\Im z-\delta)(\Im z-\gamma)}{\Re z- \alpha}&= r\left(\frac{2\Re z-\alpha-\beta +r}{\Re z-\alpha}\right)
\ge\frac{r^2}{\Re z-\alpha}.
\end{align*}
\end{proof}

\section{The spectra of $A+iB$}
In this section we will prove enclosure results for $\sigma(A+iB)$. Unless stated otherwise, $A$ is assumed to be a bounded self-adjoint operator with
\[
\min\sigma(A)=:a^-<a^+:=\max\sigma(A).
\]
We shall always assume that $B$ is a bounded self-adjoint operator with 
\[
\min\sigma(B)=:b^-<b^+:=\max\sigma(B).
\]

\begin{lemma}\label{oldestlem}
Let $u\in\mathcal{H}$, then $\Vert Bu\Vert^2 \le (b^-+b^+)\langle Bu,u\rangle -b^-b^+\Vert u\Vert^2$.
\end{lemma}
\begin{proof}
Let $F$ be the spectral measure associated to $B$. For all $\lambda\in[b^-,b^+]$ we have $(\lambda-b^-)(\lambda-b^+)\le 0$, hence
\begin{align*}
\Vert Bu\Vert^2 &=
 \int_{b^-}^{b^+}\lambda^2~d\langle F_\lambda u,u\rangle\le  \int_{b^-}^{b^+}(\lambda b^- + \lambda b^+ - b^-b^+)~d\langle F_\lambda u,u\rangle\\
 &= (b^-+b^+)\langle Bu,u\rangle - b^-b^+\Vert u\Vert^2.
\end{align*}
\end{proof}

We will make use of the following spectral enclosure result which is due to Kato; see  \cite[Lemma 1]{kat}. Let $u\in\Dom(A)$ where $A$ may be unbounded, $\Vert u\Vert=1$, $\langle Au,u\rangle=\eta$ and $\Vert(A-\eta)u\Vert=\zeta$, then
\begin{align}
\xi<\eta\quad\Rightarrow\quad\left(\xi,\eta+\frac{\zeta^2}{\eta-\xi}\right]\cap\sigma(A)\ne\varnothing.\label{kato1}
\end{align} 

\void{\begin{remark}\label{rem}
If follows, from the proof of Lemma \ref{oldestlem}, that if $\sigma(B)=\{c,d\}$, then
\[
\langle Bu,u\rangle=\delta\Vert u\Vert^2\Rightarrow\Vert Bu\Vert^2 = (c\delta+d\delta-cd)\Vert u\Vert^2.
\]
\end{remark}}

\begin{definition}\label{rsK}
For $a,b\in\mathbb{R}$, $a<b$, we set
\begin{align*}
r_{a,b}&=\max\left\{\frac{b-a}{2},\Vert B\Vert\right\},\\
s_{a,b}&=\Big(r_{a,b}^2+2r_{a,b}\big(4+10\Vert B\Vert\big) + 4\Vert B\Vert^2\Big)(b-a),\\
K_{a,b}&=\max\Big\{r_{a,b}^4,2r_{a,b}^3,s_{a,b}\Big\}.
\end{align*}
\end{definition}

\begin{theorem}\label{lem1b}
Let $(a,b)\subset\rho(A)$ where $A$ may be unbounded from above and/or below. Then
\[
\mathcal{U}_{a,b}^{b^-,b^+}\subset\rho(A+iB)
\]
and
\begin{equation}\label{resbd}
\Vert(A+iB-z)^{-1}\Vert\le \frac{K_{a,b}}{\dist(z,\Gamma_{a,b}^{b^-,b^+})^4}\quad\forall z\in\mathcal{U}_{a,b}^{b^-,b^+}.
\end{equation}
\end{theorem}
\begin{proof}
Let $z\in\mathcal{U}_{a,b}^{b^-,b^+}$. First we note that $A+iB-z$ is a closed operator. Let $u\in\Dom(A)$ with $\Vert u\Vert = 1$ and $\Vert(A+iB-z)u\Vert = \varepsilon$. Assume that
\begin{equation}\label{assumption}
\varepsilon<\min\left\{1,\frac{\dist(z,\Gamma_{a,b}^{b^-,b^+})}{2}\right\}.
\end{equation}
Let us show that \eqref{assumption} implies that $a<\Re z -\varepsilon$. We consider the case where $b-a>b^+-b^-$ (the case where $b-a\le b^+-b^-$ may be treated similarly). We have
\[
\Im z = (1-t)b^- + tb^+\quad\text{for some }t\in[0,1].
\]
Then
\[
\Re f\left(\frac{\Im z-b^-}{b^+-b^-}\right) + i\Im z =f\left(\frac{\Im z-b^-}{b^+-b^-}\right)\in\Gamma_{a,b}^{b^-,b^+}
\]
and
\begin{align*}
\varepsilon<\dist(z,\Gamma_{a,b}^{b^-,b^+})&\le\left\vert z - f\left(\frac{\Im z-b^-}{b^+-b^-}\right)\right\vert\\
&=\Re z - \Re f\left(\frac{\Im z-b^-}{b^+-b^-}\right)\\
&\le \Re z -a.
\end{align*}
For some $v\in\mathcal{H}$ with $\Vert v\Vert=1$ we have $(A+iB-z)u=\varepsilon v$, therefore
\[
\langle(A-\Re z)u,u\rangle + i\langle(B-\Im z)u,u\rangle = \varepsilon\langle v,u\rangle
\]
and 
\begin{align}
a<\Re z -\varepsilon\le\langle Au,u\rangle &\le \Re z+\varepsilon,\label{thlab0}\\
\langle Bu,u\rangle &= \Im z+\varepsilon\Im \langle v,u\rangle,\label{thlab1}\\
\Vert(A-\Re z)u\Vert&\le\varepsilon + \Vert(B-\Im z)u\Vert.\label{fixed}
\end{align}
Using Lemma \ref{oldestlem} and \eqref{thlab1}, we obtain
\begin{align*}
\Vert(B-\Im z)u\Vert^2 &= \Vert Bu\Vert^2 - 2\Im z\langle Bu,u\rangle + (\Im z)^2\\
&\le \langle Bu,u\rangle(b^-+b^+)  - b^-b^+- 2\Im z\langle Bu,u\rangle +  (\Im z)^2\\
&= -(\Im z-b^+)(\Im z -b^-) + \varepsilon\Im\langle v,u\rangle(b^++b^- - 2\Im z).
\end{align*}
Now applying \eqref{kato1} with 
\[
\xi=a,\quad\eta=\langle Au,u\rangle\quad\text{and}\quad\zeta=\Vert(A-\eta)u\Vert\le\Vert(A-\Re z)u\Vert + \vert\Re z-\eta\vert,
\]
we obtain
\[
\left(a,\langle Au,u\rangle + \frac{(\Vert(A-\Re z)u\Vert + \vert\Re z-\eta\vert)^2}{\langle Au,u\rangle-a}\right]\cap\sigma(A)\ne\varnothing.
\]
Using \eqref{thlab0} and \eqref{fixed}, we have 
\[
\left(a,\Re z + \varepsilon + \frac{\big(2\varepsilon + \Vert(B-\Im z)u\Vert\big)^2}{\Re z-\varepsilon- a}\right]\cap\sigma(A)\ne\varnothing.
\]
Then, using $(a,b)\subset\rho(A)$ and $\vert\Im z\vert\le\Vert B\Vert$, and the assumption that $\varepsilon<1$,
\begin{align*}
b-\Re z&\le\varepsilon + \frac{4\varepsilon^2 +4\varepsilon\Vert(B-\Im z)u\Vert+ \Vert(B-\Im z)u\Vert^2}{\Re z-\varepsilon- a}\\
&\le\varepsilon + \frac{4\varepsilon +8\Vert B\Vert\varepsilon}{\Re z- a-\varepsilon} - \frac{(\Im z-b^+)(\Im z -b^-)}{\Re z-\varepsilon- a}\\
&\quad + \frac{\varepsilon\vert b^++b^- - 2\Im z\vert}{\Re z- a-\varepsilon}. 
\end{align*}
Combining this estimate with Lemma \ref{oldlem}
\begin{align*}
\frac{\dist(z,\Gamma_{a,b}^{b^-,b^+})^2}{\Re z-a}&\le
\varepsilon + \frac{4\varepsilon +8\Vert B\Vert\varepsilon}{\Re z- a-\varepsilon} - \frac{\varepsilon(\Im z-b^+)(\Im z -b^-)}{(\Re z - a)(\Re z- a-\varepsilon)}\\
&\quad+ \frac{\varepsilon(b^+-b^-)}{\Re z- a-\varepsilon}. 
\end{align*}
From \eqref{assumption} we deduce that
\[
\Re z- a-\varepsilon\ge \dist(z,\Gamma_{a,b}^{b^-,b^+})  - \varepsilon >\frac{\dist(z,\Gamma_{a,b}^{b^-,b^+})}{2}
\]
and hence
\begin{align*}
\frac{\dist(z,\Gamma_{a,b}^{b^-,b^+})^2}{\Re z-a}\le \left(1+2\frac{4+10\Vert B\Vert}{\dist(z,\Gamma_{a,b}^{b^-,b^+})} + \frac{4\Vert B\Vert^2}{\dist(z,\Gamma_{a,b}^{b^-,b^+})^2}\right)\varepsilon.
\end{align*}
Then
\begin{align*}
\dist(z,\Gamma_{a,b}^{b^-,b^+})^4\le \Big(r_{a,b}^2+2r_{a,b}\big(4+10\Vert B\Vert\big) + 4\Vert B\Vert^2\Big)(b-a)\varepsilon
\end{align*}
and therefore
\begin{equation}\label{ep1}
\dist(z,\Gamma_{a,b}^{b^-,b^+})^4\big/s_{a,b} \le \varepsilon.
\end{equation}
It follows from \eqref{ep1} and assumption \eqref{assumption}, that $\text{nul}(A+iB-z)=0$ and that $A+iB-z$ has closed range.
Similarly, $\text{nul}(A-iB-\overline{z})=0$ and $A-iB-\overline{z}$ has closed range. Since 
$\text{def}(A+iB-z)=\text{nul}(A-iB-\overline{z})$ we deduce that $z\in\rho(A+iB-z)$. Furthermore, combining \eqref{ep1} with assumption \eqref{assumption}, we obtain
\begin{align*}
\Vert(A+iB-z)^{-1}\Vert
&\le \frac{\max\left\{s_{a,b},\dist(z,\Gamma_{a,b}^{b^-,b^+})^4,2\dist(z,\Gamma_{a,b}^{b^-,b^+})^3\right\}}{\dist(z,\Gamma_{a,b}^{b^-,b^+})^4}\\
&\le \frac{K_{a,b}}{\dist(z,\Gamma_{a,b}^{b^-,b^+})^4}.
\end{align*}
\end{proof}

\begin{remark}\label{dimensions}Suppose that $a_1<b_1\le a_2< b_2$,
\[
(a_1,b_1)\subset\rho(A),\quad(a_2,b_2)\subset\rho(A)\quad\text{and}\quad\min\{b_1-a_1,b_2-a_2\}>b^+-b^-.
\] 
Let $f_1, g_1$ be as in Definition \ref{gamdef} with $\alpha=a_1$, $\beta=b_1$, $\gamma=\beta^-$ and $\delta=\beta^+$. Let $f_2,g_2$  be as in Definition \ref{gamdef} with $\alpha=a_2$, $\beta=b_2$, $\gamma=\beta^-$ and $\delta=\beta^+$. 
The curves $g_1$ and $f_2$ enclose a region (see Figure 2). It follows, from Theorem \ref{lem1b}, that the dimension of the spectral subspace associated to $\sigma(A+iB)$ and this enclosed region is the same as the dimension of the spectral subspace associated to $\sigma(A)$ and the interval $[b_1,a_2]$.
\begin{figure}[h!]
\centering
\includegraphics[scale=.3]{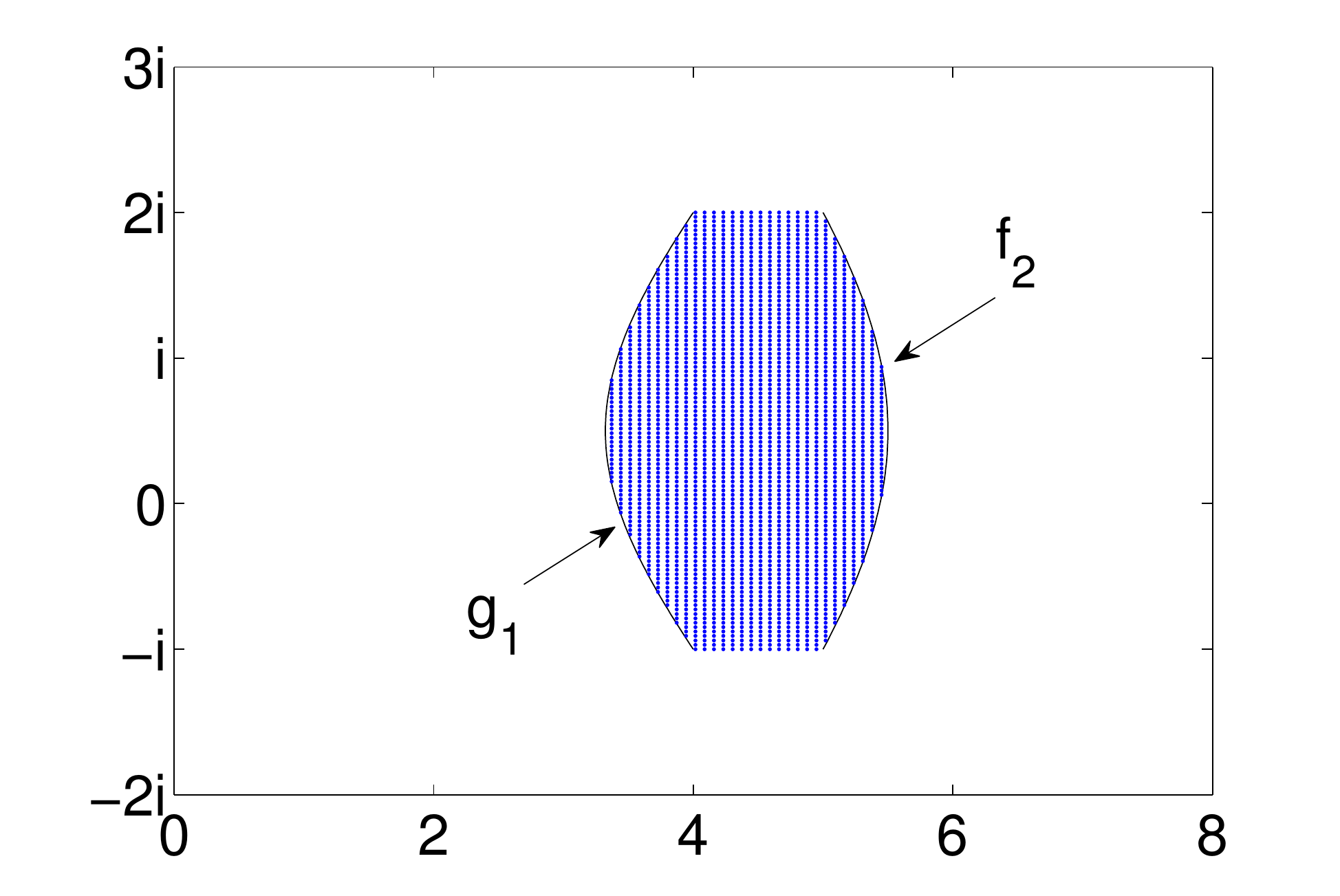}\includegraphics[scale=.3]{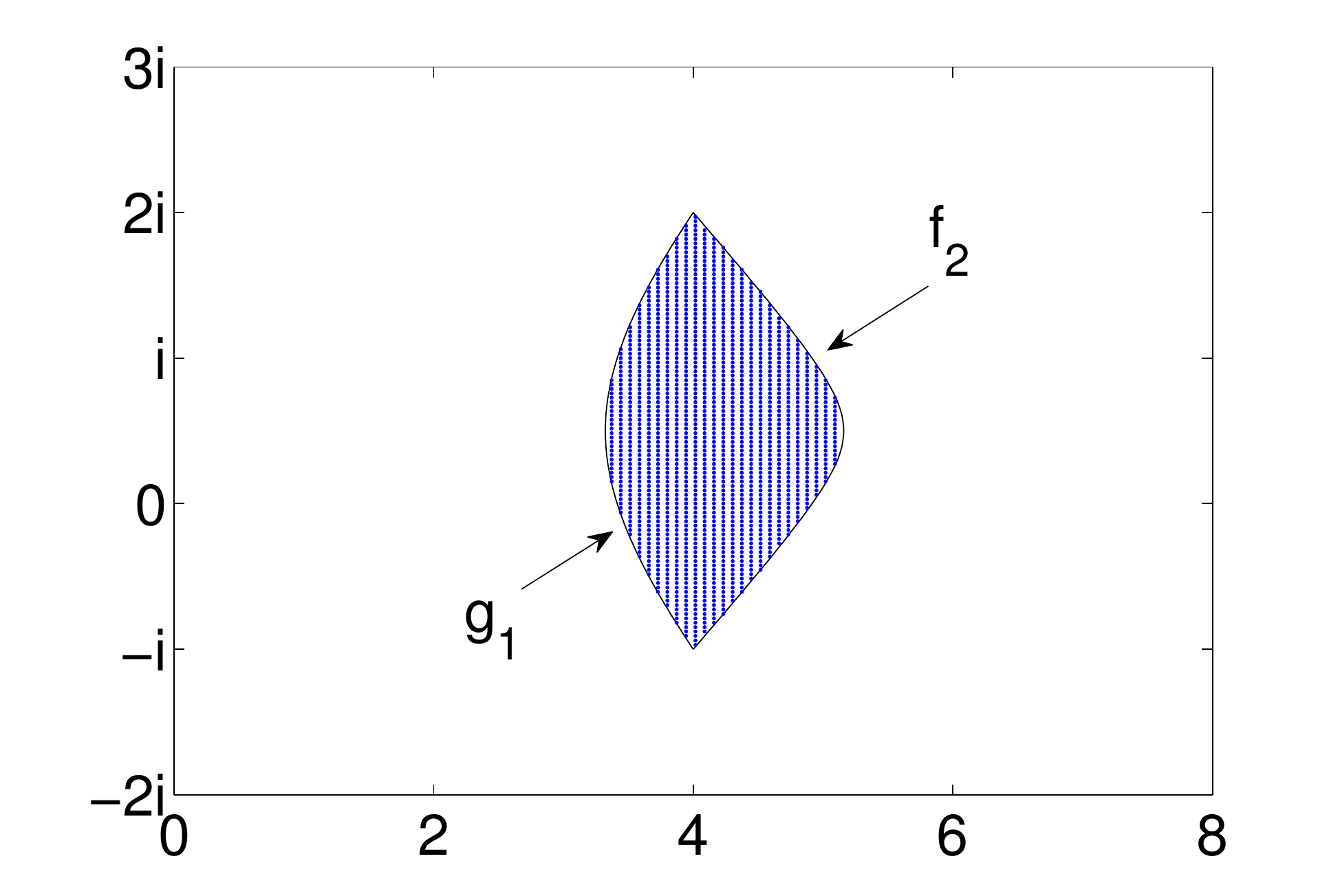}
\caption{The figures show $g_1$, $f_2$ and the shaded region they enclose; on the left $a_1=0$, $b_1 = 4$, $a_2=5$, $b_2=10$ $b^-=-1$ and $b^+=2$; on the right $a_1=0$, $b_1 = 4$, $a_2=4$, $b_2=7.1$, $b^-=-1$ and $b^+=2$.}
\end{figure}
\end{remark}

\begin{definition}\label{W}
We denote the numerical range by $W(\cdot)$, and define 
\begin{align*}
\mathcal{U}_A&:= \bigcup_{\efrac{a,b\in W(A)}{(a,b)\subset\rho(A)}}\mathcal{U}_{a,b}^{b^-,b^+},\\
\mathcal{X}_A&:=\big\{z\in\mathbb{C}:~\Re z\in \overline{W(A)}\hspace{5pt}\textrm{and}\hspace{5pt}b^-\le\Im z\le b^+\big\}\Big\backslash\mathcal{U}_A,\\
\mathcal{V}_B&:= \big\{z\in\mathbb{C}:\Im z + i\Re z\in\hat{\mathcal{V}}_B\big\}\quad\text{where}\quad\hat{\mathcal{V}}_B:=  \bigcup_{\efrac{a,b\in W(B)}{(a,b)\subset\rho(B)}}\mathcal{U}_{a,b}^{a^-,a^+},\\
\mathcal{Y}_B&:=\big\{z\in\mathbb{C}:\Im z\in\overline{W(B)}\hspace{5pt}\textrm{and}\hspace{5pt}a^-\le\Re z\le a^+\big\}\Big\backslash\mathcal{V}_B.
\end{align*}
\end{definition}

\begin{theorem}\label{cor1a}
If $A$ is unbounded from above and/or below, then
\begin{equation}\label{unbinc}
\sigma(A+iB)\subset\mathcal{X}_A.
\end{equation}
If $A$ is bounded, then
\begin{equation}\label{binc}
\sigma(A+iB)\subset\mathcal{X}_A\cap\mathcal{Y}_B.
\end{equation}
\end{theorem}
\begin{proof}
The first assertion follows immediately from Theorem \ref{lem1b}. Suppose that $A$ is bounded. Then $\sigma(A+iB)\subset\mathcal{X}_A$ again follows from Theorem \ref{lem1b}. Let $w\notin\mathcal{Y}_B$. Then either 
\[
w\notin\big\{z\in\mathbb{C}:\Im z\in\overline{W(B)}\hspace{5pt}\textrm{and}\hspace{5pt}a^-\le\Re z\le a^+\big\}\quad\text{or}\quad
w\in\mathcal{V}_B.
\]
Suppose the former is true, then
\begin{align*}
\Im w + i\Re w\in\rho(B+iA)\quad
&\Rightarrow\quad\Im w - i\Re w\in\rho(B-iA)\\
&\Rightarrow\quad\Re w + i\Im w\in\rho(A+iB).
\end{align*}
Suppose instead that $w\in\mathcal{V}_B$, then for some $a,b\in W(B)$, $(a,b)\subset\rho(B)$, we have
\begin{align*}
\Im w + i\Re w\in\mathcal{U}_{a,b}^{a^-,a^+}\quad&\Rightarrow\quad\Im w + i\Re w\in\rho(B+iA)\\
&\Rightarrow\quad\Re w + i\Im w\in\rho(A+iB).
\end{align*}
\end{proof} 

\begin{remark}
Any bounded linear operator, $T\in\mathcal{B}(\mathcal{H})$, may be expressed as
\[
T = \underbrace{\left(\frac{T+T^*}{2}\right)}_A +~ i\underbrace{\left(\frac{T-T^*}{2i}\right)}_B
\]
where $A$ and $B$ are bounded self-adjoint operators. Hence, Theorem \ref{cor1a} provides an enclosure for the spectrum, in terms of the real and imaginary parts, of any bounded linear operator.
\end{remark}

\void{
\begin{definition}
We denote by $\mathcal{Q}$ the set of all self-adjoint operators, $Q$,  with $\sigma(Q)\subset\sigma(A)$, by $\mathcal{S}$ the set of all bounded self-adjoint operators with $cI\le S\le d$, and by $\mathcal{T}$ the set of all self-adjoint operators, $T$, with $\sigma(T)\subset\sigma(B)$.
\end{definition}

The following Corollary is a straightforward consequence of Theorem \ref{cor1} and Definition \ref{W}.

\begin{corollary}
Let $A$ be possibly unbounded, then
\[\bigcup_{Q\in\mathcal{Q}}\bigcup_{S\in\mathcal{S}}\sigma(Q+iS)\subset\mathcal{X}_A.\]
If $A$ is bounded, then
\[\bigcup_{Q\in\mathcal{Q}}\bigcup_{T\in\mathcal{T}}\sigma(Q+iT)\subset\mathcal{X}_A\cap\mathcal{Z}_B.\]
\end{corollary}
}

\begin{corollary}\label{2eigs}
Let $\sigma(A)=\{a^-,a^+\}$ and $\sigma(B)=\{b^-,b^+\}$, then $\sigma(A+iB)\subset\Gamma_{a^-,a^+}^{b^-,b^+}$. For any $z\in\Gamma_{a^-,a^+}^{b^-,b^+}$, $B$ may be chosen such that $z\in\sigma(A+iB)$.
\end{corollary}
\begin{proof}
The first assertion follows immediately from theorems \ref{lem1b} and \ref{cor1a}. Let $u,v$ be normalised eigenvectors with $Au=a^-u$ and $Av=a^+v$. For $t\in[0,1]$ we define the family of self-adjoint operators
\begin{align*}
B(t)x &= b^-\langle x,\sqrt{1-t}u + \sqrt{t}v\rangle(\sqrt{1-t}u + \sqrt{t}v)\\
&\quad + b^+\langle x,\sqrt{t}u-\sqrt{1-t}v\rangle(\sqrt{t}u-\sqrt{1-t}v).
\end{align*}
Evidently, 
\[
\min\sigma(B(t))=b^-\quad\text{and}\quad\max\sigma(B(t))=b^+\quad\forall t\in[0,1].
\]
Furthermore,
\begin{align*}
Au+iB(0)u &= (a^-+ ib^-)u,\quad Av+iB(0)v= (a^+ + ib^+)v,\\
Au+iB(1)u &= (a^-+ ib^+)u,\quad Av+iB(1)v = (a^++ ib^-)v.
\end{align*}
Hence, for each $z\in\Gamma_{a^-,a^+}^{b^-,b^+}$ there exists a $t\in[0,1]$ for which $z\in\sigma(A+iB(t))$.
\end{proof}

\void{
\begin{corollary}\label{rem2}
Let $A$ and $B$ be a $2\times 2$ matrices.
Then $\sigma(A+iB)$ consists of one eigenvalue only if $b-a=d-c$. If $A+iB$ has only one eigenvalue then it is equal to $(a+b)/2 + i(c+d)/2$. If $b-a\ne d-c$, then $A+iB$ has and one eigenvalue 
on the curve $f_{a,b}^{c,d}$ and another on the curve $g_{a,b}^{c,d}$. 
\end{corollary}
\begin{proof}
For the first and third assertions we assume that $b-a>d-c$, the case where $d-c>b-a$ may be treated similarly. $T(s)=A+isB$.
Then $T(0)=A$ and $T(1)=M$. Hence, $\sigma(T(0))=\{a,b\}$. Furthermore, $b-a>sd-sc$ for each $s\in[0,1]$ and therefore the eigenvalues of $T(s)$, which by Corollary \ref{2eigs} must lie on the corresponding curves $f_{a,b}^{sc,sd}$ and $g_{a,b}^{sc,sd}$, cannot coincide for any $s\in[0,1]$. The first and third assertions follow. For the second assertion, suppose that $M$ has a double eigenvalue $\lambda\ne(a+b)/2 + i(c+d)/2$. Without loss of generality we assume that $\lambda$ lies  on $f_{a,b}^{sc,sd}$ and therefore not on $g_{a,b}^{sc,sd}$. For any sufficiently small $\varepsilon>0$,
\[\sigma(T(1+\varepsilon))\subset f_{a,b}^{sc,sd}\quad\text{where}s=1+\varepsilon\quad\text{and}\quad b-a<sd-sc,\]
which is a contradiction.
\end{proof}

\begin{example}\label{ex1}
Let
\[
A = \frac{1}{2}\left(
\begin{array}{cc}
2 & 1\\
1 & 2
\end{array}
\right)
\quad\text{and}\quad
B = \frac{s}{2i}\left(
\begin{array}{cc}
0 & 1\\
-1 & 0
\end{array}
\right)\quad\text{where}\quad s\in\mathbb{R}.
\]
Then $a=1/2$, $b=3/2$, $c=-\vert s\vert/2$ and $d=\vert s\vert/2$. Let $\sigma(A+iB)=\{\lambda_1,\lambda_2\}$, then $b-a = d-c$ only when $s=-1$ or $s=1$. It follows, from Corollary \ref{rem2}, that $\lambda_1=\lambda_2$ is permitted only if $s=-1$ or $s=1$. When $s=-1$
\[
A+iB = \left(
\begin{array}{cc}
1 & 0\\
1 & 1
\end{array}
\right)\quad\text{and}\quad\sigma(A+iB)=\{1\}
\]
When $s=1$
\[
A+iB = \left(
\begin{array}{cc}
1 & 1\\
0 & 1
\end{array}
\right) \quad\text{and}\quad\sigma(A+iB)=\{1\}.
\]
\end{example}

\begin{definition}\label{WX}
We define 
\begin{align*}
\Gamma_A&:= \left(\bigcup_{\efrac{(\alpha,\beta)\in W(A)}{\alpha,\beta\subset\rho(A)}}\Gamma_{\alpha,\beta}^{c,d}\right)\bigcup \Big(\sigma(A)-ci\Big)\bigcup \Big(\sigma(A)+di\Big),\\
\Gamma_B&:=\big\{z\in\mathbb{C}:\Im z + i\Re z\in\hat{\Gamma}_B\big\}\quad\text{where}\\
\hat{\Gamma}_B&:=\left(\bigcup_{\efrac{(\alpha,\beta)\in W(B)}{\alpha,\beta\subset\rho(B)}}\Gamma_{\alpha,\beta}^{a,b}\right)\bigcup \Big(\sigma(B)-ai\Big)\bigcup \Big(\sigma(B)+bi\Big).
\end{align*}
\end{definition}

Evidently, $\Gamma_A$ and $\Gamma_B$ are the boundaries of $\mathcal{X}_A$ and $\mathcal{Z}_B$, respectively. 

\begin{corollary}\label{cor2}
Let $A$ be unbounded and $z\in\Gamma_A$. There exist $Q\in\mathcal{Q}$ and $S\in\mathcal{S}$ such that $z\in\sigma(Q+iS)$.
Let $A$ be bounded and $z\in\Gamma_A\cap\Gamma_B$. There exist $Q\in\mathcal{Q}$ and $T\in\mathcal{T}$ with $z\in\sigma(S+iT)$.
\end{corollary}

\begin{lemma}\label{aux}
Let $\sigma(A)\subset[\alpha,\beta]$ and $\vert z - (\alpha+\beta)/2\vert>(\beta-\alpha)/2$. There exists a constant $c>0$ such that $\vert\langle(A-z)^2u,u\rangle\vert\ge c\Vert u\Vert^2$ for all $u\in\mathcal{H}$.
\end{lemma}
\begin{proof}
We argue similarly to the proof of \cite[Theorem 3.1]{shar}. If $z\in\mathbb{R}$, then we may take $c = \min\{\vert c-\alpha\vert,\vert c-\beta\vert\}$. If $\Im z\ne 0$, then, since $z$ lies outside the disc with centre $(\alpha+\beta)/2$ and radius $(\beta-\alpha)/2$, the closed and bounded set
\[
\mathcal{S}:=\{(\lambda-z)^2:\lambda\in\sigma(A)\}
\]
lies in an open half plane. Hence there exist a $c>0$ and a $\theta\in[0,\pi)$ such that
\[
c\le\min\{\Re e^{i\theta}w:w\in\mathcal{S}\}.
\]
It follows that for some we have
\begin{align*}
\vert\langle(A-z)^2u,u\rangle\vert &= \int_{\sigma(A)}e^{i\theta}(\lambda-z)^2~d\langle E_\lambda u,u\rangle\ge c\Vert u\Vert^2.
\end{align*}
\end{proof}

\begin{theorem}\label{thm2}
Let $\sigma(A)\subset[\alpha,\beta]$ and $\sigma(B)=\{\nu,\mu\}$, then
$\sigma(A+iB)\subset\mathcal{U}_{\alpha,\beta}^{\nu,\mu}\cup\Gamma_{\alpha,\beta}^{\nu,\mu}$.
\end{theorem}
\begin{proof}
Let $z\notin\mathcal{U}_{\alpha,\beta}^{\nu,\mu}\cup\Gamma_{\alpha,\beta}^{\nu,\mu}$ with $\Re z < \alpha\le\beta$ and $\nu\le\Im z\le\mu$. By Lemma \ref{olderlem}
\begin{equation}\label{fin1}
\Re z  -  \frac{(\Im z-\mu)(\Im z-\nu)}{\Re z- \alpha} > \beta.
\end{equation}
Let $\Vert(A+iB-z)u_n\Vert=\varepsilon_n\to 0$ where $(u_n)_{n\in\mathbb{N}}$ is a sequence of normalised vectors. There exists a sequence of normalised vectors  $(v_n)$ with  $(A+iB-z)u_n=\varepsilon_n v_n$. Similarly to the proof Theorem \ref{lem1b}, we obtain
\begin{align*}
&\langle(A-\Re z)u_n,u_n\rangle = \varepsilon_n\Re \langle v_n,u_n\rangle,\quad
 \langle(B-\Im z)u,u\rangle = \varepsilon\Im \langle v,u\rangle,\\
\Vert&(A-\Re z)u_n\Vert^2 = \varepsilon_n^2 + 2\varepsilon_n\Im\langle(B-\Im z)u_n,v_n\rangle + \Vert(B-\Im z)u_n\Vert^2,
\end{align*}
and, in view of Remark \ref{rem},
\begin{align*}
\Vert(B-\Im z)u_n\Vert^2 &= -(\Im z-\mu)(\Im z -\nu) + \varepsilon_n\Im\langle v_n,u_n\rangle(\mu+\nu - 2\Im z).
\end{align*}
Set $w=\sqrt{(\mu-\Im z)(\Im z -\nu)}$. Then
\begin{align*}
\langle(A-\Re z- iw)^2u_n,u_n\rangle &= \Vert(A-\Re z)u_n\Vert^2 - 2iw\langle(A-\Re z)u_n,u_n\rangle - w^2 \to 0.
\end{align*}
We deduce, from Lemma \ref{aux}, that $\vert \Re z+iw - (\alpha+\beta)/2\vert\le(\beta-\alpha)/2$ and hence
\begin{equation}\label{fin2}
\Re z - \frac{(\Im z-\mu)(\Im z -\nu)}{\Re z- \alpha}  = \Re z + \frac{w^2}{\Re z- \alpha} \le\beta
\end{equation}
since the left hand side is the real number $>\alpha$ that intersects the circle which passes through $\alpha$, $\Re z\pm iw$ and whose centre is real. Since \eqref{fin2} contradicts \eqref{fin1}, we deduce that $z\in\rho(A+iB)$. The case where $\Re z = \alpha$ may be proved by applying the above argument to $-A$.
\end{proof}
}
\begin{example}
Let $\sigma(A)=\{-1,0,2\}$ and $\sigma(B)=\{-s,0,s\}$ where $s\in\mathbb{R}$. By Theorem \ref{cor1a} we have $\sigma(A+iB)\subset\mathcal{X}_A\cap\mathcal{Y}_B$. For varying values of $s\in\mathbb{R}$, Figures 3--5 show the region(s) enclosed by $\mathcal{X}_A\cap\mathcal{Y}_B$. Also shown is $\sigma(A+iB)$ for 1000 randomly generated $3\times 3$ matrices $A$ and $B$ where $\sigma(A)=\{-1,0,2\}$ and $\sigma(B)=\{-s,0,s\}$.
\begin{figure}[h!]
\centering
\includegraphics[scale=.3]{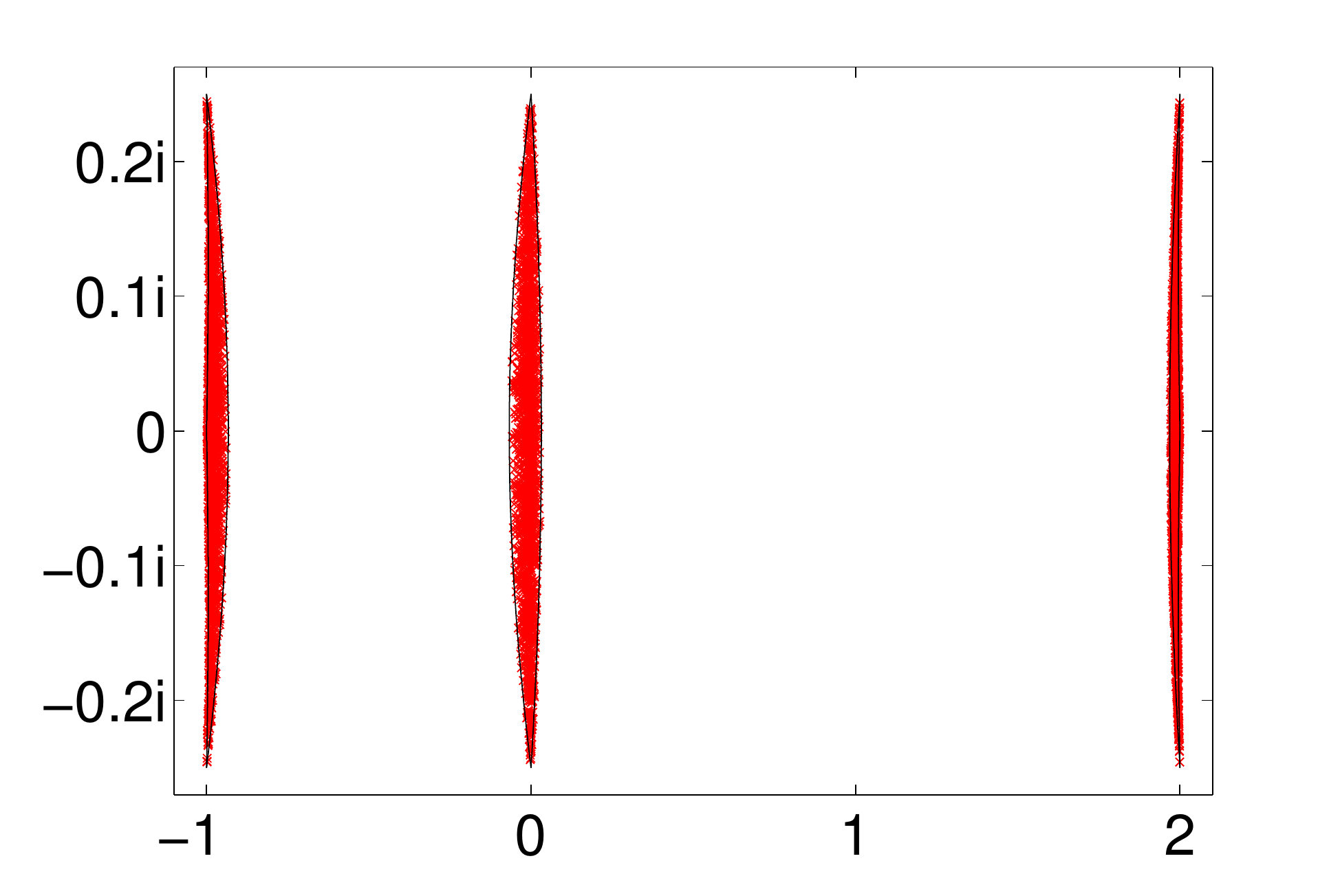}\includegraphics[scale=.3]{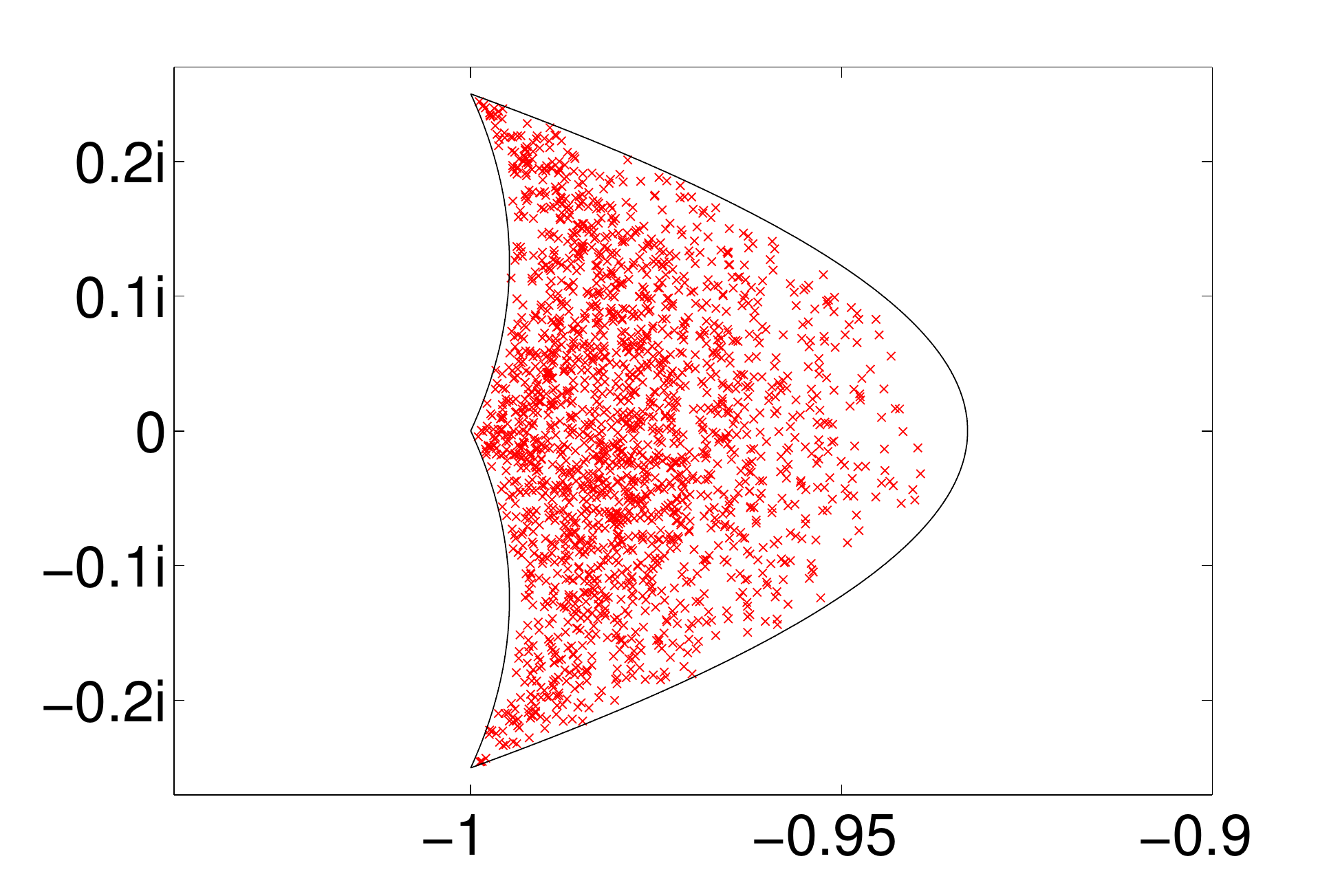}
\includegraphics[scale=.3]{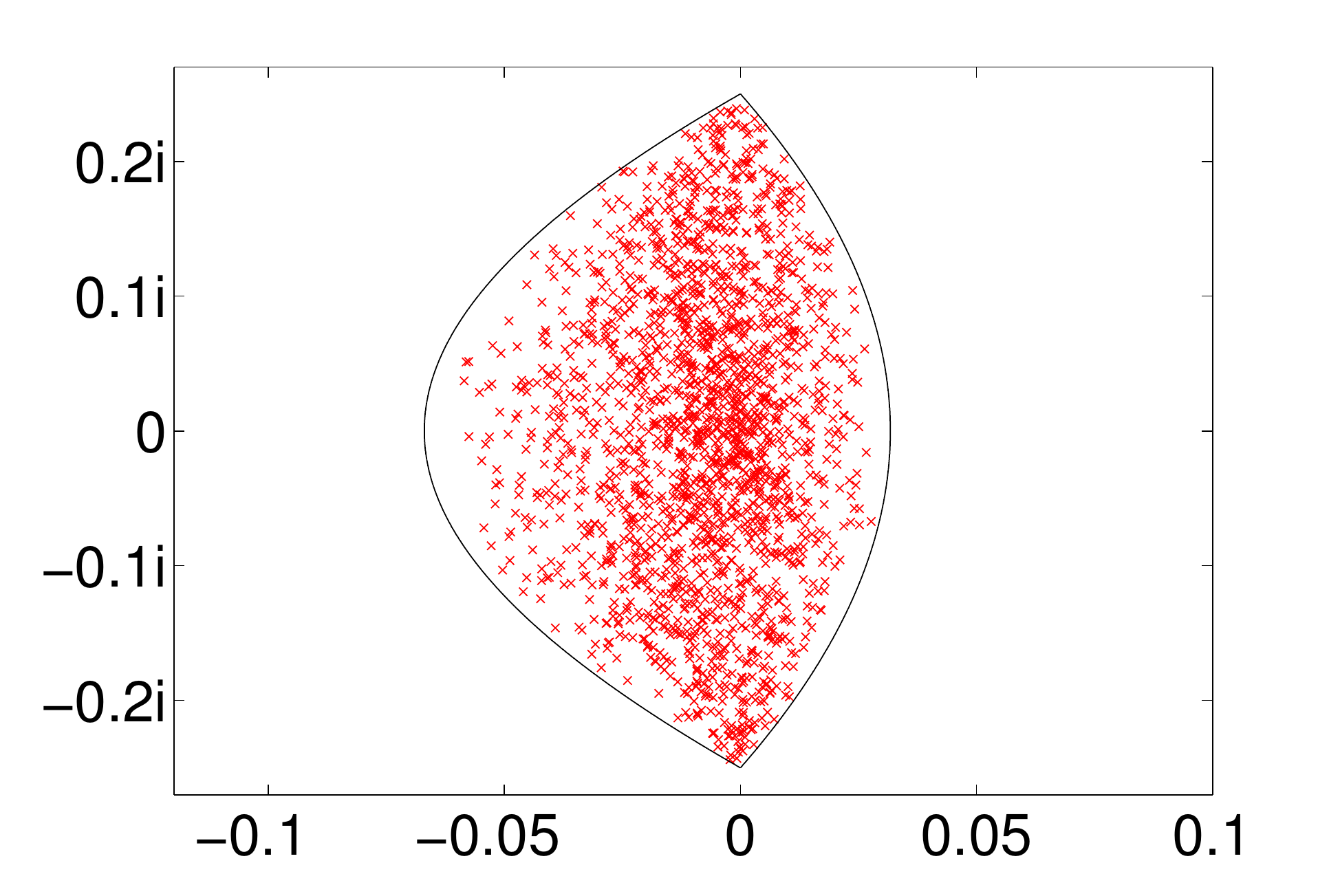}\includegraphics[scale=.3]{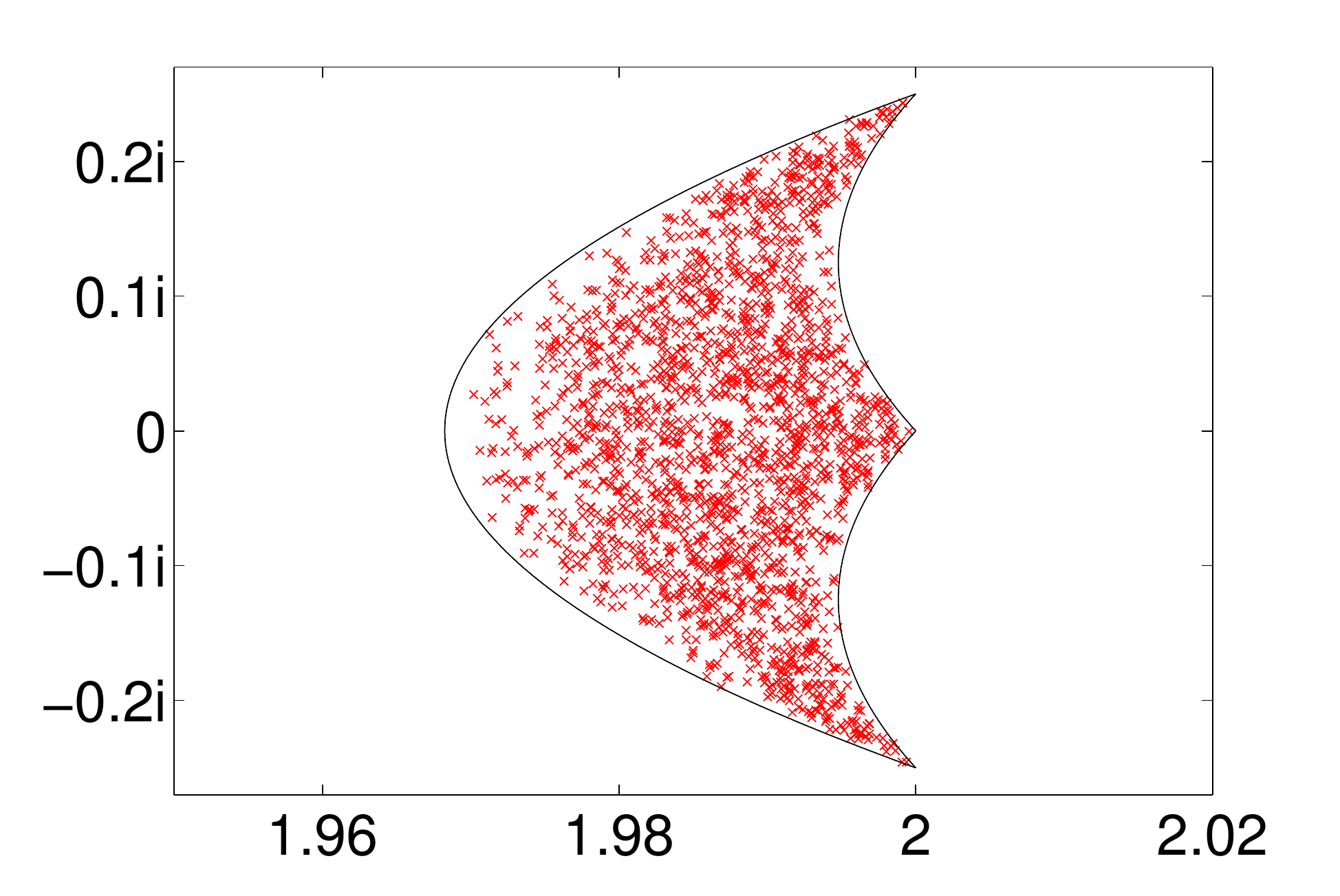}
\caption{With $s=0.25$, $\mathcal{X}_A\cap\mathcal{Y}_B$ consists of 3 disjoint regions which are shown top left. Also shown are the three regions in more detail. The red dots are $\sigma(A+iB)$ for 1000 randomly generated $3\times 3$ matrices $A$ and $B$ where $\sigma(A)=\{-1,0,2\}$ and $\sigma(B)=\{-s,0,s\}$.
}
\end{figure}
\begin{figure}[h!]
\centering
\includegraphics[scale=.3]{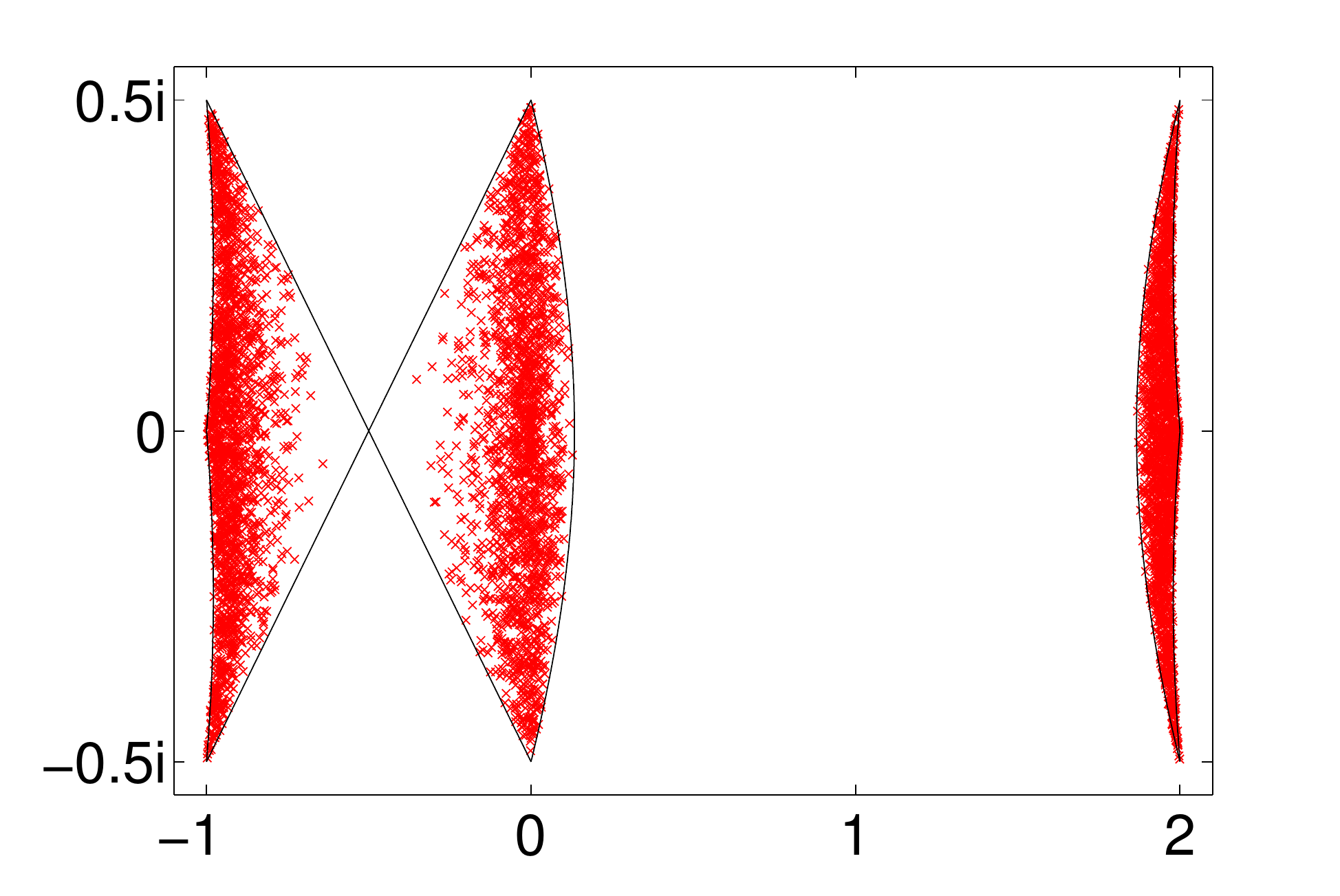}\includegraphics[scale=.3]{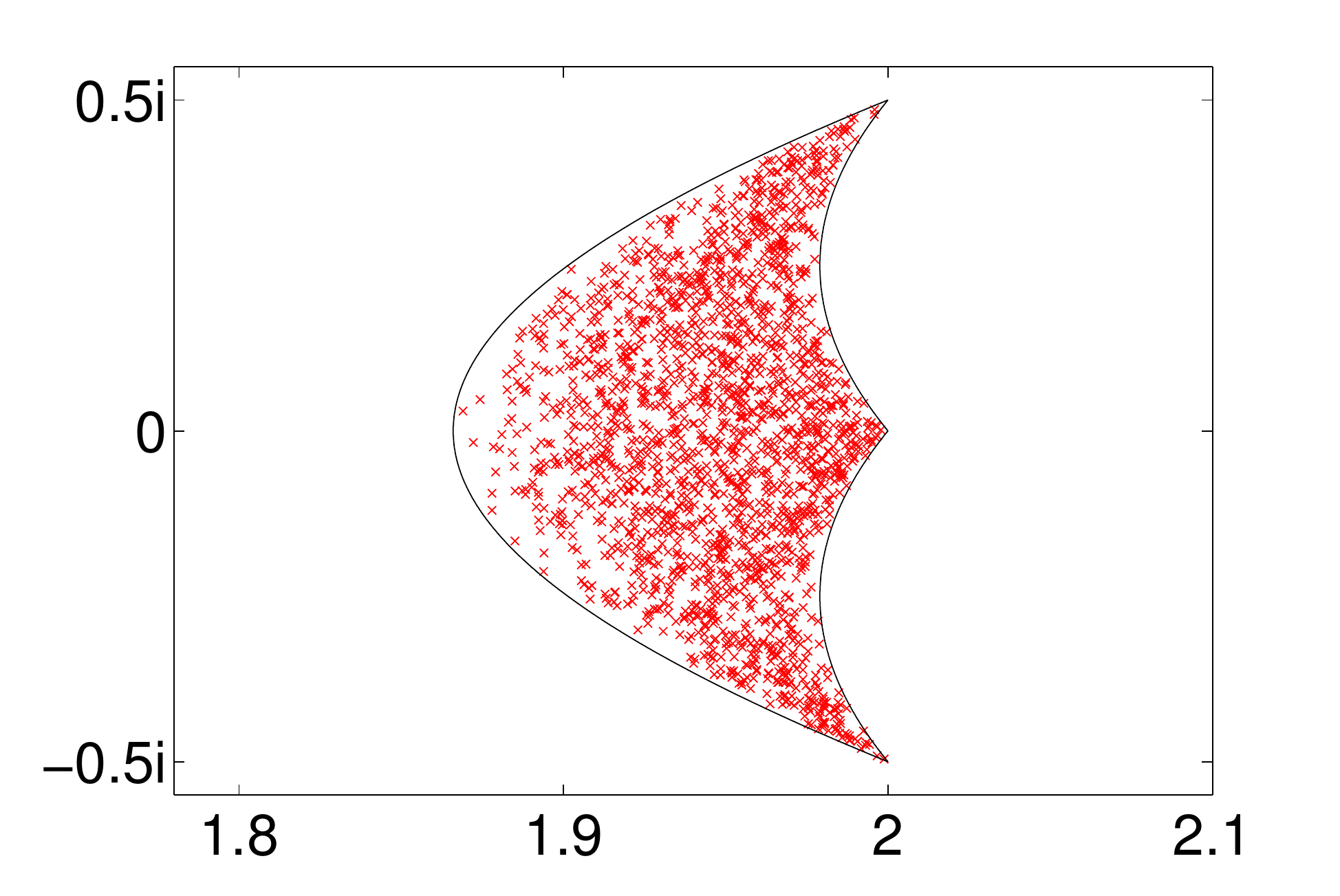}
\caption{With $s=0.5$, the region $\mathcal{X}_A\cap\mathcal{Y}_B$ now consists of 2 disjoint regions; the first two regions on the top left of Figure 3 have merged into one.  The red dots are $\sigma(A+iB)$ for 1000 randomly generated $3\times 3$ matrices $A$ and $B$ where $\sigma(A)=\{-1,0,2\}$ and $\sigma(B)=\{-s,0,s\}$.}
\end{figure}
\begin{figure}[h!]
\centering
\includegraphics[scale=.3]{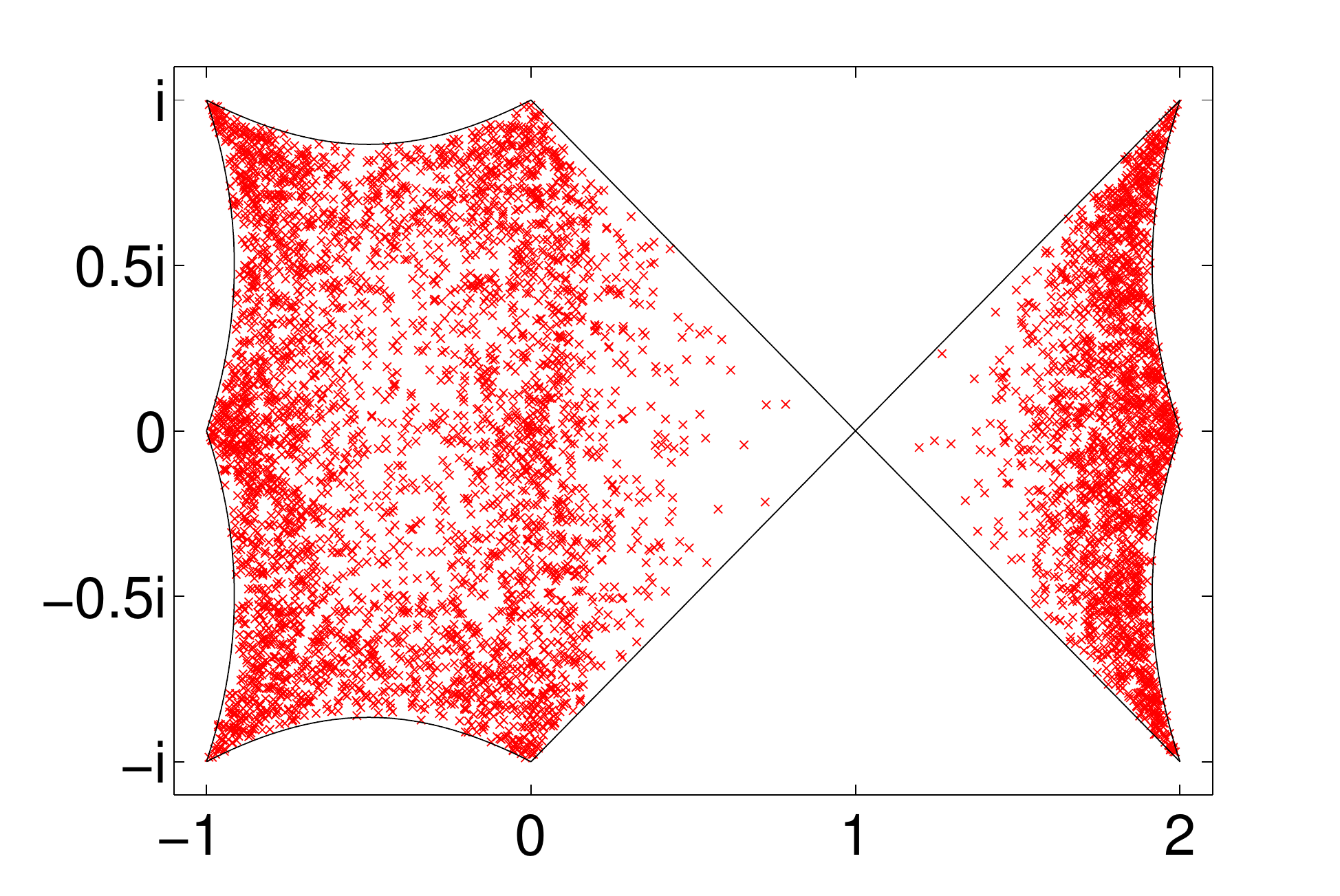}\includegraphics[scale=.3]{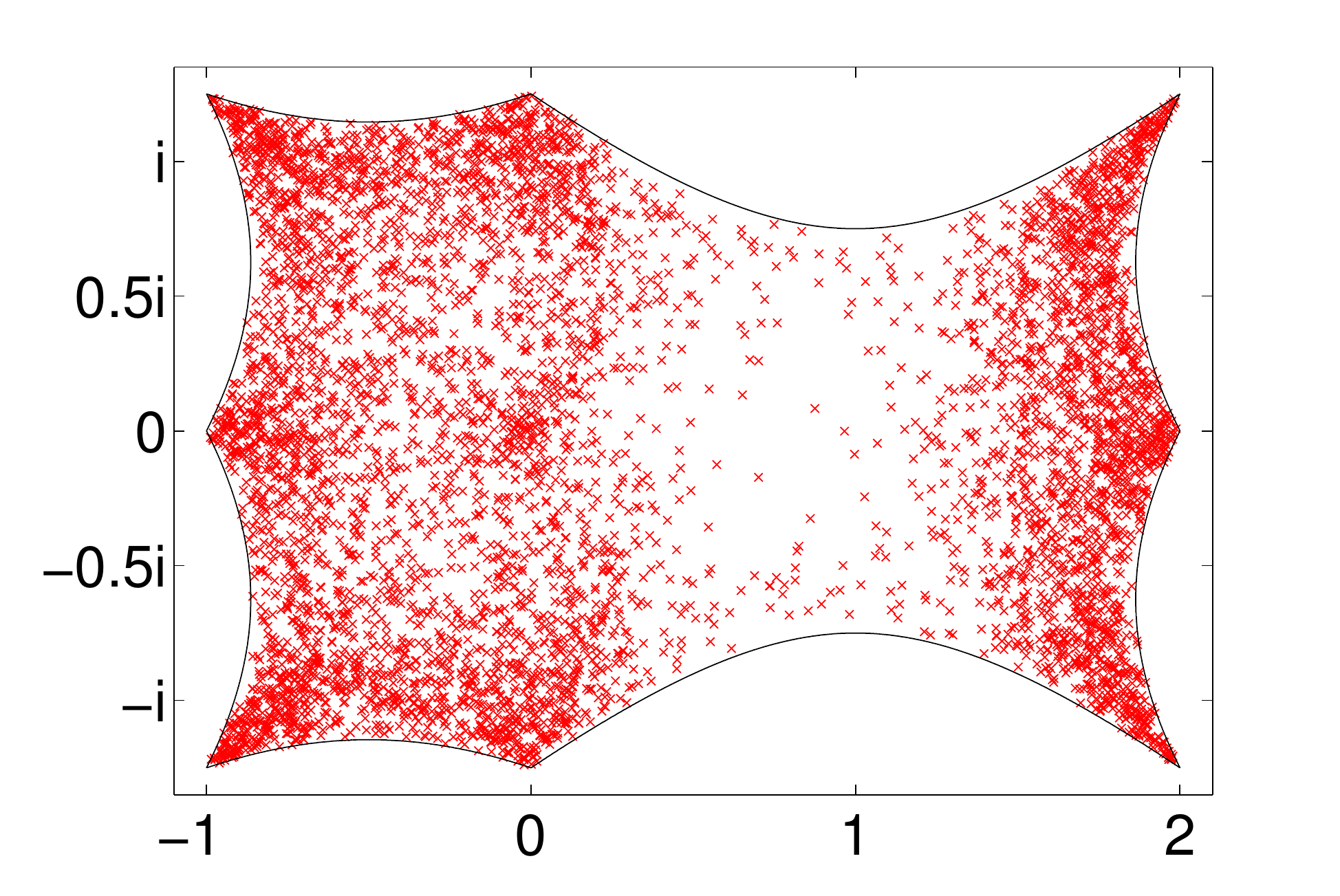}
\includegraphics[scale=.3]{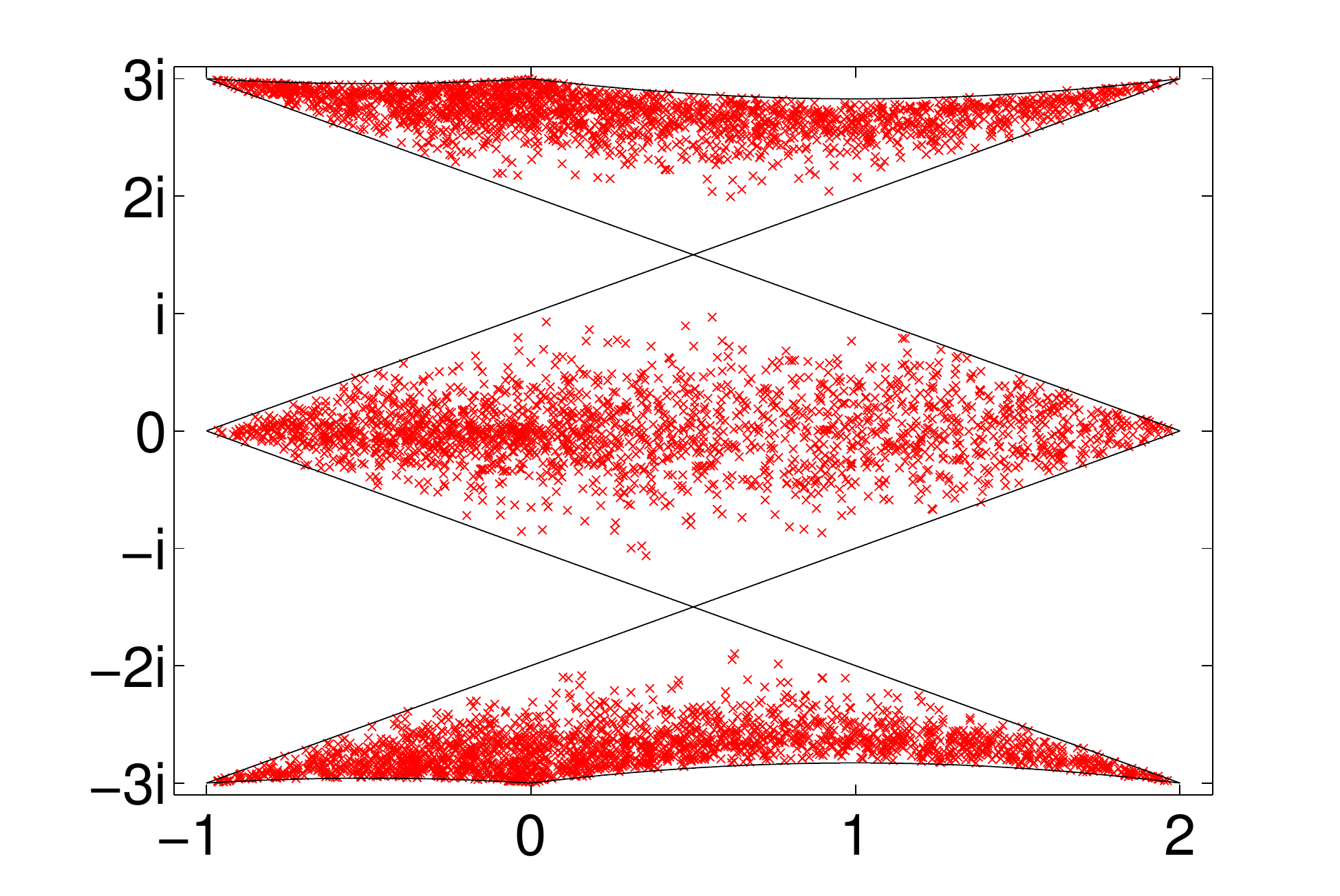}\includegraphics[scale=.3]{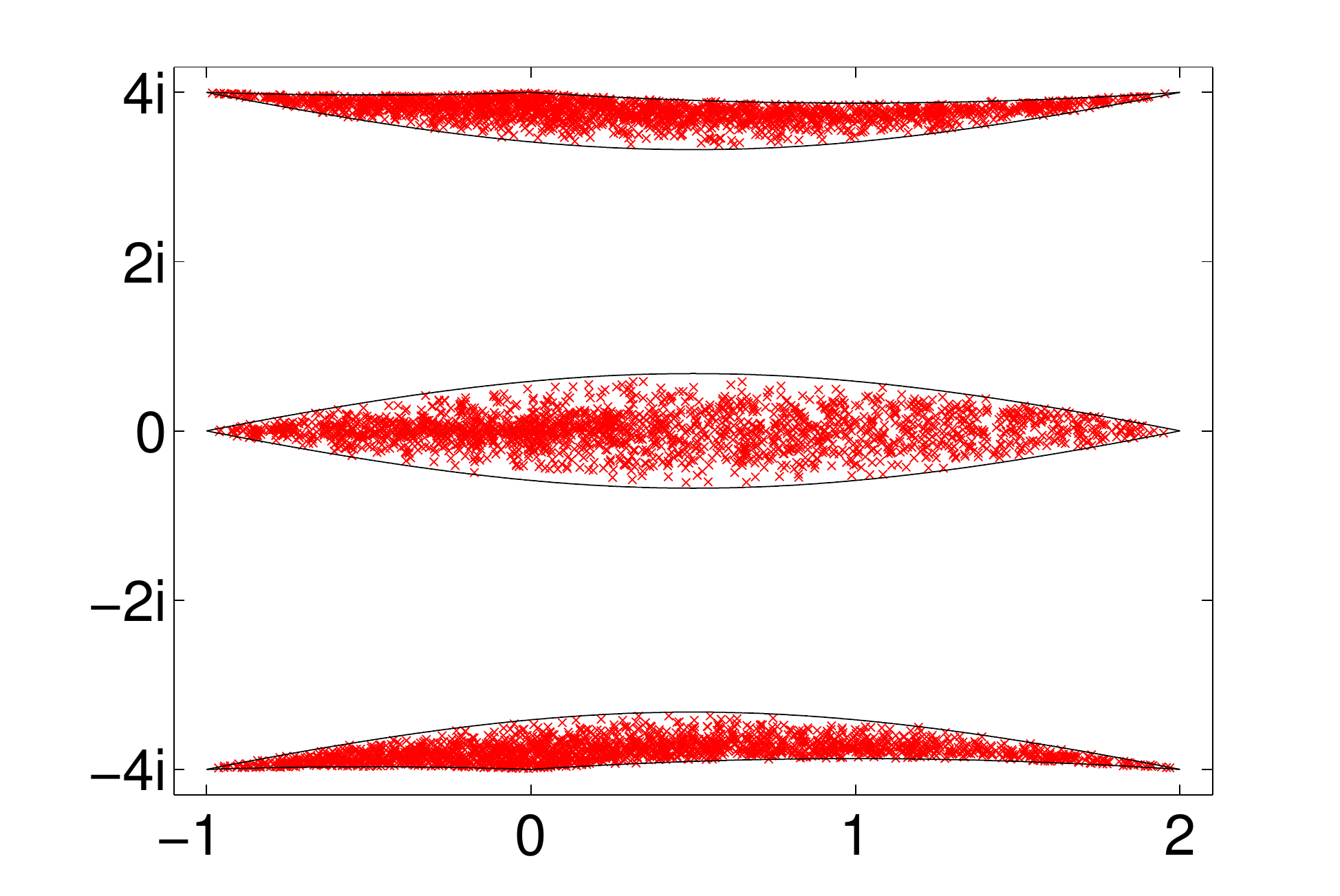}
\caption{Clockwise from top left $s=1,1.25,4,3$.  The red dots are $\sigma(A+iB)$ for 1000 randomly generated $3\times 3$ matrices $A$ and $B$ where $\sigma(A)=\{-1,0,2\}$ and $\sigma(B)=\{-s,0,s\}$.}
\end{figure}
\end{example}

\void{
\begin{example}
Let $\mathcal{H}=L^2(0,\pi)$, $A=-d^2/dx^2$ with homogeneous Dirichlet boundary conditions, and $-I\le B\le 5I$. Then, $\sigma(A)=\{n^2: n\in\mathbb{N}\}$ and Figure 6 shows the region $\mathcal{X}_A$.
\begin{figure}[h!]
\centering
\includegraphics[scale=.325]{may22pic1.eps}\includegraphics[scale=.325]{may22pic2.eps}
\caption{On the left are the first three regions enclosed of $\mathcal{X}_A$. On the right are the first nine regions enclosed of $\mathcal{X}_A$.}
\end{figure}
\end{example}
}

\void{
\section{Galerkin and Quadratic methods}
The Galerkin eigenvalues of $A$ with respect to a finite-dimensional trial space $\mathcal{L}\subset\Dom(\frak{a})$, denoted $\sigma(A,\mathcal{L})$, consists of those $\mu\in\mathbb{C}$ for which $\exists u\in\mathcal{L}\backslash\{0\}$ with
\[
\frak{a}(u,v) = \mu\langle u,v\rangle\quad\forall v\in\mathcal{L}.
\]
Unless stated otherwise, $(\cL_n)_{n\in\mathbb{N}}\subset\Dom(\frak{a})$ will be a sequence of finite-dimensional trial spaces with corresponding sequence of orthogonal projections $(P_n)$. We shall always assume that:
\begin{equation}\label{subspaces}
\forall u\in\Dom(\frak{a})\quad\exists u_n\in\cL_n:\quad\Vert u-u_n\Vert_{\frak{a}}\to0
\end{equation}
where $\Vert\cdot\Vert_{\frak{a}}$ is the norm associated to the Hilbert space $\mathcal{H}_\frak{a}$ with inner-product
\[
\langle u,v\rangle_{\frak{a}}:=\frak{a}(u,v) - (m-1)\langle u,v\rangle\quad\forall u,v\in\Dom(\frak{a})\quad\textrm{where}\quad m=\min\sigma(A).
\]
The distance from a subspace $\mathcal{M}\subset\mathcal{H}$ to another subspace $\mathcal{N}\subset\mathcal{H}$ is defined as
\[
\delta(\mathcal{M},\mathcal{N}) = \sup_{u\in\mathcal{M},\Vert u\Vert=1}\dist(u,\mathcal{N}),
\]
the gap between the two subspaces is
\[
\hat{\delta}(\mathcal{M},\mathcal{N}) = \max\big\{\delta(\mathcal{M},\mathcal{N}),\delta(\mathcal{N},\mathcal{M})\big\};
\]
see \cite[Section IV.2.1]{katopert} for further details. We shall write $\delta_{\frak{a}}$ and $\hat{\delta}_{\frak{a}}$ to indicate the distance and gap between subspaces of $\mathcal{H}_\frak{a}$. 
For trial spaces satisfying \eqref{subspaces} the Galerkin method is an extremely powerful tool for approximating those eigenvalues which lie below the essential spectrum; see for example \cite{chat}. It is well-known that
\begin{equation}\label{galerkin1}
\left(\lim_{n\to\infty}\sigma(A,\mathcal{L}_n)\right)\cap\big(-\infty,\min\sigma_{\ess}(A)\big) = \sigma(A)\cap\big(-\infty,\min\sigma_{\ess}(A)\big).
\end{equation}
Furthermore, for an eigenvalue $\lambda<\min\sigma_{\ess}(A)$ with eigenspace $\mathcal{L}(\{\lambda\})$, we have the \emph{superconvergence} property
\begin{equation}\label{galerkin2}
\dist\big(\lambda,\sigma(A,\mathcal{L}_n)\big) = \mathcal{O}\big(\delta_{\frak{a}}(\mathcal{L}(\{\lambda\}),\mathcal{L}_n)^2\big).
\end{equation}

In general, the Galerkin method cannot be relied upon for approximating eigenvalues above $\min\sigma_{\ess}(A)$. This is due to a phenomenon known as spectral pollution which is the presence of sequences of Galerkin eigenvalues which converge to points in $\rho(A)$. A typical situation is $\min\sigma_{\ess}(A)\le\alpha<\beta$, $(\alpha,\beta)\cap\sigma_{\ess}(A) = \varnothing$, and
\[
\left(\lim_{n\to\infty}\sigma(A,\mathcal{L}_n)\right)\cap(\alpha,\beta) = (\alpha,\beta).
\]
Hence, any approximation of $\sigma_{\dis}\cap(\alpha,\beta)$ is lost within an increasingly dense \emph{fog} of spurious Galerkin eigenvalues; see examples \ref{uber2} \& \ref{magneto} and \cite{boff,boff2,bost,dasu,DP,lesh,rapp}. Although this means that a direct application of the Galerkin method often fails to identify eigenvalues, in view of \eqref{galerkin1} and \eqref{galerkin2}, there is every reason to suppose that eigenvalues above $\min\sigma_{\ess}(A)$ could, in principle, be approximated with a superconvergent technique using trial spaces satisfying only  \eqref{subspaces}. The absence of such a technique has resulted in the development of quadratic methods.

The quadratic methods are so-called because of their reliance on truncations of the square of the operator in question; the Galerkin method relies only on the quadratic form. They have been studied and applied extensively over the last two decades. The quadratic method which has received the most attention is the second order relative spectrum. This is because it can be applied without \'a priori information and it was widely thought to approximate the whole spectrum of an arbitrary self-adjoint operator. The latter has recently been shown to be false; see \cite{shar2}. However, it is known that the method will reliably approximate the discrete spectrum of a self-adjoint operator and part of the discrete spectrum of a normal operator; see \cite{bo} and \cite{me3}, respectively.
The appeal of quadratic methods is that they can approximate eigenvalues without interference from spectral pollution, in fact, they can even provide enclosures for eigenvalues. The latter is often regarded as a major selling point of these methods. In practice though, we are more likely to be interested in accuracy and convergence rather than enclosures.

A drawback of quadratic methods is that they require trial spaces to belong to the operator domain. From a computational perspective this can be highly awkward as typically FEM software will not support the operator domain. Particularly inconvenient, is that for a second order differential operator we cannot use the standard FEM space of piecewise linear trial functions. Furthermore, it is straightforward to show that \eqref{subspaces}, with the added condition $\mathcal{L}_n\subset\Dom(A)$ $\forall n\in\mathbb{N}$, is not sufficient to ensure approximation of $\sigma_{\dis}(A)$. A sufficient condition is
\begin{equation}\label{subspaces1}
\forall u\in\Dom(A)\quad\exists u_n\in\cL_n:\quad\Vert u-u_n\Vert_A\to0;
\end{equation}
see \cite[Corollary 3.6]{bost}. With \eqref{subspaces1} satisfied, we have for each $\lambda\in\sigma_{\dis}(A)$ an element $z_n$ belonging to the second order spectrum of $A$ relative to $\mathcal{L}_n$ with
\begin{equation}\label{spec2con}
\vert\lambda-z_n\vert = \mathcal{O}\big(\delta_A(\mathcal{L}(\{\lambda\}),\mathcal{L}_n)\big)\quad\textrm{and}\quad\vert\lambda-\Re z_n\vert = \mathcal{O}\big(\delta_A(\mathcal{L}(\{\lambda\}),\mathcal{L}_n)^2\big)
\end{equation}
where $\delta_{A}(\mathcal{L}(\{\lambda\}),\mathcal{L}_n)$ is the distance from the eigenspace $\mathcal{L}(\{\lambda\})$ to the trial space $\mathcal{L}_n$ with respect to the graph norm; see \cite[Section 6]{me3}. That the convergence rate is measured in terms of the graph norm means that the convergence, and therefore the accuracy, of this method can be poor when compared to the superconvergence of the Galerkin method; see Example \ref{diracex} and \cite[Example 3.5 \& 4.3]{bost}. Convergence rates analogous to the right hand side of \eqref{spec2con} are also known for the Davies \& Plum and Zimmermann \& Mertins methods; see \cite[Lemma 2]{bost1}.
}

\section{Perturbation of $\sigma_{\dis}(A)$}

In this section we consider $\sigma(A+iP)$ where $P$ is a non-trivial orthogonal projection. 
Let $a,b\in\mathbb{R}$, $a,b\in\rho(A)$, $a<b$ and denote $\Delta=[a,b]$. For the remainder of this manuscript, we assume that
\[
\sigma(A)\cap\Delta=\{\lambda_1,\dots,\lambda_d\}\subset\sigma_{\dis}(A)\quad\text{where}\quad d<\infty.
\]
We are concerned with the perturbation of the eigenvalues $\{\lambda_1,\dots,\lambda_d\}$. By theorems \ref{lem1b} and \ref{cor1a}, we have
\[
\mathcal{U}_{a,\lambda_1}^{0,1}\cup\mathcal{U}_{\lambda_1,\lambda_2}^{0,1}\cup\dots\cup\mathcal{U}_{\lambda_{d-1},\lambda_d}^{0,1}\cup\mathcal{U}_{\lambda_d,b}^{0,1}\subset\rho(A+iP)\quad\text{and}\quad\sigma(A+iP)\subset\mathcal{X}_A.
\]
However, we shall be interested in the set
\begin{equation}\label{theinter}
\mathcal{U}_{a,b}^{0,1}\cap\sigma(A+iP).
\end{equation}
We will show that when $\Vert(I-P)E(\Delta)\Vert$ is sufficiently small, then \eqref{theinter} will consist only of elements in a small neighbourhood of $\Gamma_{a,b}^{0,1}$, and of eigenvalues which are in small neighbourhoods of the $\lambda_j+i$, $1\le j\le d$.

\begin{definition}
For $z\in\mathcal{U}_{a,b}^{0,1}$, and $K_{a,b}$ as in Definition \ref{rsK} (with $b^-=0$ and $b^+=1$), we set
\begin{align*}
d(z) &=\min\left\{\frac{\dist\Big(z,\Gamma_{a,b}^{0,1}\Big)^4}{K_{a,b}},\dist\Big(z,\big\{\lambda_1+i,\dots,\lambda_d+i\big\}\Big)\right\}. 
\end{align*}
\void{
For $\varepsilon>0$ we set
\begin{align*}
\mathcal{X}_\varepsilon:=\big\{z\in\mathcal{U}_\Delta:d(z)> 3\varepsilon\big\}.
\end{align*}
}
\end{definition}

\begin{proposition}\label{thm1}
Let $z\in\mathcal{U}_{a,b}^{0,1}$ and $d(z)>3\Vert (I-P)E(\Delta)\Vert$, then 
\begin{equation}\label{resolvent_bound2}
z\in\rho(A+iP)\quad\textrm{and}\quad\Vert (A+iP-z)^{-1}\Vert\le \Big(d(z)-3\Vert (I-P)E(\Delta)\Vert\Big)^{-1}.
\end{equation}
\end{proposition}
\begin{proof}
For simplicity let us denote $E=E(\Delta)$ and $\varepsilon=\Vert (I-P)E(\Delta)\Vert$. We readily deduce that
$\Vert(I-E)PE\Vert\le\varepsilon$ and $\Vert EP(I-E)\Vert\le\varepsilon$. With these inequalities and the
identity
$P = EPE +(I-E)PE + EP(I-E) + (I-E)P(I-E)$, we obtain for any $u\in\Dom(A)$
\begin{align*}
\Vert (A+iP -z)u\Vert & =\Vert(A-z)(I-E)u + (A-z)Eu + iPu\Vert\\
&= \Vert(A-z)(I-E)u + (A-z)Eu\\
&\quad + i\big(EPE+(I-E)PE \\
&\quad + EP(I-E)+ (I-E)P(I-E)\big)u\Vert\\
&\ge \Vert(A-z)(I-E)u + i(I-E)P(I-E)u\\
&\quad +(A-z)Eu+ iEPEu\Vert\\
&\quad - \Vert(I-E)PEu + EP(I-E)u\Vert\\
&\ge\Vert(A-z)(I-E)u + i(I-E)P(I-E)u\\
&\quad+ (A-z)Eu + iEPEu\Vert - 2\varepsilon\Vert u\Vert.
\end{align*}
The term $(A-z)Eu + iEPEu$ satisfies the estimate
\begin{align*}
\Vert(A-z)Eu + iEPEu\Vert &= \Vert(A-z+i)Eu + iE(P-I)Eu\Vert\\
&\ge (d(z)-\varepsilon)\Vert Eu\Vert.
\end{align*}
Next consider the term $(A-z)(I-E)u + i(I-E)P(I-E)u$. We have
\[
(A-z)(I-E) + i(I-E)P(I-E) :\mathcal{H}\ominus\mathcal{L}(\Delta) \to\mathcal{H}\ominus\mathcal{L}(\Delta).
\]
The restriction of $A$ to $\mathcal{H}\ominus\mathcal{L}(\Delta)$ is a self-adjoint operator with no spectrum in the interval $\Delta$. The restriction of $(I-E)P$ to $\mathcal{H}\ominus\mathcal{L}(\Delta)$ is a self-adjoint operator with $0\le(I-E)P\le 1$. Therefore, by Theorem \ref{lem1b}, 
\begin{align*}
\Vert(A-z)(I-E)u + i(I-E)P(I-E)u\Vert &\ge \frac{\dist(z,\Gamma_{a,b}^{0,1})^4}{K_{a,b}}\Vert(I-E)u\Vert\\
&\ge d(z)\Vert(I-E)u\Vert.
\end{align*}
Combining these three estimates yields the result.
\end{proof}

\void{
\begin{corollary}\label{cor0}
There exist constants $c_r,\varepsilon_r>0$, independent of $P$, such that whenever $\Vert (I-P)E(\Delta)\Vert\le\varepsilon_r$ and $\vert\lambda_j+i-z\vert=r$ for some $1\le j\le d$, we have
\begin{equation}\label{resb}
z\in\rho(A+iP)\quad\textrm{with}\quad\Vert(A+iP-z)^{-1}\Vert\le \frac{1}{c_r}.
\end{equation}
\end{corollary}
\begin{proof}
We may choose any $\varepsilon_r>0$ such that, for each $1\le j\le d$, 
\[d(z)-3\varepsilon_r>0\quad\textrm{for all}\quad\vert\lambda_j+i-z\vert=r.\]
It then follows, from Proposition \ref{thm1}, that $z\in\rho(A+iP)$ and $\Vert(A+iP-z)^{-1}\Vert$ is uniformly bounded for all $z\in\mathcal{U}_\Delta$ with $\vert\lambda_j+i-z\vert=r$.
It will therefore suffice to show that $A+iP-z$ is also invertible with uniformly bounded inverse whenever $\vert\lambda_j+i-z\vert=r$ and $\Im z>1$. That $z\in\rho(A+iP)$ follows immediately from \eqref{nr}. Suppose that $(z_n^\pm)$ is  sequence converging to $\lambda_j\pm r+i$ with $\Im z_n^\pm>1$ for every $n\in\mathbb{N}$. Let $(S_n)$ be a sequence of orthogonal projections with $\Vert (I-S_n)E(\Delta)\Vert\le\varepsilon_r$ for every $n\in\mathbb{N}$. If there exists a sequence of normalised vectors $(u_n)$ with
$(A+iS_n-z_n)u_n\to 0$, then
\[
\big(A+iS_n-(\lambda_j\pm r+i)\big)u_n = \big(A+iS_n-z_n^\pm)u_n +\big(z_n^\pm-(\lambda_j\pm r+i)\big)\to 0.
\]
However, by Corollary \ref{thm1}, we have the bound
\[
\big\Vert\big(A+iS_n-(\lambda_j\pm r+i)\big)^{-1}\big\Vert\le\frac{1}{d(\lambda_j\pm r+i)-3\varepsilon_r}\quad\forall n\in\mathbb{N}.
\]
The existence of a constant $c_r>0$ satisfying \eqref{resb} now follows from the estimate
\[
\Vert(A+iP-z)^{-1}\Vert\le\frac{1}{\dist\big(z,W(A+iP)\big)}\quad\forall z\notin W(A+iP).
\]
\end{proof}
}

\begin{lemma}\label{cor1}
Let $(I-P)E(\Delta)=0$. Then
\[\mathcal{U}_{a,b}^{0,1}\Big\backslash\{\lambda_1+i,\dots,\lambda_d+i\}\subset\rho(A+iP).\]
Moreover, $\lambda_j+i$ is an eigenvalue of $A+iP$ with spectral subspace $\mathcal{L}(\{\lambda_j\})$, and $A+iP-(\lambda_j+i)$ is Fredholm with index zero. 
\end{lemma}
\begin{proof}
The first assertion follows immediately from Proposition \ref{thm1}. Let $\lambda\in\{\lambda_1,\dots,\lambda_d\}$. Whenever $u\in\cL(\{\lambda\})$ we have $(A+iP)u = (\lambda + i)u$. Further, if $v\in\Dom(A)$ and $(A+iP)v = (\lambda+i)v$, then
$(A-\lambda)v = i(I-P)v$ and therefore
\begin{displaymath}
\langle(A-\lambda)v,v\rangle = i\langle(I-P)v,v\rangle = i\Vert(I-P)v\Vert^2.
\end{displaymath}
It follows that $(I-P)v=0$ and $(A-\lambda)v=0$, and hence $v\in\cL(\{\lambda\})$.
We deduce that $\nul(A+iP-(\lambda+i)) = \nul(A-\lambda)$ where $\mathcal{L}(\{\lambda\})$ is the null space for both operators. Suppose that $\lambda+i$ is not
semi-simple. Then there exists a non-zero vector $w\perp\cL(\{\lambda\})$ with
$(A+iP-\lambda-i)w = u\in\cL(\{\lambda\})$.
Hence,
\begin{displaymath}
(A-\lambda-i)w\perp\cL(\{\lambda\})\quad\textrm{with}\quad\Vert(A-\lambda-i)w\Vert>\Vert w\Vert,
\end{displaymath}
and
\begin{displaymath}
iPw = u - (A-\lambda-i)w\quad\textrm{where}\quad u\perp(A-\lambda-i)w.
\end{displaymath}
It follows that
\[\Vert Pw\Vert^2 = \Vert u\Vert^2 + \Vert(A-\lambda-i)w\Vert^2 > \Vert w\Vert^2,\]
which is a contradiction since $\Vert P\Vert =1$.

Next  we show that $A+iP-(\lambda+i)$ is Fredholm. The operator $A+iP-(\lambda+i)$ has closed range iff there exists a $\gamma>0$ such that
\begin{equation}\label{notclosed1}
\Vert(A+iP-(\lambda+i))v\Vert\ge\gamma\dist(v,\mathcal{L}\{\lambda\})\quad\forall v\in\Dom(A);
\end{equation}
see \cite[Theorem IV.5.2]{katopert}. We suppose that \eqref{notclosed1} is false. There exist $0\le\gamma_n\to 0$ and $v_n\in\Dom(A)$ with
\[
\Vert(A+iP-(\lambda+i))v_n\Vert<\gamma_n\dist(v_n,\mathcal{L}\{\lambda\}),\quad n\in\mathbb{N}.
\]
Set $\tilde{w}_n=(I-E(\{\lambda\})v_n$, note $\tilde{w}_n\ne 0$ for all $n\in\mathbb{N}$, and denote $w_n=\tilde{w}_n/\Vert\tilde{w}_n\Vert$. Using $(I-P)E(\{\lambda\})=0$, we have
\begin{align*}
\gamma_n&=\frac{\gamma_n\dist(v_n,\mathcal{L}\{\lambda\})}{\Vert\tilde{w}_n\Vert}>\frac{\Vert(A+iP-(\lambda+i))v_n\Vert}{\Vert\tilde{w}_n\Vert}
=\Vert(A+iP-(\lambda+i))w_n\Vert,
\end{align*}
and hence $(A-\lambda)w_n - i(I-P)w_n\to 0$. Since
\[
\langle(A-\lambda)w_n,w_n\rangle\in\mathbb{R}\quad\text{and}\quad \langle i(I-P)w_n,w_n\rangle = i\Vert(I-P)w_n\Vert^2,
\]
we deduce that $(I-P)w_n\to 0$ and therefore also that $(A-\lambda)w_n\to 0$. The latter implies that $\dist(w_n,\mathcal{L}(\{\lambda\}))\to 0$ which is a contradiction since $w_n\perp\mathcal{L}(\{\lambda\})$ and $\Vert w_n\Vert=1$ for all $n\in\mathbb{N}$. From the contradiction we deduce that $A+iP-(\lambda+i)$ has closed range. Furthermore,
\[\Big(A+iP-(\lambda+i)\Big)^* = A-iP-(\lambda-i)
\]
and arguing as above we deduce that $0$ is an eigenvalue of $A-iP-(\lambda-i)$ with null space $\mathcal{L}(\{\lambda\})$. Hence
\[
\text{def}(A+iP-(\lambda+i)) = \nul(A-iP-(\lambda-i)) = \nul(A+iP-(\lambda+i));
\]
see \cite[Theorem IV.5.13]{katopert}. Thus $A+iP-(\lambda+i)$ is Fredholm and the result follows.
\end{proof}
We fix a $\lambda\in\{\lambda_1,\dots,\lambda_d\}$ with $\dim\mathcal{L}(\{\lambda\})=\kappa$, and an $0<r<1/2$ with
\begin{equation}\label{rhyp}
\mathbb{D}(\lambda+i,2r)\cap\big(\sigma(A)+i\big)=\{\lambda+i\}\quad\text{and}\quad\mathbb{D}(\lambda+i,2r)\cap\Gamma_{a,b}^{0,1}=\varnothing.
\end{equation}
where $\mathbb{D}(x,y)$ is the closed disc with centre $x$ and radius $y$. 

\begin{lemma}\label{cor2}
If $\Vert (I-P)E(\Delta)\Vert$ is sufficiently small,
then
\[\mathbb{D}(\lambda+i,r)\cap\sigma(A+iP)\ne\varnothing
\]
and the dimension of the corresponding spectral subspace is equal to $\kappa$.
\end{lemma}
\begin{proof}
Let $u_1,\dots,u_\kappa$ be an orthonormal basis for $\cL(\{\lambda\})$. Set $v_j=Pu_j$ and let $v_{\kappa+1},v_{\kappa+2},\dots$ be such that
\[
\range(P) = \Span\{v_1,\dots,v_\kappa,v_{\kappa+1},v_{\kappa+2},\dots\}.
\]
For $t\in[0,1]$ set $w_j(t)=tu_j + (1-t)v_j$ and let $P_t$ be the orthogonal projection onto $\Span\{w_1(t),\dots,w_\kappa(t),v_{\kappa+1},v_{\kappa+2},\dots\}$. For any normalised $u\in\cL(\{\lambda\})$ we have $u=c_1u_1+\dots+c_ku_\kappa$ and
\begin{align*}
\Vert(I-P_t)u\Vert &\le \Vert c_1u_1+\dots+c_\kappa u_\kappa - c_1w_1(t)-\dots-c_\kappa w_\kappa(t)\Vert\\
&=(1-t)\Vert (1-P)u\Vert.
\end{align*}
Thus
\begin{equation}\label{projfam}
\Vert(I-P_t)E(\Delta)\Vert\le(1-t)\Vert(I-P)E(\Delta)\Vert\quad\forall t\in[0,1].
\end{equation}
By Lemma \ref{cor1}, we have $\lambda+i\in\sigma(A+iP_1)$ with spectral subspace $\cL(\{\lambda\})$. By Proposition \ref{thm1} and \eqref{projfam} we deduce that whenever $\Vert (I-P)E(\Delta)\Vert$ is sufficiently small, the operator $A + iP_t-zI$ is invertible with uniformly bounded inverse for all $\vert z-\lambda-i\vert= r$ and $t\in[0,1]$. Hence, we may define the family of spectral projections
\[
Q_t := \int_{\vert\lambda+i-z\vert=r}(A+iP_t - \zeta)^{-1}~d\zeta.
\]
Evidently, $Q(t)$ is a continuous family and therefore 
\[
\kappa=\dim\big(\cL(\{\lambda\})\big)= \rank(Q_1) = \rank(Q_t) \quad\forall t\in[0,1].
\]
\end{proof}

\void{
\subsection{Galerkin Subspaces and Projections}\label{galsub}

Associated to the form $\frak{a}$ is the Hilbert space $\mathcal{H}_\frak{a}$ which has inner-product
\[
\langle u,v\rangle_{\frak{a}}:=\frak{a}(u,v) - (m-1)\langle u,v\rangle\quad\forall u,v\in\Dom(\frak{a})\quad\textrm{where}\quad m=\min\sigma(A)
\]
and norm
\begin{align*}
\Vert u\Vert_{\frak{a}} &=\big(\frak{a}(u,u) - (m-1)\langle u,u\rangle\big)^{\frac{1}{2}}=\left(\int_{\mathbb{R}}\lambda-m+1~d\langle E_\lambda u,u\rangle\right)^{\frac{1}{2}}\\
&=\Vert(A-m+1)^{\frac{1}{2}}u\Vert.
\end{align*}
The gap or distance between two subspaces $\mathcal{M}$ and $\mathcal{N}$ of $\mathcal{H}$, is defined as
\[
\hat{\delta}(\mathcal{M},\mathcal{N}) = \max\big[\delta(\mathcal{M},\mathcal{N}),\delta(\mathcal{N},\mathcal{M})\big]
\quad\textrm{where}\quad
\delta(\mathcal{M},\mathcal{N}) = \sup_{u\in\mathcal{M},\Vert u\Vert=1}\dist(u,\mathcal{N});
\]
see \cite[Section IV.2.1]{katopert} for further details. We shall write $\delta_{\frak{a}}$ and $\hat{\delta}_{\frak{a}}$ to indicate the gap between subspaces of $\mathcal{H}_\frak{a}$. We denote by $P_n$ the orthogonal projection from $\mathcal{H}$ onto the trial space $\cL_n$, and we set
\[
\varepsilon_n=\delta_{\frak{a}}(\cL(\Delta),\cL_n),\quad \mathcal{L}_n(\Delta) := \range\big(E_n(\Delta)\big)\quad\textrm{and}\quad B_n:=E_n(\Delta)P_n.
\] 

\void{
\begin{lemma}\label{l1}
If $A$ is bounded, then $B_nu\longrightarrow E(\Delta)u$ for all $u\in\mathcal{H}$.
\end{lemma}
\begin{proof}
$A_n$ converges strongly to $A$, therefore $E_n((-\infty,a))$ and $E_n((b,\infty))$ converge strongly to $E((-\infty,a))$ and $E_n((b,\infty))$, respectively; see for example \cite[corollaryollary VIII.1.6]{katopert}. Let $u\in\mathcal{H}$, then
\begin{align*}
B_nu &=  P_nu - E_n((-\infty,a))P_nu - E_n((b,\infty))u\longrightarrow E(\Delta)u.
\end{align*}
\end{proof}
}

\begin{lemma}\label{l1b}
If $A$ is bounded, then
\[\Vert(I-B_n)E(\Delta)\Vert=\mathcal{O}\big(\delta(\cL(\Delta),\cL_n)\big)\quad\textrm{and}\quad\delta\big(\cL(\Delta),\cL_n(\Delta)\big)=\mathcal{O}\big(\delta(\cL(\Delta),\cL_n)\big).\]
\end{lemma}
\begin{proof}
Let $(A-\lambda_j)u = 0$ with $\Vert u\Vert=1$ and $\lambda_j\in\Delta$. Set $u_n=P_nu$, then
\begin{align*}
\Vert(A_n-\lambda_j)u_n\Vert 
=\Vert(P_nA-\lambda_j)u_n - P_n(A-\lambda_j) u\Vert\le\Vert A\Vert\delta(\cL(\Delta),\cL_n),
\end{align*}
and
\begin{align*}
\Vert(I-B_n)u_n\Vert^2 &\le\int_{\mathbb{R}\backslash\Delta}\frac{\vert\mu-\lambda_j\vert^2}{\dist[\lambda_j,\{a,b\}]^2}~d\langle (E_n)_\mu u_n,u_n\rangle\\
&\le\frac{1}{\dist[\lambda_j,\{a,b\}]^2}\int_{\mathbb{R}}\vert\mu-\lambda_j\vert^2~d\langle (E_n)_\mu u_n,u_n\rangle\\
&\le\frac{\Vert A\Vert^2\delta(\cL(\Delta),\cL_n)^2}{\dist(\lambda_j,\{a,b\})^2}.
\end{align*}
Therefore
\begin{align*}
\Vert(I- B_n)u\Vert &\le \Vert(I-B_n)u_n\Vert + \Vert(I-B_n)(u-u_n)\Vert\\
&\le\left(\frac{\Vert A\Vert}{\dist(\lambda_j,\{a,b\})} + 1\right)\delta(\cL(\Delta),\cL_n),
\end{align*}
from which both assertions follow.
\end{proof}

\void{
\begin{lemma}\label{l2}
If $T$ is unbounded, then $Q_nu \stackrel{\frak{t}}{\longrightarrow}E(\Delta)u$ for all $u\in\Dom(\frak{t})$.
\end{lemma}
\begin{proof}
Let $\psi_{n,1},\dots,\psi_{n,d_n}$ be orthonormal eigenvectors associated to $\sigma(T,\cL_n)\cap\Delta$, so that
\[\frak{t}(\psi_{n,j},v)=\mu_{n,j}\langle\psi_{n,j},v\rangle\quad\textrm{for all}\quad v\in\cL_n\quad\textrm{and where}\quad\mu_{n,j}\in\Delta.\]
For each $v\in\cL_n$ we set
\[(T-m+1)^{\frac{1}{2}}v=:\tilde{v}\in\tilde{\cL}_n:=(T-m+1)^{\frac{1}{2}}\cL_n,\]
hence
\begin{equation}\label{es}
\langle(T-m+1)^{-1}\tilde{\psi}_{n,j},\tilde{v}\rangle=\frac{1}{\mu_{n,j}-m+1}\langle\tilde{\psi}_{n,j},\tilde{v}\rangle\quad\textrm{for all}\quad\tilde{v}\in\tilde{\cL}_n
\end{equation}
where
\[
\frac{1}{\mu_{n,j}-m+1}\in\tilde{\Delta}:=\left[\frac{1}{b-m+1},\frac{1}{a-m+1}\right].
\]
Note that $E(\Delta)$ is the spectral projection associated to the self-adjoint operator $(T-m+1)^{-1}$ and the interval $\tilde{\Delta}$. Evidently, the set
\[\left\{\frac{\tilde{\psi}_{n,1}}{\sqrt{\mu_{n,1}-m+1}},\dots,\frac{\tilde{\psi}_{n,d_n}}{\sqrt{\mu_{n,d_n}-m+1}}\right\}\]
are orthonormal eigenvectors associated to $\sigma((T-m+1)^{-1},\tilde{\cL}_n)\cap\tilde{\Delta}$.

Let $u\in\Dom(\frak{t})$ and $(T-m+1)^{\frac{1}{2}}u=v$. Denote by $\tilde{P}_n$ the orthogonal projection from $\mathcal{H}$ onto $\tilde{\mathcal{L}}_n$. Using Lemma \ref{l1} and \eqref{es}, we have
\begin{align*}
\Big\Vert\sum_{j=1}^{d_n}\Big\langle(T-m+1)^{-\frac{1}{2}}&\tilde{P}_nv,
\psi_{n,j}\Big\rangle\psi_{n,j} - E(\Delta)u\Big\Vert_{\frak{t}}\\
&=\Big\Vert\sum\Big\langle\tilde{P}_nv,
(T-m+1)^{-1}\tilde{\psi}_{n,j}\Big\rangle\psi_{n,j} - E(\Delta)u\Big\Vert_{\frak{t}}\\
&=\Big\Vert\sum\frac{\langle\tilde{P}_nv,
\tilde{\psi}_{n,j}\rangle\psi_{n,j}}{\mu_{n,j}-m+1} - E(\Delta)u\Big\Vert_{\frak{t}}\\
&=\Big\Vert\sum\frac{\langle v,\tilde{\psi}_{n,j}\rangle\psi_{n,j}}{\mu_{n,j}-m+1} - E(\Delta)(T-m+1)^{-\frac{1}{2}}v\Big\Vert_{\frak{t}}\\
&=\Big\Vert\sum\frac{\langle v,\tilde{\psi}_{n,j}\rangle\tilde{\psi}_{n,j}}{\Vert\tilde{\psi}_{n,j}\Vert^2} - E(\Delta)v\Big\Vert\longrightarrow 0,
\end{align*}
and
\begin{align*}
\Vert Q_nu - E(\Delta)u\Vert_{\frak{t}}
&\le\left\Vert\sum\Big\langle(T-m+1)^{-\frac{1}{2}}\tilde{P}_nv,
\psi_{n,j}\Big\rangle\psi_{n,j} - E(\Delta)u\right\Vert_{\frak{t}}\\
&+\left\Vert\sum\Big\langle(T-m+1)^{-\frac{1}{2}}(I-\tilde{P}_n)v,
\psi_{n,j}\Big\rangle\psi_{n,j}\right\Vert_{\frak{t}}.
\end{align*}
The first term on the right hand side converges to zero, for the second term we have
\begin{align*}
&\Big\Vert\sum\Big\langle(T-m+1)^{-\frac{1}{2}}(I-\tilde{P}_n)v,
\psi_{n,j}\Big\rangle\psi_{n,j}\Big\Vert_{\frak{t}}\\
&\qquad\qquad=\Big\Vert\sum\Big\langle(T-m+1)^{-1}(I-\tilde{P}_n)v,
\tilde{\psi}_{n,j}\Big\rangle\tilde{\psi}_{n,j}\Big\Vert\\
&\qquad\qquad=\Big\Vert\sum(\mu_{n,j}-m+1)\bigg\langle(T-m+1)^{-1}(I-\tilde{P}_n)v,
\frac{\tilde{\psi}_{n,j}}{\Vert\tilde{\psi}_{n,j}\Vert}\bigg\rangle\frac{\tilde{\psi}_{n,j}}{\Vert\tilde{\psi}_{n,j}\Vert}\Big\Vert\\
&\qquad\qquad\le(b-m+1)\Vert(T-m+1)^{-1}(I-\tilde{P}_n)v\Vert\longrightarrow 0.
\end{align*}
\end{proof}
}

\begin{lemma}\label{l2b}
$\delta_{\frak{a}}\big(\cL(\Delta),\cL_n(\Delta)\big)=\mathcal{O}(\varepsilon_n)$.
\end{lemma}
\begin{proof}
Let  $\cL_n(\Delta) = \Span\{u_{n,1},\dots,u_{n,d_n}\}$
where the $u_{n,j}$ are orthonormal eigenvectors of $A_n$, so that
\[\frak{a}(u_{n,j},v)=\mu_{n,j}\langle u_{n,j},v\rangle\quad\forall v\in\cL_n\quad\textrm{where}\quad\mu_{n,j}\in\Delta.\]
For each $v\in\cL_n$ we set $(T-m+1)^{\frac{1}{2}}v=:\tilde{v}\in\tilde{\cL}_n:=(T-m+1)^{\frac{1}{2}}\cL_n$,
and hence
\begin{equation}\label{es}
\langle(T-m+1)^{-1}\tilde{u}_{n,j},\tilde{v}\rangle=\frac{1}{\mu_{n,j}-m+1}\langle\tilde{u}_{n,j},\tilde{v}\rangle\quad\forall\tilde{v}\in\tilde{\cL}_n
\end{equation}
where
\[
\frac{1}{\mu_{n,j}-m+1}\in\tilde{\Delta}:=\left[\frac{1}{b-m+1},\frac{1}{a-m+1}\right].
\]
Evidently, the set
\[\left\{\frac{\tilde{u}_{n,1}}{\sqrt{\mu_{n,1}-m+1}},\dots,\frac{\tilde{u}_{n,d_n}}{\sqrt{\mu_{n,d_n}-m+1}}\right\}\]
consists of orthonormal eigenvectors associated to $\sigma((T-m+1)^{-1},\tilde{\cL}_n)\cap\tilde{\Delta}$. It is straightforward to show that $\delta(\cL(\Delta),\tilde{\cL}_n)=\mathcal{O}(\varepsilon_n)$. Using Lemma \ref{l1b} we have for any normalised $u$ with $(T-\lambda)u = 0$ and $\lambda\in\Delta$,
\begin{align*}
\mathcal{O}(\varepsilon_n)&=\delta(\cL(\Delta),\tilde{\cL}_n)\\
&\ge\left\Vert\sum_{j=1}^{d_n}\left\langle u,\frac{\tilde{u}_{n,j}}{\Vert\tilde{u}_{n,j}\Vert}\right\rangle
\frac{\tilde{u}_{n,j}}{\Vert\tilde{u}
_{n,j}\Vert} - u\right\Vert\\
&=\left\Vert\sum\frac{\langle u,(A-m+1)^{\frac{1}{2}}u_{n,j}\rangle}{\mu_{n,j}-m+1}
(A-m+1)^{\frac{1}{2}}u_{n,j} - u\right\Vert\\
&=\left\Vert(A-m+1)^{\frac{1}{2}}\left(\sum\frac{\sqrt{\lambda-m+1}}{\mu_{n,j}-m+1}\langle u,u_{n,j}\rangle
u_{n,j} - \frac{u}{\sqrt{\lambda-m+1}}\right)\right\Vert\\
&=\left\Vert\sum\frac{\sqrt{\lambda-m+1}}{\mu_{n,j}-m+1}\langle u,u_{n,j}\rangle
u_{n,j} -\frac{u}{\Vert u\Vert_{\frak{a}}}\right\Vert_{\frak{a}}\\
&\ge \dist_{\frak{a}}\left(\frac{u}{\Vert u\Vert_{\frak{a}}},\cL_n(\Delta)\right).
\end{align*}
\end{proof}

\void{
\begin{corollary}\label{cor3}
There exists a constant $C_1\ge0$ such that
\[\sigma(A+iB_n)\cap\mathbb{D}(\lambda_j+i,C_1\varepsilon_n)\ne\varnothing\quad1\le j\le d\quad\textrm{for all sufficiently large}\quad n\in\mathbb{N}
\]
with the total multiplicity of the intersection equal to the multiplicity of $\lambda_j\in\sigma(T)$.
\end{corollary}
\begin{proof}
Using Lemma \ref{l2b}, there exists a constant $C_0\ge 0$ such that for each $u\in\cL(\Delta)$ with $\Vert u\Vert_\frak{a} =1$ there exists a $u_n\in\cL_n(\Delta)$ such that $\Vert u - u_n\Vert_{\frak{a}}\le C_0\varepsilon_n$. Since $\Delta = (\alpha,\beta)$ we deduce that $\Vert u\Vert>\alpha-m+1$, and hence that
\[
\left\Vert \frac{u}{\Vert u\Vert} -\frac{u_n}{\Vert u\Vert}\right\Vert \le \frac{C_0\varepsilon_n}{\alpha-m+1}.
\]
Therefore $\Vert(I-B_n)E(\Delta)\Vert = \mathcal{O}(\varepsilon_n)$, and the result follows from Corollary \ref{cor2}.
\end{proof}

\begin{remark}
An immediate consequence of Lemma \ref{l2b} is $\Vert(I-B_n)E(\Delta)\Vert=\mathcal{O}(\varepsilon_n)$. Hence with this Corollary \ref{cor2} holds with $B=B_n$ for all sufficiently large $n$. However, it would also seem possible that $\Vert(I-B_n)E(\Delta)\Vert$ could converge to zero faster than this estimate, since
\end{remark}
}
}

In view of Lemma \ref{cor2}, it is natural to consider operators of the form $A+iP_n$ where $(P_n)$ is as sequence of orthogonal projections which converge strongly to the identity operator. The range of $P_n$ is denoted $\mathcal{L}_n$. It follows from Proposition \ref{thm1} and Lemma \ref{cor2} that
\begin{equation}\label{alimit}
\lim_{n\to\infty}\sigma(A+iP_n)\cap\mathcal{U}_{a,b}^{0,1} = \big\{\lambda_1+i,\dots,\lambda_d+i\big\}.
\end{equation}
We prove below, in Theorem \ref{QQ}, that elements from $\sigma(A+iP_n)$ converge to $\lambda+i$ extremely rapidly. To this end, we denote by $\mathcal{M}_n(\{\lambda+i\})$ the spectral subspace  which corresponds  to $\sigma(A+iP_n)\cap\mathbb{D}(\lambda+i,r)$. We also set $\varepsilon_n=\delta(\cL(\Delta),\cL_n)$.
 
\begin{theorem}\label{est}
There exists a constant $c_0>0$ such that 
\[\hat{\delta}_{\frak{a}}\big(\cL(\{\lambda\}),\mathcal{M}_n(\{\lambda+i\})\big)\le c_0\varepsilon_n\quad\text{for all sufficiently large n}.\]
\end{theorem}
\begin{proof}
For simplicity, let us denote $E=E(\Delta)$. Note that
\[\sigma(A+iE)=\big\{\sigma(A)\backslash\{\lambda_1,\dots,\lambda_d\}\big\}\cup\{\lambda_1+i,\dots,\lambda_d+i\}\]
and the spectral subspace associated to $\lambda+i$ is $\cL(\{\lambda\})$. With $r$ as in \eqref{rhyp}, we have for any $\vert\lambda+i-z\vert=r$,
\[
z\in\rho(A+iE)\quad\textrm{with}\quad\Vert(A+iE-z)^{-1}\Vert=\frac{1}{r}.
\]
It follows easily from Proposition \ref{thm1} that there exists a $c_1>0$ and $N\in\mathbb{N}$ such that, for all $n\ge N$ and any $\vert\lambda+i-z\vert=r$, we have
\[z\in\rho(A+iP_n)\quad\textrm{with}\quad\Vert(A + iP_n-z)^{-1}\Vert\le\frac{1}{c_1}.
\]
Let $u\in\mathcal{H}$ with $\Vert u\Vert=1$ then, using the
identity
\begin{align*}
(A+iP_n - z)^{-1} &= (A+iE - z)^{-1}\\
&\quad+(A+iE - z)^{-1}(iE-iP_n)(A+iP_n - z)^{-1}
\end{align*}
and recalling that $m=\min\sigma(A)$, we obtain 
\begin{align*}
&\Vert(A-m+1)^{\frac{1}{2}}(A+iP_n - z)^{-1}u\Vert\\
&\qquad\le\Vert(A-m+1)^{\frac{1}{2}}(A+iE - z)^{-1}u\Vert\\
&\qquad\quad+\Vert(A-m+1)^{\frac{1}{2}}(A+iE - z)^{-1}(iE-iP_n)(A+iP_n - z)^{-1}u\Vert\\
&\qquad\le\Vert(A-m+1)^{\frac{1}{2}}(A+iE - z)^{-1}\Vert\\
&\qquad\quad+ 2\Vert(A-m+1)^{\frac{1}{2}}(A+iE - z)^{-1}\Vert\Vert(A+iP_n - z)^{-1}u\Vert\\
&\qquad\le \max_{\vert\lambda+i - z\vert=r}\left\{\frac{(2+c_1)\Vert(A-m+1)^{\frac{1}{2}}(A+iE - z)^{-1}\Vert}{c_1}\right\} =:M.
\end{align*}
Now let $u\in\cL(\{\lambda\})$ with $\Vert u\Vert=1$. The above estimate gives
\begin{align}
&\Vert(A+iE - z)^{-1}u -(A+iP_n - z)^{-1}u\Vert_{\frak{a}}\nonumber\\
&\qquad\qquad\qquad\qquad= \Vert(A+iP_n - z)^{-1}(P_n-E)(A+iE - z)^{-1}u\Vert_{\frak{a}}\nonumber\\
&\qquad\qquad\qquad\qquad= \frac{\Vert(A-m+1)^{\frac{1}{2}}(A+iP_n - z)^{-1}(P_n-I)u\Vert}{r}\nonumber\\
&\qquad\qquad\qquad\qquad\le\frac{M\Vert(I-P_n)E\Vert}{r}\nonumber\\
&\qquad\qquad\qquad\qquad=\frac{M\delta(\cL(\Delta),\cL_n)}{r}\label{Mr}.
\end{align}
Set
\begin{align*}
u_n&:=-\frac{1}{2i\pi}\int_{\vert\lambda+i-z\vert=r}(A + iP_n-\zeta)^{-1}u~d\zeta,
\end{align*}
then $u_n\in\mathcal{M}_n(\{\lambda+i\})$. Using estimate \eqref{Mr},
\begin{align*}
\Bigg\Vert&\frac{u}{\Vert u\Vert_{\frak{a}}}-\frac{u_n}{\Vert u\Vert_{\frak{a}}}\Bigg\Vert_{\frak{a}}\\
&\qquad=\frac{1}{2\pi\Vert u\Vert_{\frak{a}}}\Bigg\Vert\int_{\vert\lambda+i-z\vert=r}(A + iE-\zeta)^{-1}u - (A + iP_n-\zeta)^{-1}u~d\zeta\Bigg\Vert_{\frak{a}}\\
&\qquad\le\frac{1}{2\pi}\int_{\vert\lambda+i-z\vert=r}\Big\Vert(A + iE-\zeta)^{-1}u - (A + iP_n-\zeta)^{-1}u\Big\Vert_{\frak{a}}~\vert d\zeta\vert\\
&\qquad=\mathcal{O}\big(\delta(\cL(\Delta),\cL_n)\big)
\end{align*}
hence
\[\delta_{\frak{a}}\big(\cL(\{\lambda\}),\mathcal{M}_n(\{\lambda+i\})\big)=\mathcal{O}(\varepsilon_n).
\]
Furthermore, using Lemma \ref{cor2},
$\dim \mathcal{M}_n(\{\lambda+i\})=\dim \cL(\{\lambda\})=\kappa<\infty$ for all sufficiently large $n$, therefore the following formula holds
\begin{align*}
\delta_{\frak{a}}\big(\mathcal{M}_n(\{\lambda+i\}),\cL(\{\lambda\})\big)\le
\frac{\delta_{\frak{a}}\big(\cL(\{\lambda\}),\mathcal{M}_n(\{\lambda+i\})\big)}{1-\delta_{\frak{a}}\big(\cL(\{\lambda\}),\mathcal{M}_n(\{\lambda+i\})\big)};
\end{align*}
see \cite[Lemma 213]{kato}.
\end{proof}

It follows, from Theorem \ref{est}, that for all sufficiently large $n\in\mathbb{N}$, the operator $A+iP_n$ will have $\kappa$ (repeated) eigenvalues enclosed by the circle $\vert\lambda+i-z\vert=r$. We denote these eigenvalues by $\mu_{n,1},\dots,\mu_{n,\kappa}$. 

\begin{theorem}\label{QQ}
$\max_{1\le j\le\kappa}\vert\lambda+i-\mu_{n,j}\vert = \mathcal{O}(\varepsilon_n^2)$.
\end{theorem}
\begin{proof}
Let $u_1,\dots,u_\kappa$ be an orthonormal basis for $\cL(\{\lambda\})$. Let $Q_n$ be the orthogonal projection from $\mathcal{H}_{\frak{a}}$ onto $\mathcal{M}_n(\{\lambda+i\})$ and set $u_{n,j}=Q_nu_j$ for each $1\le j\le \kappa$. By Theorem \ref{est},
\[
\Vert u_j-u_{n,j}\Vert_{\frak{a}}=\Vert(I-Q_n)u_j\Vert_{\frak{a}} = \dist_{\frak{a}}\big(u_j,\mathcal{M}_n(\{\lambda+i\})\big) = \mathcal{O}(\varepsilon_n),
\]
and we may assume that $Q_n$ maps $\cL(\{\lambda\})$ one-to-one onto $\mathcal{M}_n(\{\lambda+i\})$. 

Consider the $\kappa\times\kappa$ matrices
\[
[L_n]_{p,q}=\langle(A+iP_n)u_{n,q},u_{n,p} \rangle\quad\textrm{and}\quad
[M_n]_{p,q}=\langle u_{n,q},u_{n,p} \rangle.
\]
Evidently, $M_n$ converges to the $\kappa\times\kappa$ identity matrix and $\sigma(L_nM_n^{-1})$ is precisely the set $\{\mu_{n,1},\dots,\mu_{n,\kappa}\}$.
We have
\[
[L_n]_{p,q} = \frak{a}(u_{n,q},u_{n,p}) + i\langle P_n u_{n,q},u_{n,p}\rangle.
\]
Consider the first term on the right hand side,
\begin{align*}
\frak{a}(u_{n,q},u_{n,p}) &= \frak{a}((Q_n-I)u_q,u_p) + \frak{a}((Q_n-I)u_q,(Q_n-I)u_p)\\
&\quad+\frak{a}(u_q,(Q_n-I)u_p)+\frak{a}(u_q,u_p)\\
&=\lambda\langle(Q_n-I)u_q,u_p\rangle + \frak{a}((Q_n-I)u_q,(Q_n-I)u_p)\\
&\quad+\lambda\langle u_q,(Q_n-I)u_p\rangle+\lambda\delta_{pq}
\end{align*}
where
\begin{align*}
\vert(\lambda-m+1)\langle u_q,(Q_n-I)u_p\rangle\vert&=\vert\frak{a}(u_q,(Q_n-I)u_p)\\
&\quad+(1-m)\langle u_q,(Q_n-I)u_p\rangle\vert\\
&=\vert\langle u_q,(Q_n-I)u_p\rangle_{\frak{a}}\vert\\
&=\vert\langle (Q_n-I)u_q,(Q_n-I)u_p\rangle_{\frak{a}}\vert\\
&\le\Vert(Q_n-I)u_q\Vert_{\frak{a}}\Vert(Q_n-I)u_p\Vert_{\frak{a}},
\end{align*}
hence $\frak{a}(u_{n,q},u_{n,p}) = \lambda\delta_{pq} + \mathcal{O}(\varepsilon_n^2)$.
Similarly,
\begin{align*}
\langle P_nu_{n,q},u_{n,p}\rangle&=\langle P_{n}(Q_n-I)u_q,(Q_n-I)u_p\rangle + \langle (Q_n-I)u_q,(P_{n}-I)u_p\rangle\\
&\quad+\langle (Q_n-I)u_q,u_p\rangle+\langle (P_{n}-I)u_q,(Q_n-I)u_p\rangle\\
&\quad+\langle u_q,(Q_n-I)u_p\rangle
+\langle(P_{n}-I)u_q,u_p\rangle+\delta_{pq}
\end{align*}
and
\begin{align*}
[M_n]_{p,q}&= \langle(Q_n-I)u_{q},(Q_n-I)u_{p} \rangle + \langle(Q_n-I)u_{q},u_{p} \rangle+ \langle u_{q},(Q_n-I)u_{p} \rangle\\
&\quad + \delta_{pq}.
\end{align*}
Hence
\[
i\langle P_nu_{n,q},u_{n,p}\rangle = i\delta_{pq} + \mathcal{O}(\varepsilon_n^2)\quad\textrm{and}\quad[M_n]_{p,q} = \delta_{pq}+\mathcal{O}(\varepsilon_n^2).
\]
Then
\[
[L_n]_{p,q}=(\lambda+i)\delta_{p,q}+\mathcal{O}(\varepsilon_n^2)\quad\textrm{and}\quad [M_n]^{-1}_{p,q} = \delta_{pq}+\mathcal{O}(\varepsilon_n^2),
\]
and we deduce that 
$[L_nM_n^{-1}]_{p,q}=(\lambda+i)\delta_{p,q}+\mathcal{O}(\varepsilon_n^2)$.
The result follows from the Gershgorin circle theorem.
\end{proof}

\section{The Perturbation Method}
 
The perturbation method, for locating $\sigma_{\dis}(A)$, was introduced in \cite{mar1} where it was formulated for Schr\"odinger operators. A more general version was presented in \cite{me} which required \'a priori knowledge about the location of gaps in the essential spectrum. In this section we present a new perturbation method which requires no \'a priori information and converges rapidly to $\sigma_{\dis}(A)$. In fact, our examples suggest that the method will actually capture the whole of $\sigma(A)$.

The idea is to perturb eigenvalues off the real line by adding a perturbation $iP$ where $P$ is a finite-rank orthogonal projection. The results from the previous sections allow us to perturb eigenvalues very precisely. The perturbed eigenvalues and their multiplicities may then be approximated with the Galerkin method without incurring spectral pollution; see \cite[Theorem 2.5 \& Theorem 2.9]{me}. As above, $(P_n)$ denotes a sequence of finite-rank orthogonal projections each with range $\mathcal{L}_n$. We shall assume that
\begin{equation}\label{subspaces}
\forall u\in\Dom(\frak{a})\quad\exists u_n\in\cL_n:\quad\Vert u-u_n\Vert_{\frak{a}}\to0.
\end{equation}
This is the usual hypothesis for a sequence of trial spaces
when using the Galerkin method. For sufficiently large $n$ we have, by Proposition \ref{thm1}, that
\[
\mathcal{U}_{a,b}^{0,1}\cap\sigma(A+iP_n)
\]
will consist of eigenvalues in a small neighbourhood of $\Gamma_{a,b}^{0,1}$, and, by Theorem \ref{QQ}, eigenvalues within $\varepsilon_n^2$ neighbourhoods of the $\lambda_j+i$; recall that
\begin{equation}\label{epg}
\varepsilon_{n}=\delta(\mathcal{L}(\Delta),\cL_{n}).
\end{equation}
We stress that $\varepsilon_n^2$ is extremely small; indeed, if pollution does not occur and we use the Galerkin method to approximate the eigenvalue $\lambda$, then our approximation will be of the order $\epsilon_n^2$ where
\begin{equation}\label{epb}
\epsilon_n:=\delta_{\frak{a}}(\mathcal{L}(\Delta),\cL_{n}).
\end{equation}

In this section we are concerned with the approximation of the eigenvalues of $A+iP_n$ using the Galerkin method. To this end, for our fixed $\lambda\in\{\lambda_1,\dots,\lambda_d\}$, let us fix an $N\in\mathbb{N}$ such that
\[
\dim\mathcal{M}_{n}(\{\lambda+i\})=\dim\cL(\{\lambda\})=\kappa\quad\forall n\ge N;
\]
such an $N$ is assured by Theorem \ref{est}.

Associated to the restriction of the form $\frak{a}$ to the trial space $\cL_k$ is a self-adjoint operator acting in the Hilbert space $\cL_k$; denote this operator and  corresponding spectral measure by $A_k$ and $E_k$, respectively. The Galerkin eigenvalues of $A+iP_n$ with respect to the trial space $\mathcal{L}_k$ are denoted $\sigma(A+iP_n,\mathcal{L}_k)$ and are precisely the eigenvalues of 
\[
A_k+iP_kP_n:\mathcal{L}_k\to\mathcal{L}_k.
\]
For our $\lambda\in\Delta$, we denote by $\mathcal{M}_{n,k}(\{\lambda+i\})$ the spectral subspace associated to the operator $A_k + iP_kP_n:\cL_k\to\cL_k$ and those eigenvalues in a neighbourhood of $\lambda+i$. Then, for a fixed $n\ge N$, we have for all sufficiently large $k$
\begin{equation}\label{eqdims}
\dim\mathcal{M}_{n,k}(\{\lambda+i\})=\dim\mathcal{M}_{n}(\{\lambda+i\})=\dim\cL(\{\lambda\})=\kappa.
\end{equation}
We now study the convergence properties of $\mathcal{M}_{n,k}(\{\lambda+i\})$ and associated eigenvalues, where our main convergence results are expressed in terms of $\varepsilon_{k}$ and $\epsilon_{n}$ from \eqref{epg} and \eqref{epb}, respectively.
We note that, using Theorem \ref{est},
\begin{equation}\label{dbounds}
\delta_{\frak{a}}(\mathcal{M}_n(\{\lambda+i\}),\cL_{k})\le\delta_{\frak{a}}(\mathcal{M}_n(\{\lambda+i\}),\cL(\{\lambda\})) + \delta_{\frak{a}}(\mathcal{L}(\{\lambda\}),\cL_{k})\le c_0\varepsilon_n + \epsilon_k
\end{equation}
where $c_0>0$ is independent of $n$. 

\begin{lemma}\label{l2d}
There exists a constant $c_2>0$, independent of $n\ge N$, such that
\[
\max_{\vert\lambda+i-z\vert=r}\Vert(A_k + iP_kP_n - z)^{-1}\Vert\le c_2\quad\textrm{for all sufficiently large } k.
\]
\end{lemma}
\begin{proof}
We assume that the assertion is false. Then there exist sequences $(n_p)$ and $(\gamma_p)$ with $\gamma_p\to\infty$, such that, for each fixed $p$ there is a subsequence $k_q$ with
\[
\max_{\vert\lambda+i-z\vert=r}\Vert(A_{k_q} + iP_{k_q}P_{n_p} - z)^{-1}\Vert> 2\gamma_p\quad\textrm{for all sufficiently large }q.
\]
Let us fix a $p$. We may assume, without loss of generality, that
\[
\max_{\vert\lambda+i-z\vert=r}\Vert(A_{k} + iP_{k}P_{n_p} - z)^{-1}\Vert> 2\gamma_p\quad\textrm{for all sufficiently large }k.
\]
Let $(z_k)$ be a sequence with $\vert\lambda+i-z_k\vert=r$ and
\[
\Vert(A_{k} + iP_{k}P_{n_p} - z_k)^{-1}\Vert> 2\gamma_p\quad\textrm{for all sufficiently large }k.
\]
The sequence $z_k$ has a convergent subsequence, without loss of generality, we assume that $z_k\to z$, where $\vert\lambda+i-z\vert=r$. For all sufficiently large $k$, there exists a normalised vector $u_k$ with
\[
\Vert(A_{k} + iP_{k}P_{n_p} - z_k)u_k\Vert<\frac{1}{2\gamma_p}.
\]
Then
\begin{align*}
\Vert(A_{k} + iP_{k}P_{n_p} - z)u_k\Vert &\le\Vert(A_{k} + iP_{k}P_{n_p} - z_k)u_k\Vert + \vert z_k-z\vert\\
&<\frac{1}{2\gamma_p} + \vert z_k-z\vert,
\end{align*}
and therefore
\[
\max_{\efrac{v\in\cL_{k}}{\Vert v\Vert=1}}\vert\frak{a}(u_k,v) + i\langle P_{n_p}u_k,v\rangle - z\langle u_k,v\rangle\vert <\frac{1}{\gamma_p}\quad\textrm{for all sufficiently large }k.
\]
The sequence $P_{n_p}u_k$ has a convergent subsequence. We assume, without loss of generality, that $iP_{n_p}u_k\to w$. Therefore 
\begin{align*}
\max_{\efrac{v\in\cL_{k}}{\Vert v\Vert=1}}\vert\frak{a}(u_k,v) + \langle w,v\rangle - z\langle u_k,v\rangle\vert <\frac{1}{\gamma_p} + \alpha_k\quad\textrm{for some} \quad 0\le\alpha_k\to 0.
\end{align*}
Denote by $\hat{P}_k$ the orthogonal projection from $\mathcal{H}_{\frak{a}}$ onto $\mathcal{L}_k$. Let $x=-(A-z)^{-1}w$ and set $x_{k}=\hat{P}_{k}x$, then for any $v\in\cL_k$
\begin{align*}
\frak{a}(x_{k},v) - z\langle x_{k},v\rangle &=
\frak{a}(x,v) - z\langle x,v\rangle - \frak{a}((I-\hat{P}_{k})x,v) + z\langle(I-\hat{P}_{k})x,v\rangle\\
&=\frak{a}(x,v) - z\langle x,v\rangle + (z-m+1)\langle(I-\hat{P}_{k})x,v\rangle\\
&=-\langle w,v\rangle + (z-m+1)\langle(I-\hat{P}_{k})x,v\rangle.
\end{align*}
We deduce that
\begin{align*}
\max_{\efrac{v\in\cL_{k}}{\Vert v\Vert=1}}\vert\frak{a}(u_k-x_{k},v) - z\langle u_k-x_{k},v\rangle\vert <\frac{1}{\gamma_p} + \beta_k\quad\textrm{for some} \quad 0\le\beta_k\to 0,
\end{align*}
hence
\[
\Vert u_k-x_{k}\Vert <\left(\frac{1}{\gamma_p} + \beta_k\right)\Big/\Im z\le \frac{1}{\gamma_p(1-r)} + \frac{\beta_k}{(1-r)}
\]
and therefore
\begin{equation}\label{psib}
\Vert x\Vert\leftarrow\Vert x_{k}\Vert > 1 -\frac{1}{\gamma_p(1-r)} - \frac{\beta_k}{(1-r)}\to 1-\frac{1}{\gamma_p(1-r)}.
\end{equation}
Let $y=(A+iP_{n_p}-z)x= -w - iP_{n_p}(A-z)^{-1}w$. Since $iP_{n_p}u_k\to w$ implies that $w\in\cL_{n_p}\subset\mathcal{H}_{\frak{a}}$, we deduce that $y\in\mathcal{H}_{\frak{a}}$ and we set $y_{k}=\hat{P}_{k}y$. Using \eqref{psib} and with $c_1>0$ as in the proof of Theorem \ref{est},
\begin{align*}
\vert\frak{a}(x_{k},y_{k}) + i\langle P_{n_p}x_{k},y_{k}\rangle - z\langle x_{k},y_{k}\rangle\vert &\to \Vert(A+iP_{n_p}-z) x\Vert^2\\&\ge c_1^2\left(1-\frac{1}{\gamma_p(1-r)}\right)^2.
\end{align*}
Furthermore, using the estimates above we have
\begin{align*}
\vert\frak{a}(x_{k},y_{k}) + i\langle &P_{n_p}x_{k},y_{k}\rangle - z\langle x_{k},y_{k}\rangle\vert\\
&=\vert\frak{a}(x_{k}-u_k,y_k) + i\langle P_{n_p}(x_{k}-u_k),y_{k}\rangle - z\langle x_{k}-u_k,y_{k}\rangle\\
&\quad+\frak{a}(u_k,y_{k}) + i\langle P_{n_p}u_k,y_{k}\rangle - z\langle u_k,y_{k}\rangle\vert\\
&\le \vert\frak{a}(x_{k}-u_k,y_k) - z\langle x_{k}-u_k,y_{k}\rangle\vert + \vert\langle P_{n_p}(x_{k}-u_k),y_{k}\rangle\vert\\
&\quad+\vert\frak{a}(u_k,y_{k}) + i\langle P_{n_p}u_k,y_{k}\rangle - z\langle u_k,y_{k}\rangle\vert\\
&<\left(\frac{1}{\gamma_p} + \beta_k\right)\Vert y_{k}\Vert +\left(\frac{1}{\gamma_p(1-r)} + \frac{\beta_k}{(1-r)}\right)\Vert y_{k}\Vert + \frac{1}{\gamma_p}\Vert y_{k}\Vert.
\end{align*}
Since $y=(A+iP_{n_p}-z)x = -w - iP_{n_p}(A-z)^{-1}w$ where $\Vert w\Vert\le 1$,
\[\Vert y_k\Vert\to\Vert y\Vert=\Vert-w - iP_{n_p}(A-z)^{-1}w\Vert\le\Vert w\Vert+\Vert iP_{n_p}(A-z)^{-1}w\Vert\le 1+\frac{1}{1-r},
\]
hence
\begin{align*}
\left(\frac{1}{\gamma_p} + \beta_k\right)\Vert y_{k}\Vert +\left(\frac{1}{\gamma_p(1-r)} + \frac{\beta_k}{1-r}\right)&\Vert y_{k}\Vert + \frac{1}{\gamma_p}\Vert y_{k}\Vert\\
&\to\left(\frac{2}{\gamma_p}+\frac{1}{\gamma_p(1-r)}\right)\Vert y\Vert\\
&\le \left(\frac{2}{\gamma_p}+\frac{1}{\gamma_p(1-r)}\right)\left(1+\frac{1}{1-r}\right).
\end{align*}
Therefore, we have
\begin{align*}
c_1^2\left(1-\frac{1}{\gamma_p(1-r)}\right)^2&\le\Vert(A+iP_{n_p}-z)x\Vert^2\\
&\leftarrow \vert\frak{a}(x_{k},y_{k}) + i\langle P_{n_p}x_{k},y_{k}\rangle - z\langle x_{k},y_{k}\rangle\vert\\
&\le\left(\frac{1}{\gamma_p} + \beta_k\right)\Vert y_{k}\Vert +\left(\frac{1}{\gamma_p(1-r)} + \frac{\beta_k}{1-r}\right)\Vert y_{k}\Vert\\
&\quad + \frac{1}{\gamma_p}\Vert y_{k}\Vert\\
&\to\left(\frac{2}{\gamma_p}+\frac{1}{\gamma_p(1-r)}\right)\Vert y\Vert\\
&\le \left(\frac{2}{\gamma_p}+\frac{1}{\gamma_p(1-r)}\right)\left(1+\frac{1}{1-r}\right).
\end{align*}
Evidently, the left hand side is larger than the right hand side for all sufficiently large $p$. The result follows from the contradiction.
\end{proof}

\begin{theorem}\label{limlem1}
There exist constants $c_3,c_4>0$, both independent of $n\ge N$, such that
\begin{equation}\label{sscon1}
\hat{\delta}_{\frak{a}}\big(\mathcal{M}_{n}(\{\lambda+i\}),\mathcal{M}_{n,k}(\{\lambda+i\})\big)\le c_3\delta_{\frak{a}}(\mathcal{M}_n(\{\lambda+i\}),\cL_{k})
\end{equation}
and
\begin{equation}\label{sscon2}
\hat{\delta}_{\frak{a}}\big(\mathcal{M}_{n,k}(\{\lambda+i\}),\mathcal{L}(\{\lambda\})\big)\le c_4(\varepsilon_n+\epsilon_k)
\end{equation}
for all sufficiently large $k$.
\end{theorem}
\begin{proof}
First we prove \eqref{sscon1}. Let $u\in\mathcal{M}_{n}(\{\lambda+i\})$ with $\Vert u\Vert=1$. For $\vert\lambda+i-z\vert=r$, we denote
\[A_k(z)=A_k+iP_kP_n-z\quad\textrm{and}\quad x(z)=(A + iP_n-z)^{-1}u\in\mathcal{M}_{n}(\{\lambda+i\}).
\]
Then, with $c_1>0$ as in the proof of Theorem \ref{est}, we have $\Vert x(z)\Vert\le c_1^{-1}$ and therefore
\begin{align*}
\Vert x(z)\Vert_{\frak{a}}^2 &= \frak{a}[x(z)] -(m-1)\Vert x(z)\Vert^2\\
&= \langle Ax(z),x(z)\rangle -(m-1)\Vert x(z)\Vert^2\\
&= \langle A(A + iP_n-z)^{-1}u,x(z)\rangle -(m-1)\Vert x(z)\Vert^2\\
&= \langle u,x(z)\rangle -\langle(iP_n-z)x(z),x(z)\rangle-(m-1)\Vert x(z)\Vert^2\\
&\le \Vert x(z)\Vert + (2+m+\vert z\vert)\Vert x(z)\Vert^2\\
&\le\frac{1}{c_1} + \frac{2+m+\vert z\vert}{c_1^2}.
\end{align*}
Hence
\begin{equation}\label{sbd}
\Vert(A + iP_n-z)^{-1}u\Vert_{\frak{a}} = \Vert x(z)\Vert_{\frak{a}}\le K_1
\end{equation}
for constant $K_1>0$ which is independent of $n\ge N$ and $\vert\lambda+i-z\vert=r$. Let $v\in\cL_k$ with $\Vert v\Vert=1$, then
\begin{align*}
\langle A_k(z)\hat{P}_kx(z)-u,v\rangle
&= \frak{a}(\hat{P}_kx(z),v) + i\langle P_{n}\hat{P}_kx(z),v\rangle - z\langle \hat{P}_kx(z),v\rangle - \langle u,v\rangle\\
&=i\langle P_{n}(\hat{P}_k-I)x(z),v\rangle - (z-m+1)\langle (\hat{P}_k-I)x(z),v\rangle.
\end{align*}
Hence
\begin{displaymath}
\Vert A_k(z)\hat{P}_kx(z) - P_ku\Vert
\le\big(1+\vert(z-m+1)\vert\big)\Vert(\hat{P}_k-I)x(z)\Vert
\end{displaymath}
then, using Lemma \ref{l2d},
\begin{align*}
\Vert A_k(z)^{-1}P_ku - \hat{P}_kx(z)\Vert&\le c_2\Vert A_k(z)\hat{P}_kx(z) - P_ku\Vert\\
&\le c_2\big(1+\vert(z-m+1)\vert\big)\Vert(\hat{P}_k-I)x(z)\Vert
\end{align*}
where $c_2$ is independent of $n\ge N$ and $\vert\lambda+i-z\vert=r$. Furthermore,
\begin{align*}
\Vert A_k(z)^{-1}P_ku - x(z)\Vert_{\frak{a}}&\le\Vert A_k(z)^{-1}P_ku - \hat{P}_kx(z)\Vert_{\frak{a}} + \Vert(\hat{P}_k-I)x(z)\Vert_{\frak{a}}
\end{align*}
where
\begin{align*}
\Vert A_k(z)^{-1}&P_ku - \hat{P}_kx(z)\Vert_{\frak{a}}^2\\
&=(\frak{a}-m)[A_k(z)^{-1}P_ku - \hat{P}_kx(z)]
+ \Vert A_k(z)^{-1}P_ku - \hat{P}_kx(z)\Vert^2\\
&=\langle P_ku -A_k(z)\hat{P}_kx(z),A_k(z)^{-1}P_ku - \hat{P}_kx(z)\rangle\\
&\quad-\langle(iP_kP_{n}-z)(A_k(z)^{-1}P_ku - \hat{P}_kx(z)),A_k(z)^{-1}P_ku - \hat{P}_kx(z)\rangle\\
&\quad+ (1-m)\Vert A_k(z)^{-1}P_ku - \hat{P}_kx(z)\Vert^2\\
&\le\Vert P_ku -A_k(z)\hat{P}_kx(z)\Vert\Vert A_k(z)^{-1}P_ku - \hat{P}_kx(z)\Vert\\
&\quad+\Vert iP_kP_{n}-z\Vert\Vert A_k(z)^{-1}P_ku - \hat{P}_kx(z)\Vert^2\\
&\quad+ \vert 1-m\vert\Vert A_k(z)^{-1}P_ku - \hat{P}_kx(z)\Vert^2\\
&\le c_2\big(1+ \vert(z-m+1)\vert\big)^2\Vert(\hat{P}_k-I)x(z)\Vert^2\\
&\quad+ \big(1+\vert z\vert\big)c_2^2\big(1+\vert(z-m+1)\vert\big)^2\Vert(\hat{P}_k-I)x(z)\Vert^2\\
&\quad+\vert 1-m\vert c_2^2\big(1+\vert(z-m+1)\vert\big)^2\Vert(\hat{P}_k-I)x(z)\Vert^2.
\end{align*}
Therefore, 
\begin{align}
\Vert A_k(z)^{-1}P_ku - (A + iP_{n}-z)^{-1}u\Vert_{\frak{a}}&\le K_2\Vert(\hat{P}_k-I)x(z)\Vert_{\frak{a}}\nonumber\\\
&\le K_2\Vert x(z)\Vert_{\frak{a}}\delta_{\frak{a}}(\mathcal{M}_{n}(\{\lambda+i\}),\cL_k)\nonumber\\
&\le K_1K_2\delta_{\frak{a}}(\mathcal{M}_{n}(\{\lambda+i\}),\cL_k)\label{ifff}
\end{align}
for constant $K_2>0$ which is independent of $n\ge N$ and $\vert\lambda+i-z\vert=r$. Set
\begin{align*}
u_k&:=-\frac{1}{2i\pi}\int_{\vert\lambda+i-z\vert=r}A_k(\zeta)^{-1}P_ku~d\zeta,
\end{align*}
then $u_k\in\mathcal{M}_{n,k}(\{\lambda+i\})$ and
\begin{align*}
\Bigg\Vert\frac{u}{\Vert u\Vert_{\frak{a}}}&-\frac{u_k}{\Vert u\Vert_{\frak{a}}}\Bigg\Vert_{\frak{a}}=\\
&\frac{1}{2\pi\Vert u\Vert_{\frak{a}}}\Bigg\Vert\int_{\vert\lambda+i-z\vert=r}A_k(\zeta)^{-1}P_ku - (A + iP_n-\zeta)^{-1}u~d\zeta\Bigg\Vert_{\frak{a}}\\
&\le\frac{1}{2\pi\Vert u\Vert_{\frak{a}}}\int_{\vert\lambda+i-z\vert=r}\big\Vert A_k(\zeta)^{-1}P_ku - (A + iP_n-\zeta)^{-1}u\big\Vert_{\frak{a}}~\vert d\zeta\vert.
\end{align*}
Combining this estimate with \eqref{ifff}, we deduce that for some constant $K_3> 0$ which is independent of  $n\ge N$ and $\vert\lambda+i-z\vert=r$, we have
\[\delta_{\frak{a}}(\mathcal{M}_n(\{\lambda+i\}),\mathcal{M}_{n,k}(\{\lambda_j+i\}))\le K_3\delta_{\frak{a}}(\mathcal{M}_{n}(\{\lambda+i\}),\cL_k).
\]
Then, by virtue of \eqref{eqdims}, the following formula holds for all sufficiently large $k$,
\[
\delta_{\frak{a}}\big(\mathcal{M}_{n,k}(\{\lambda+i\}),\mathcal{M}_n(\{\lambda+i\})\big)\le\frac{\delta_{\frak{a}}\big(\mathcal{M}_n(\{\lambda+i\}),\mathcal{M}_{n,k}(\{\lambda+i\})\big)}{1-\delta_{\frak{a}}\big(\mathcal{M}_n(\{\lambda+i\}),\mathcal{M}_{n,k}(\{\lambda+i\})\big)}.
\]
The assertion \eqref{sscon1} is proved. Now using \eqref{sscon1}, \eqref{dbounds} and Theorem \ref{est}, we have
\begin{align*}
\delta_{\frak{a}}\big(\mathcal{M}_{n,k}(\{\lambda+i\}),\mathcal{L}(\{\lambda\})\big)&\le\delta_{\frak{a}}\big(\mathcal{M}_{n,k}(\{\lambda+i\}),\mathcal{M}_n(\{\lambda+i\})\big)\\
&\quad +\delta_{\frak{a}}\big(\mathcal{M}_{n}(\{\lambda+i\}),\mathcal{L}(\{\lambda\})\big)\\
&\le c_3\delta_{\frak{a}}\big(\mathcal{M}_{n}(\{\lambda+i\}),\mathcal{L}_k\big)+c_0\varepsilon_n\\
&\le c_0(c_3+1)\varepsilon_n + c_3\epsilon_k.
\end{align*}

\end{proof}

Let $\mu_{n,k,1},\dots,\mu_{n,k,\kappa}$ be the repeated eigenvalues of $A_k + iP_kP_n$ which are associated to the subspace $\mathcal{M}_{n,k}(\{\lambda+i\})$.

\begin{theorem}\label{eigconv}
There exists a constant $c_5>0$, independent of $n\ge N$, such that
\[\max_{1\le j\le \kappa}\vert\mu_{n,k,j}-\lambda-i\vert \le c_5(\varepsilon_n+\epsilon_k)^2
\]
for all sufficiently large $k$.
\end{theorem}
\begin{proof}
Let $u_1,\dots,u_\kappa$ be an orthonormal basis for $\cL(\{\lambda\})$. Let $R_k$ be the orthogonal projection from $\mathcal{H}_{\frak{a}}$ onto $\mathcal{M}_{n,k}(\{\lambda+i\})$ and set $u_{j,k}=R_ku_j$. Using Theorem \ref{limlem1},
\begin{align*}
\Vert u_j-u_{j,k}\Vert_{\frak{a}}&=\Vert(I-R_k)u_j\Vert_{\frak{a}}\\
&=\dist_{\frak{a}}\big(u_j,\mathcal{M}_{n,k}(\{\lambda+i\})\big)\\
&\le \Vert u_j\Vert_{\frak{a}}\hat{\delta}_{\frak{a}}\big(\cL(\{\lambda\}),\mathcal{M}_{n,k}(\{\lambda+i\})\big)\\
&\le\Vert u_j\Vert_{\frak{a}}c_4(\varepsilon_n+\epsilon_k)\\
&\le K_4(\varepsilon_n+\epsilon_k)
\end{align*}
for constant $K_4>0$ which is independent of $n\ge N$. Consider the matrices
\[
[L_{n,k}]_{p,q}=\frak{a}(u_{q,k},u_{p,k}) + i\langle P_{n}u_{q,k},u_{p,k}\rangle\quad\textrm{and}\quad
[M_{n,k}]_{p,q}=\langle u_{q,k},u_{p,k} \rangle.
\]
Evidently, $\sigma(L_{n,k}M_{n,k}^{-1})$ is precisely the set $\{\mu_{n,k,1},\dots,\mu_{n,k,\kappa}\}$. 
\void{Let $Q_n$ be the orthogonal projection from $\mathcal{H}_{\frak{a}}$ onto $\mathcal{M}_{n}(\{\lambda_j+i\})$ and consider the matrix
\[
[M_n]_{p,q}=\langle Q_nu_{q},Q_nu_{p} \rangle
\]
From the proof of Theorem \ref{QQ}, we have $[M_n]_{p.q} = \delta_{p,q} + \mathcal{O}(\tau_n^2)$. Hence, there exists a sequence $(\delta_n)\subset\mathbb{R}$ with $\delta_n\to0$ and 
\[
\Vert N_{n,k}^{-1} - I_\kappa\Vert \le \delta_n\quad\textrm{for all sufficiently large }k.
\]
}
We have
\begin{align*}
\frak{a}(u_{q,k},u_{p,k}) &= \frak{a}((R_k-I)u_q,u_p) + \frak{a}((R_k-I)u_q,(R_k-I)u_p)\\
&\quad+\frak{a}(u_q,(R_k-I)u_p)+\frak{a}(u_q,u_p)\\
&=\lambda\langle(R_k-I)u_q,u_p\rangle + \frak{a}((R_k-I)u_q,(R_k-I)u_p)\\
&\quad+\lambda\langle u_q,(R_k-I)u_p\rangle+\lambda\delta_{qp}
\end{align*}
and
\begin{align*}
\vert(\lambda-m+1)\langle u_q,(R_k-I)u_p\rangle\vert&=\vert\langle u_q,(R_k-I)u_p\rangle_{\frak{a}}\vert\\
&=\vert\langle(R_k-I)u_q,(R_k-I)u_p\rangle_{\frak{a}}\vert\\
&\le\Vert(R_k-I)u_q\Vert_{\frak{a}}\Vert(R_k-I)u_p\Vert_{\frak{a}},
\end{align*}
hence
\[
\vert\frak{a}(u_{q,k},u_{p,k}) - \lambda\delta_{qp}\vert \le
K_5(\varepsilon_n+\epsilon_k)^2
\]
for constant $K_5> 0$ which is independent of $n\ge N$.
Similarly,
\begin{align*}
\langle P_{n}u_{q,k},u_{p,k}\rangle&=\langle P_{n}(R_k-I)u_q,(R_k-I)u_p\rangle + \langle (R_k-I)u_q,(P_{n}-I)u_p\rangle\\
&\quad+\langle (R_k-I)u_q,u_p\rangle+\langle (P_{n}-I)u_q,(R_k-I)u_p\rangle\\
&\quad+\langle u_q,(R_k-I)u_p\rangle
+\langle(P_{n}-I)u_q,u_p\rangle+\langle u_q,u_p\rangle,
\end{align*}
hence
\[
\vert i\langle P_{n}u_{q,k},u_{p,k}\rangle - i\delta_{qp}\vert \le K_6(\varepsilon_n+\epsilon_k)^2,
\]
for constant $K_6> 0$ which is independent of $n\ge N$. Furthermore,
\begin{align*}
[M_{n,k}]_{p,q}&=\langle(R_k-I)u_{q},(R_k-I)u_{p} \rangle + \langle(R_k-I)u_{q},u_{p} \rangle + \langle u_{q},(R_k-I)u_{p} \rangle\\
&\quad + \delta_{pq},
\end{align*}
hence for constants $K_7,K_8> 0$ both independent of $n\ge N$, we have
\[
\vert [M_{n,k}]_{pq} - \delta_{pq}\vert \le K_7(\varepsilon_n + \epsilon_k)^2\quad\Rightarrow\quad\vert [M_{n,k}]^{-1}_{pq} - \delta_{pq}\vert \le K_8(\varepsilon_n + \epsilon_k)^2.
\]
Then 
\[
\vert[L_{n,k}M_{n,k}^{-1}]_{p,q}-(\lambda+i)\delta_{p,q}\vert \le K_{9}(\varepsilon_n+\epsilon_k)^2
\]
for constant $K_{9}> 0$ which is independent of $n\ge N$. The result follows from the Gershgorin circle theorem.
\end{proof}

\begin{remark}
We note that, for any orthogonal projection $P$, all non-real eigenvalues of  $A+iP$ can provide information about $\sigma(A)$. Indeed, whenever $(A+iP - z)u=0$ with $u\ne0$, we have
\begin{equation}\label{enclosure1}
(A- \Re z)u = i\Im zu - iPu,\quad\langle Au,u\rangle = \Re z\Vert u\Vert^2,\quad\Vert Pu\Vert^2=\Im z\Vert u\Vert^2
\end{equation}
and using the first and third terms from \eqref{enclosure1} yields
\begin{equation}\label{enclosure2}
\Vert(A-\Re z)u\Vert^2 = (\Im z)^2\Vert u\Vert^2 + (1-2\Im z)\Vert Pu\Vert^2= \Im z(1-\Im z)\Vert u\Vert^2,
\end{equation}
then
\[
\Big[\Re z-\sqrt{\Im z(1-\Im z)},\Re z+\sqrt{\Im z(1-\Im z)}\Big]\cap\sigma(A)\ne\varnothing.
\]
Further, suppose that $(a',b')\cap\sigma(A)=\lambda$ and $a'<\Re z<b'$. 
Then, using \cite[Lemma 1 \& 2]{kat} with the second term in \eqref{enclosure1} and the equality \eqref{enclosure2}, we obtain the enclosure
\[
\lambda\in\left(\Re z - \frac{\Im z(1-\Im z)}{b'-\Re z},\Re z + \frac{\Im z(1-\Im z)}{\Re z - a'}\right).
\]

\end{remark}

Let us now verify our main results with a illustrative example.

\begin{example}\label{uber2}With $\mathcal{H}=\big[L^2\big((0,1)\big)\big]^2$ we consider the following block-operator matrix
\begin{displaymath}
A_0=\small{\left(
\begin{array}{cc}
-\frac{d^2}{dx^2} & -\frac{d}{dx}\\
~&~\\
\frac{d}{dx} & 2I
\end{array} \right),\textrm{~~}\Dom(A_0)=H^2\big((0,1)\big)\cap H^1_0\big((0,1)\big)\times H^1\big((0,1)\big).}
\end{displaymath}
$A_0$ is essentially self-adjoint with closure $A$. We have $\sigma_{\ess}(A)=\{1\}$ (see for example \cite[Example 2.4.11]{Tretter}) while
$\sigma_{\dis}(A)$ consists of the simple eigenvalue $2$ with eigenvector $(0,1)^T$, and the two sequences of simple eigenvalues
\[
\lambda_k^\pm := \frac{2+k^2\pi^2 \pm\sqrt{(k^2\pi^2 + 2)^2  - 4k^2\pi^2}}{2}.
\]
The sequence $\lambda_k^-$ lies below, and accumulates at, the essential spectrum. The sequence $\lambda_k^+$ lies above the eigenvalue $2$ and accumulates at $\infty$.

Let $\cL_h^0$ be the FEM space of piecewise linear functions on $[0,1]$ with a uniform mesh of size $h$ and which satisfy homogeneous Dirichlet boundary conditions. Let $\cL_h$ be the space without boundary conditions.  First we apply the Galerkin method directly to $A$ with trial spaces $L_{h}=\cL_h^0\oplus\cL_h$. We find that spectral pollution occurs in the interval $(1,2)\subset\rho(A)$ which obscures the approximation of the genuine eigenvalue $2$; see the left-hand side of Figure 6. 
\begin{figure}[h]
\centering
\includegraphics[scale=.3]{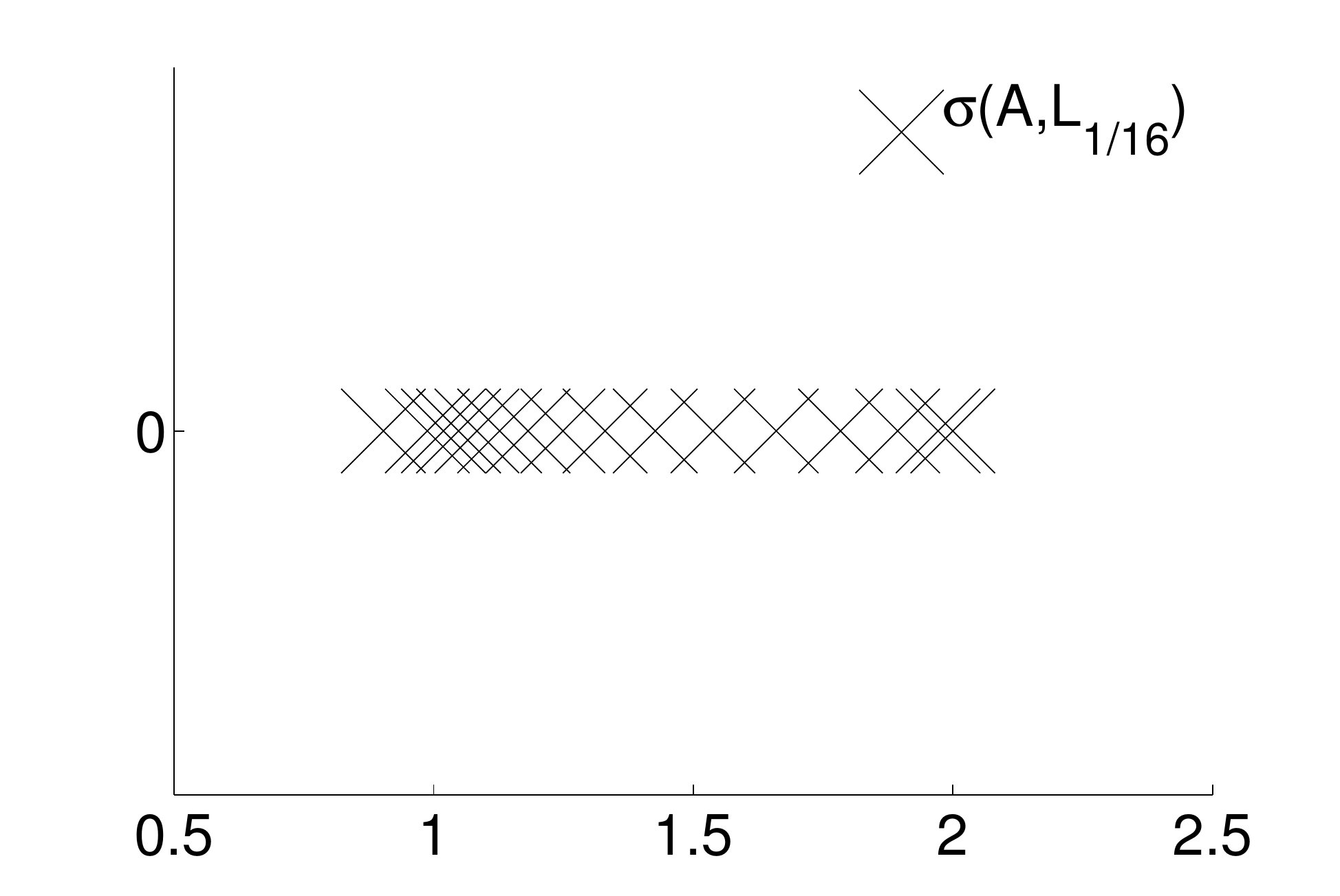}\includegraphics[scale=.3]{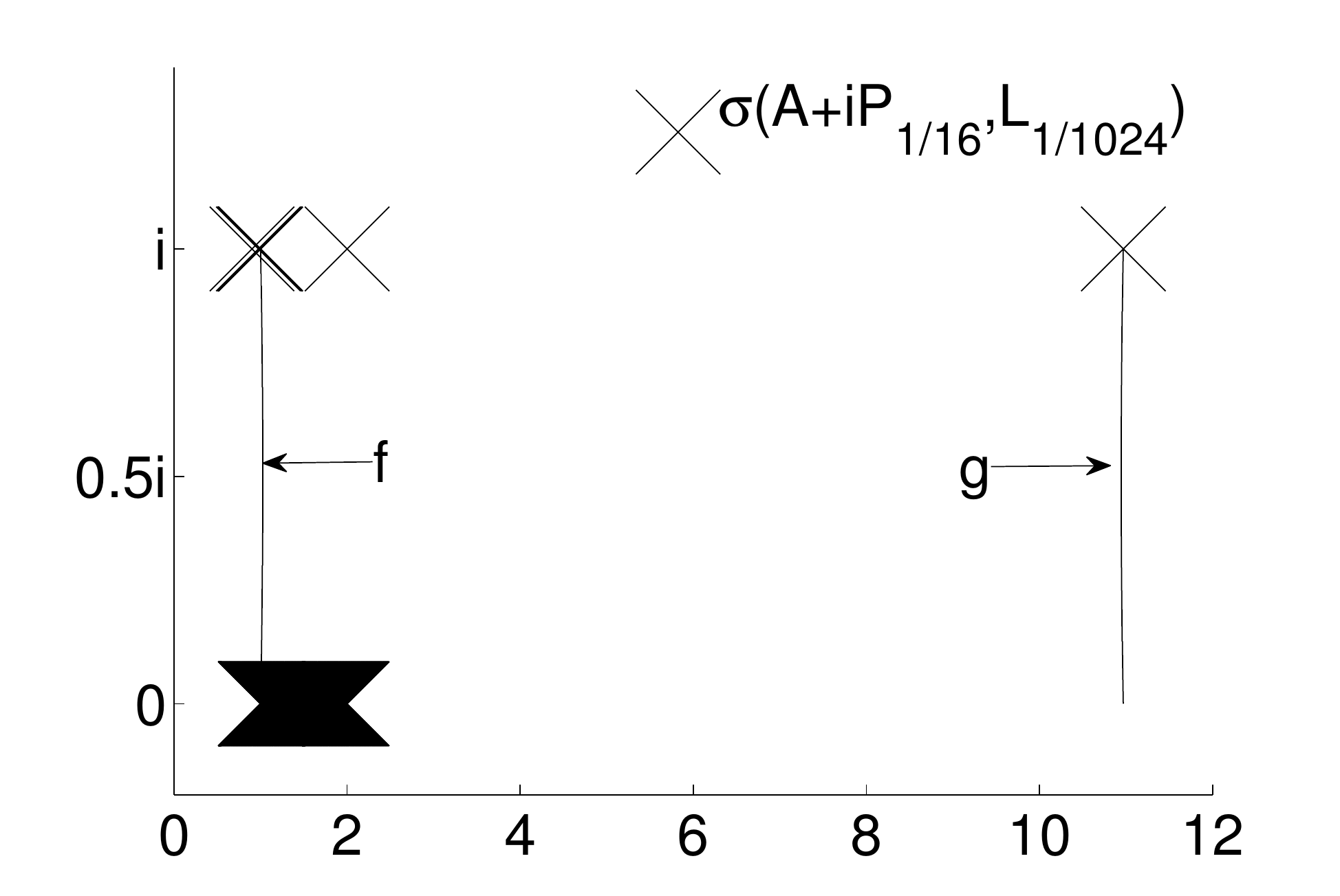}
\caption{On the left-hand side, Galerkin method approximation for $\sigma(A)$ from Example \ref{uber2}, spectral pollution in the interval $(1,2)$ obscures the approximation of the genuine eigenvalue $2$. On the right-hand side, the Galerkin method approximation for $\sigma(A+iP_{1/16})$ from Example \ref{uber2}, the curves $f$ and $g$, which together form $\Gamma_{1,\lambda_1^+}^{0,1}$. The region $\mathcal{U}_{1,\lambda_1^+}^{0,1}$ consists of complex numbers which lie to the right of $f$ and to the left of $g$.}
\end{figure}
Let $P_{1/16}$ be the orthogonal projection onto the trial space $L_{1/16}$. Since
\[(1,\lambda_1^+)\cap\sigma(A)=\{2\}\in\sigma_{\dis}(A)
\]
and $(0,1)^T\in L_h$ for all $h\in(0,1]$, the hypothesis of Lemma \ref{cor1} is satisfied, hence
\[
\sigma(A+iP_{1/16})\cap\mathcal{U}_{1,\lambda_1^+}^{0,1}= \{2+i\}\in\sigma(A+iP_{1/16}).
\]
Furthermore, by \cite[Theorem 2.5]{me}, we can approximate the eigenvalue $\{2+i\}$, with the Galerkin method, without incurring any spectral pollution, i.e.,
\[
\bigg(\lim_{h\to 0}\sigma(A+iP_{1/16},L_h)\bigg)\cap\bigg(\mathcal{U}_{1,\lambda_1^+}^{0,1}\big\backslash\mathbb{R}\bigg)=\{2+i\}.
\]
The right-hand side of Figure 6 shows the Galerkin method approximation of $\sigma(A+iP_{1/16})$ with the trial space $L_{1/1024}$. We see that $2+i\in\sigma(A+iP_{1/16},L_{1/1024})$ and the only elements from
\[
\bigg(\sigma(A+iP_{1/16},L_{1/2048})\cap\mathcal{U}_{1,\lambda_1^+}^{0,1}\bigg)\big\backslash\{2+i\}
\]
are very close to the real line which is where spectral pollution is still permitted. The perturbation method has demonstrated that the Galerkin eigenvalues in the interval $(1,2)$ are all spurious. Furthermore, the genuine eigenvalue 2 is approximated by the perturbation method without being obscured by pollution.
\void{
\begin{figure}[h]
\centering
\includegraphics[scale=.5]{nov2014pic12-eps-converted-to.pdf}
\caption{The Galerkin eigenvalues for the operator $A+iP_{1/16}$ and trial space $L_{1/2048}$. Also shown are the curves $f$ and $g$, which together form $\Gamma_{1,\lambda_1^+}^{0,1}$. The region $\mathcal{U}_{1,\lambda_1^+}^{0,1}$ consists of those complex numbers which lie to the right of $f$ and to the left of $g$.}
\end{figure}
\begin{figure}[h]
\centering
\includegraphics[scale=.5]{nov2014pic14.eps}
\caption{The Galerkin eigenvalues near $1+i$ all lie to the left of the curve $f$ and are, therefore, not contained in $\mathcal{U}_{1,\lambda_1^+}^{0,1}$.}
\end{figure}
}

Next we approximate the eigenvalue $\lambda_1^+$. Applying the Galerkin method directly to $A$ we do not incur spectral pollution near $\lambda_1^+$ and consequently we have the standard superconvergence result:
\begin{equation}\label{spcon}
\dist\big(\lambda_1^+,\sigma(A,L_h)\big)=\mathcal{O}(\delta_{\frak{a}}(\mathcal{L}(\{\lambda_1^+\}),\mathcal{L}_h)^2)=\mathcal{O}(h^2).
\end{equation}
By Theorem \ref{QQ} we have
\begin{equation}\label{spspcon}
\dist\big(\lambda_1^++i,\sigma(A+iP_h)\big)=\mathcal{O}(\delta(\mathcal{L}(\{\lambda_1^+\}),\mathcal{L}_h)^2)=\mathcal{O}(h^4).
\end{equation}
The second column in Table 1 shows the distance of $\lambda_1^+$ to $\sigma(A,L_h)$, the third column shows the distance of $\lambda^+_1+i$ to a Galerkin approximation (with very refined mesh) of the eigenvalue of $A+iP_h$ which is close to $\lambda_1^++i$. The left-hand side of Figure 7 displays a loglog plot of the data in Table 1, and verifies both \eqref{spcon} and \eqref{spspcon}.
\begin{table}[h!]
\begin{tabular}{c|c|c}
h & $\dist\big(\lambda_1^+,\sigma(A,L_h)\big)$ & $\dist\big(\lambda_1^++i,\sigma(A+iP_h,L_{h\times 2^{-7}})\big)$\\
\hline
1/2 &   \hspace{30pt}1.861045647858232\hspace{30pt}  &  0.014440864705963\\
1/4 &  0.458746253205135  &  0.000609676693732\\
1/8 &  0.113442149493080  &  0.000034835584324\\
1/16 & 0.028273751580725 &  \hspace{3pt}0.000002688221958
\vspace{10pt}
\end{tabular}\caption{Approximation of $\lambda_1^+$ from $\sigma(A,L_h)$ and from an approximation of $\sigma(A+ iP_h)$.}
\end{table}
\begin{figure}\label{puper}
\centering
\includegraphics[scale=.3]{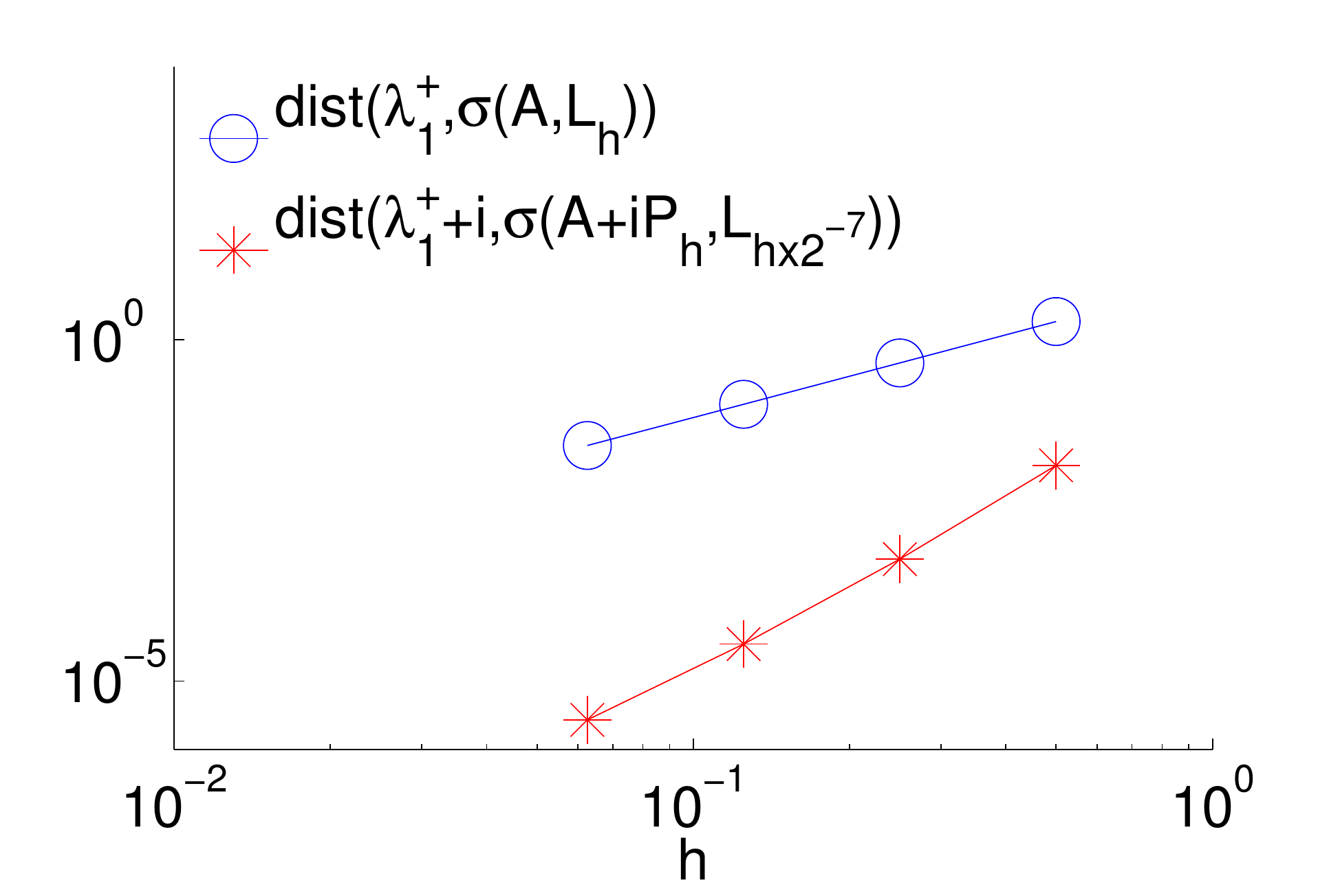}\includegraphics[scale=.3]{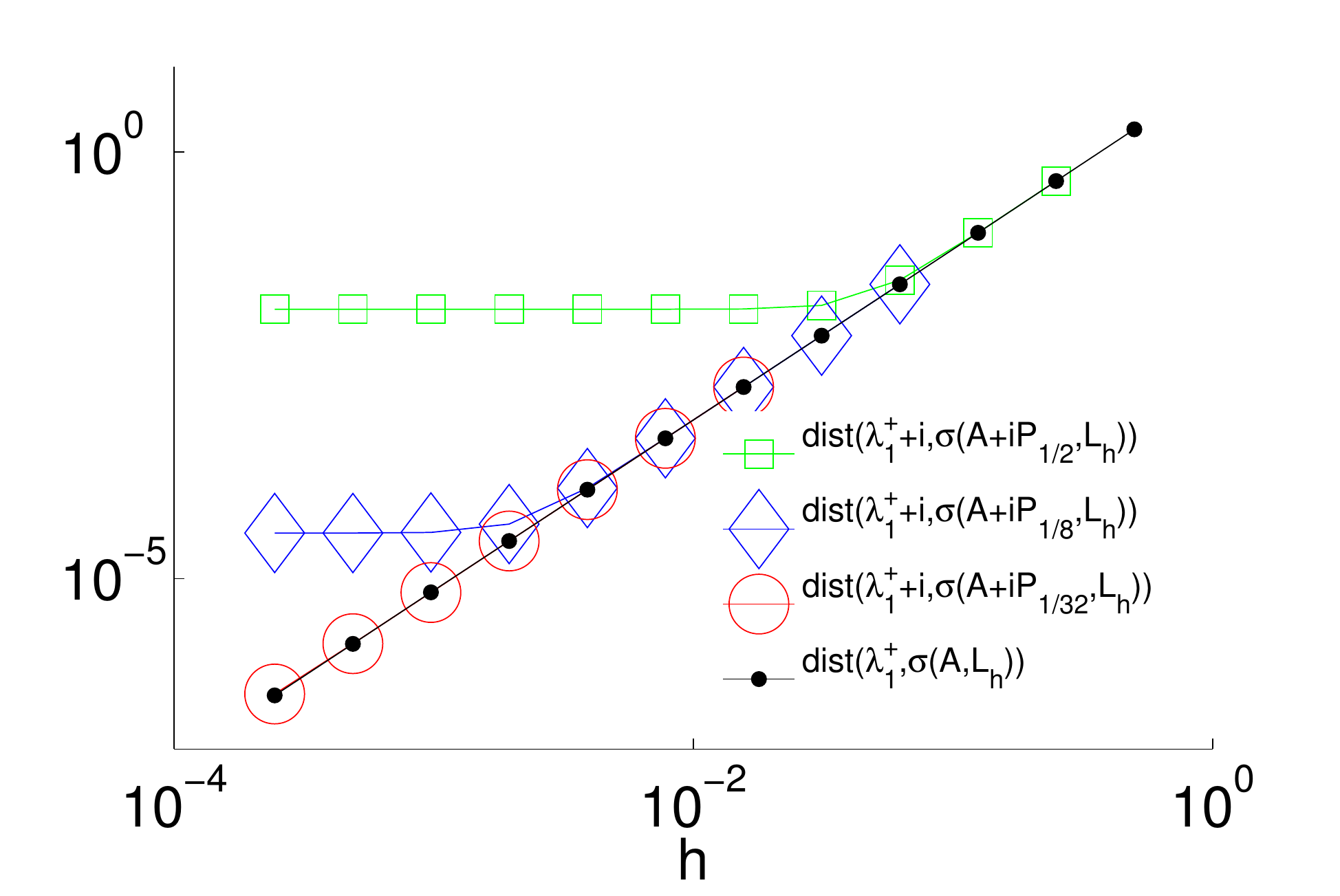}
\caption{On the left-hand side, approximation of $\lambda_1^+$ with $\sigma(A,L_h)$ and with an approximation of $\sigma(A+ iP_h)$. The gradients blue and red lines are approximately $2$ and $4$, respectively. On the right-hand side, approximation of $\lambda_1^++i$ and $\lambda_1^+$ using the perturbation and Galerkin methods, respectively}
\end{figure}

We now compare the approximation of $\lambda_1^+$ by applying the Galerkin method directly to $A$ and to $A+iP_h$. The results are displayed on the right-hand side of  Figure 7; we see that the approximation and convergence achieved by the perturbation method are essentially the same as those achieved by the Galerkin method. It is clear and consistent with Theorem \ref{eigconv} that we need not be concerned with locking-in poor accuracy with a relatively low dimensional projection $P_{h}$. In fact, it is quite remarkable that the approximation with $\sigma(A+iP_{1/32},L_{1/32\times 2^7})$ is essentially the same as $\sigma(A,L_{1/32\times 2^7})$.
\void{
\begin{figure}[h]
\centering
\includegraphics[scale=.5]{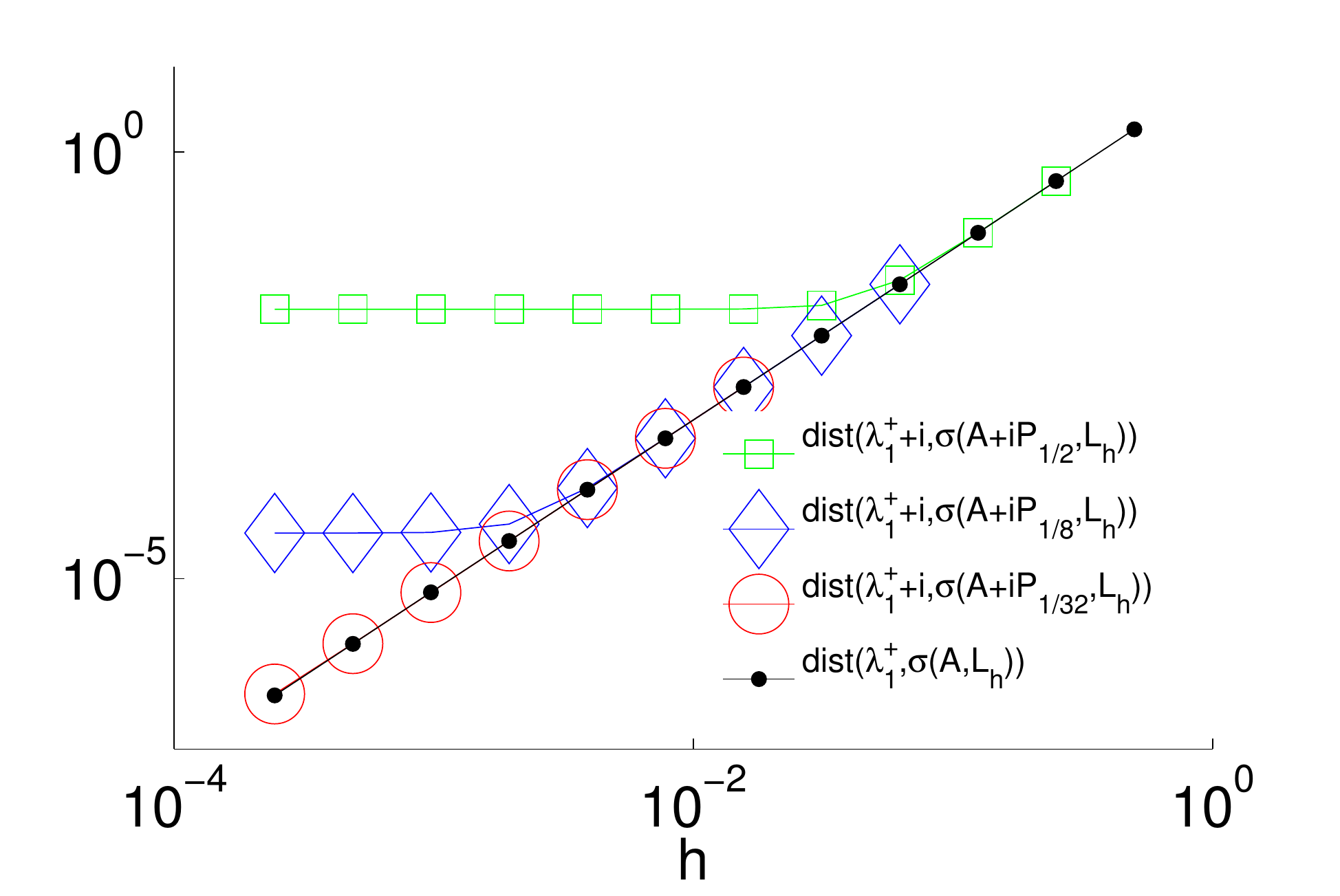}\caption{Approximation of $\lambda_1^++i$ and $\lambda_1^+$ using the perturbation and Galerkin methods, respectively.}
\end{figure}
}

\void{With $\mathcal{H}=[L^2((0,1),dx]^2$ we consider the following block-operator matrix
\begin{displaymath}
T_0=\small{\left(
\begin{array}{cc}
-d^2/dx^2 & -d/dx\\
d/dx & 2I
\end{array} \right)}
\end{displaymath}
with homogeneous Dirichlet boundary conditions in the first component. We have $\sigma_{\ess}(T)=\{1\}$ (see for example \cite[Example 2.4.11]{Tretter}) while
$\sigma_{\dis}(T)$ consists of the simple eigenvalue $\{2\}$ with eigenvector $(0,1)^T$, and the two sequences of simple eigenvalues
\[
\lambda_k^\pm := \frac{2+k^2\pi^2 \pm\sqrt{(k^2\pi^2 + 2)^2  - 4k^2\pi^2}}{2}.
\]
The sequence $\lambda_k^-$ lies below, and accumulates at, the essential spectrum. While the sequence $\lambda_k^+$ lies above the eigenvalue $2$ and accumulates at $\infty$.

Denote by $L_h^0$ the FEM space of piecewise linear functions on a uniform mesh of size $h$ and satisfying homogeneous Dirichlet boundary conditions, and by $L_h$ the space without boundary conditions. We consider the subspaces $\cL_{h}=L_h^0\oplus L_h$ which belong to the domain of the quadratic form associated to $T$.

\begin{figure}[h!]
\centering
\includegraphics[scale=.4]{twenty.eps}\includegraphics[scale=.4]{twenty2.eps}
\caption{displaying something.}
\end{figure}

\begin{figure}[h!]
\centering
\includegraphics[scale=.4]{twenty3.eps}\includegraphics[scale=.4]{twenty4.eps}
\caption{displaying something.}
\end{figure}

\begin{table}[hhh]
\begin{tabular}{c|c|c}
h & $\lambda_{1,h}^+$ & $\vert\lambda_{1,h}^+-\lambda_1^+-i\vert$\\
\hline
1/5 &   11.0201854242 + 0.99919794297i & 0.05028557407 \\
1/10 &  10.9798272969 + 0.99997146825i & 0.00992109091 \\
1/20 &  10.9721318334 + 0.99999858898i & 0.00222558684 \\
1/40 &  10.9704344517 + 0.99999992082i & 0.00052820471 \\
1/80 &  10.9700349748 + 0.99999999531i & 0.00012872783 \\
1/160 & 10.9699380255 + 0.99999999988i & 0.00003177855 \\
1/320 & 10.9699141428 + 0.99999999838i & 0.00000789582
\end{tabular}
\end{table}
}
\end{example}

\section{Further examples}

\void{\subsection{Perturbed multiplication Operator}
With $\mathcal{H}=L^2(-\pi,\pi)$ consider the bounded operator $T\phi = a(x)\phi$ where
\[
a(x)=\begin{cases}
  -\frac{3}{2} + \frac{1}{2}\cos\sqrt{5}x & \text{for}\quad \pi\le x < 0, \\
  2 + \cos\sqrt{2}x & \text{for}\quad 0\le x<\pi.
  \end{cases}
\]
We also consider the operator $T+K$ where $K$ is the rank $2$ operator
\[
K\psi(x)=
  -\frac{3}{2\pi}e^{-ix}\int_{-\pi}^\pi\psi(y)e^{iy}~dy + \frac{1}{2\pi}e^{2ix}\int_{-\pi}^\pi\psi(y)e^{-2iy}~dy.
\]
We have
\[\sigma_{\ess}(T+K)=\sigma_{\ess}(T)=\sigma(T)=[-2,-1]\cup[1,3].\]
Note that $(-1,1)$ is a gap in the spectrum of $T$, and a gap in the essential spectrum of $T+K$. The latter can have at most $2$ eigenvalues in this gap. With the orthonormal basis $\psi_k(x) = e^{-ikx}$ $k\in\mathbb{Z}$, we apply the Galerkin method to $T$ and $T+K$ with the subspaces $\cL_{2n+1}=\spn(\psi_{-n}(x),\dots,-\psi_0(x),\dots,\psi_n(x))$ $n\in\mathbb{N}$. The operators $T$ and $T+K$ are considered in \cite[Section 3]{lesh}, using the method of \emph{second order relative spectra} two eigenvalues of $T+K$ were approximated $\lambda_1\approx -0.8947$ and $\lambda_2\approx 0.6082$ both in the gap $(-1,1)$. We shall approximate these eigenvalues with our perturbation method.

The left hand side of Figure \ref{f1} show the Galerkin eigenvalues for $T$ and $T+K$. In both cases the essential spectrum is well approximated, but in both cases there are several \emph{spurious} Galerkin eigenvalues in the interval $(-1,1)$. On the right we have the perturbation method applied to $T$ where $Q_n=E_n((-1,1))P_n$. Since $(-1,1)\subset\rho(T)$ we should not have any elements from $\sigma(T,\mathcal{L}_{2n+1})$ with imaginary part close to 1 unless those points have real part very close the circles

\begin{figure}[h!]\label{f1}
\centering
\includegraphics[scale=.4]{five-eps-converted-to.pdf}\includegraphics[scale=.4]{six-eps-converted-to.pdf}
\caption{On the left we have the Galerkin eigenvalues for $T$ and $T+K$. On the right we have the perturbation method applied to $T$ where $Q_n=E_n((-1,1))P_n$.}
\end{figure}

\begin{figure}[h!]
\centering
\includegraphics[scale=.4]{seven.eps}\includegraphics[scale=.4]{two.eps}
\caption{On the left we have the perturbation method applied to $T$ where $Q_n=E_n((-1,1))P_n$ also shown are the circles $\Gamma_{-1}$ and $\Gamma_1$. On the right we have the perturbation method applied to $T+K$ where $Q_n=E_n((-1,1))P_n$.}
\end{figure}

\begin{figure}[h!]
\centering
\includegraphics[scale=.4]{three.eps}\includegraphics[scale=.4]{eight.eps}
\caption{displaying something.}
\end{figure}
}

\begin{example}\label{magneto}
With $\mathcal{H}=\big[L^2((0,1),\rho_0dx\big)\big]^3$ we consider the magnetohydrodynamics operator
\begin{displaymath}
A=\small{\left(\begin{array}{ccc}
        -\frac{d}{dx}(\upsilon_a^2 + \upsilon_s^2)\frac{d}{dx} + k^2\upsilon_a^2 & -i(\frac{d}{dx}(\upsilon_a^2 + \upsilon_s^2) -1)k_\perp & -i(\frac{d}{dx}\upsilon_s^2 -1)k_\parallel\\
        ~&~&~\\
	-ik_\perp((\upsilon_a^2 + \upsilon_s^2)\frac{d}{dx} +1) & k^2\upsilon_a^2 + k_\perp^2\upsilon_s^2 & k_\perp k_\parallel\upsilon_s^2\\
	~&~&~\\
	-ik_\parallel(\upsilon_s^2\frac{d}{dx} +1) & k_\perp k_\parallel\upsilon_s^2 & k_\parallel^2\upsilon_s^2
      \end{array}\right)}.
\end{displaymath}
With $\rho_0=k_\perp=k_\parallel=g=1$, $\upsilon_{a}(x)=\sqrt{7/8 - x/2}$ and $\upsilon_s(x)=\sqrt{1/8+x/2}$, we have 
\[\sigma_{\ess}(A)=[7/64,1/4]\cup[3/8,7/8].
\]
The discrete spectrum contains a sequence of simple eigenvalues which accumilate only at $\infty$. These eigenvalues are above, and not close to,  the essential spectrum. They are approximated by the Galerkin method, with trial spaces $L_h=\cL_{h}^0\oplus\cL_h\oplus\cL_h$, without incurring spectral pollution. It was shown, using the second order relative spectrum, that there is also an eigenvalue $\lambda_1\approx 0.279$ in the gap in the essential spectrum; see \cite[Example 2.7]{me2}. The top row of Figure 8 shows many Galerkin eigenvalues in the gap in the essential spectrum and many more just above the essential spectrum; we should be suspicious of spectral pollution in these regions. We define
\[
\tau(A+iP_{h_0},L_h):=\big\{\Re z + (1-\Im z)i:z\in\sigma(A+iP_{h_0},L_h)\big\}
\]
and we are therefore interested in those elements from $\tau(A+iP_{h_0},L_h)$ which are close to the real line, i.e., we would prefer our approximate eigenvalues to converge to $\sigma(A)$ rather than $\sigma(A)+i$. The second row of Figure 8 shows $\tau(A+iP_{1/64},L_{1/1024})$, the two bands of essential spectrum are clearly approximated along with an approximation of $\lambda_1$ in the gap, and a second eigenvalue above the essential spectrum. The perturbation method has approximated the essential spectrum, identified the spectral pollution, and approximated two eigenvalues which were obscured by the spectral pollution. 
\begin{figure}[h!]
\centering
\includegraphics[scale=.3]{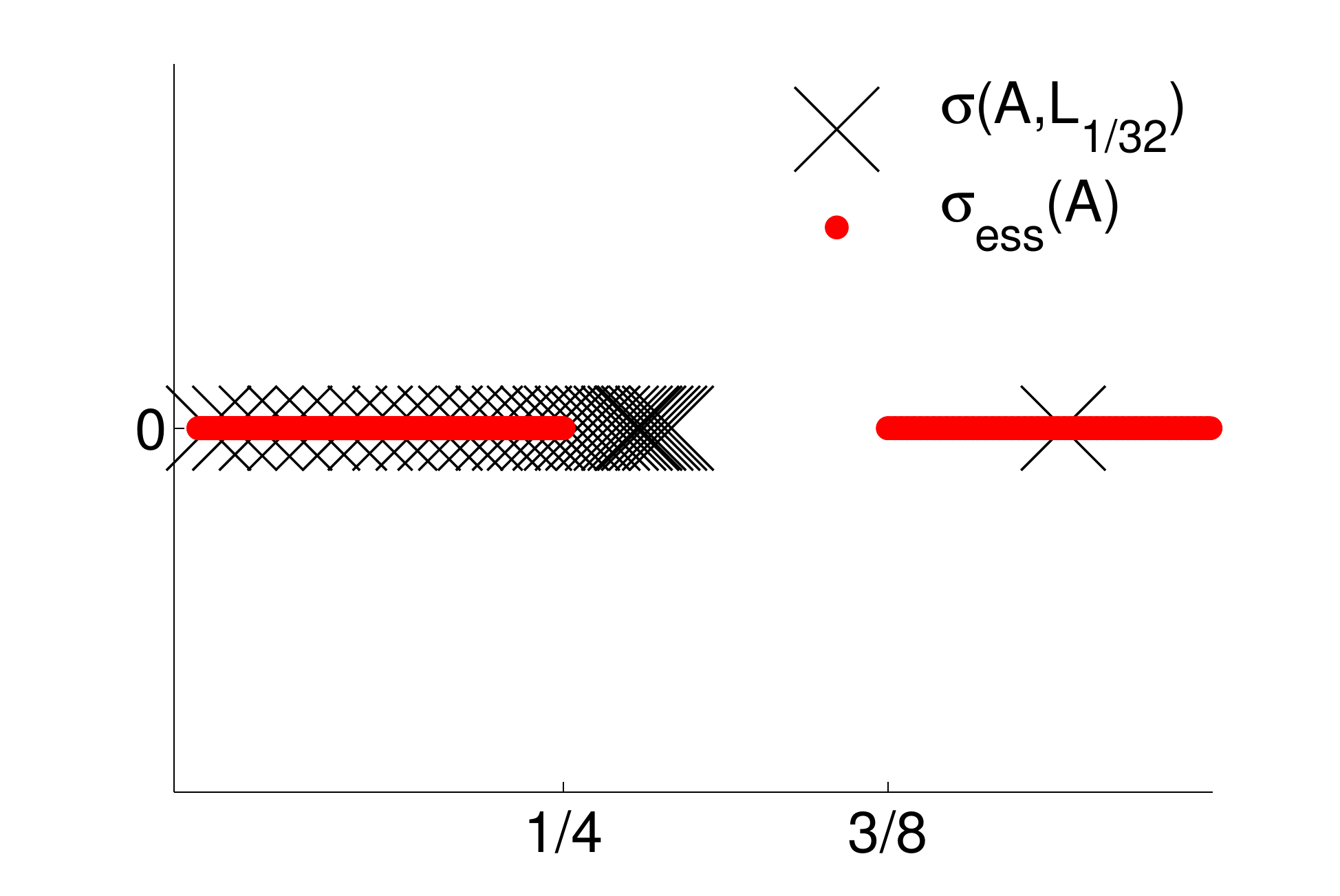}\includegraphics[scale=.3]{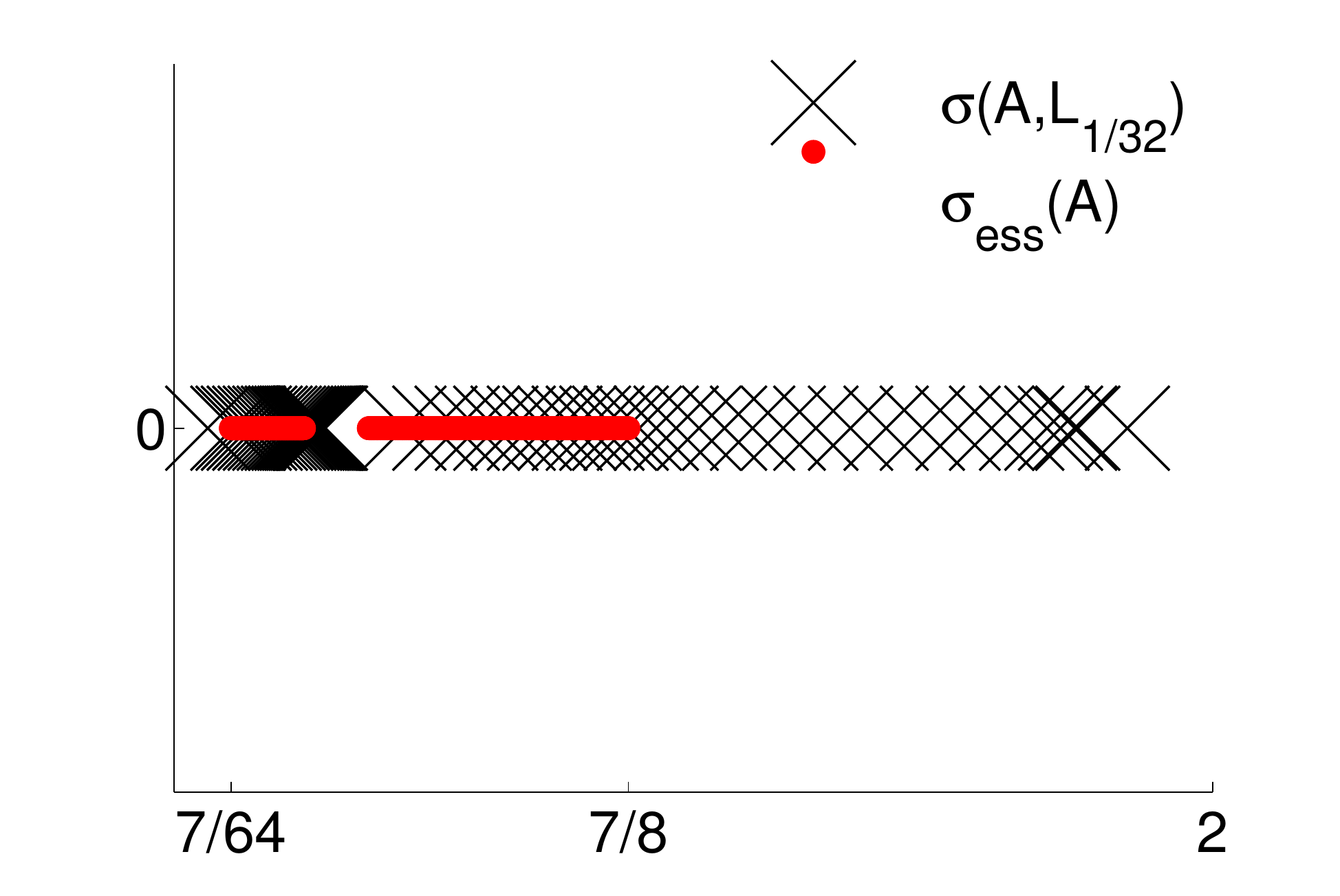}
\includegraphics[scale=.3]{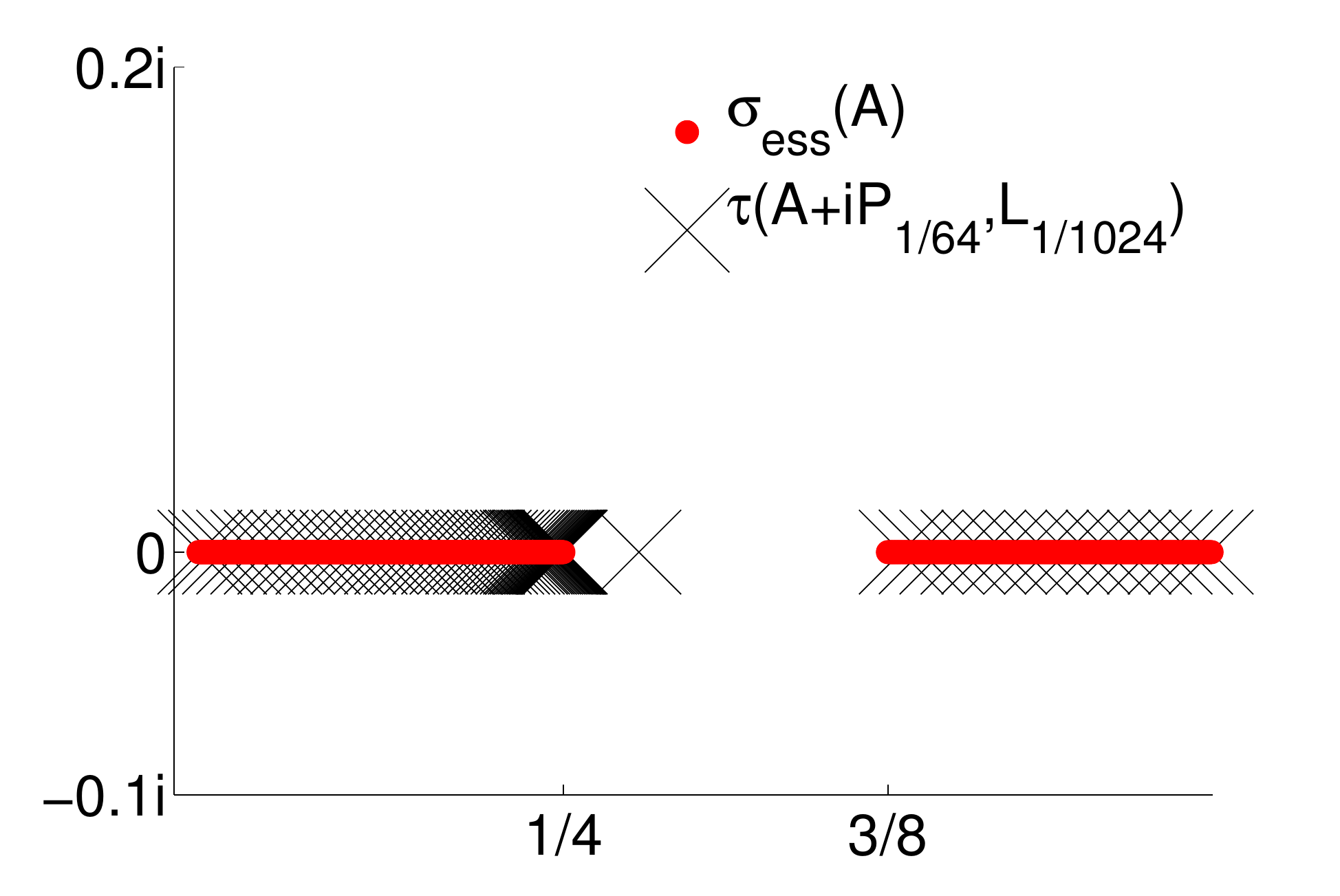}\includegraphics[scale=.3]{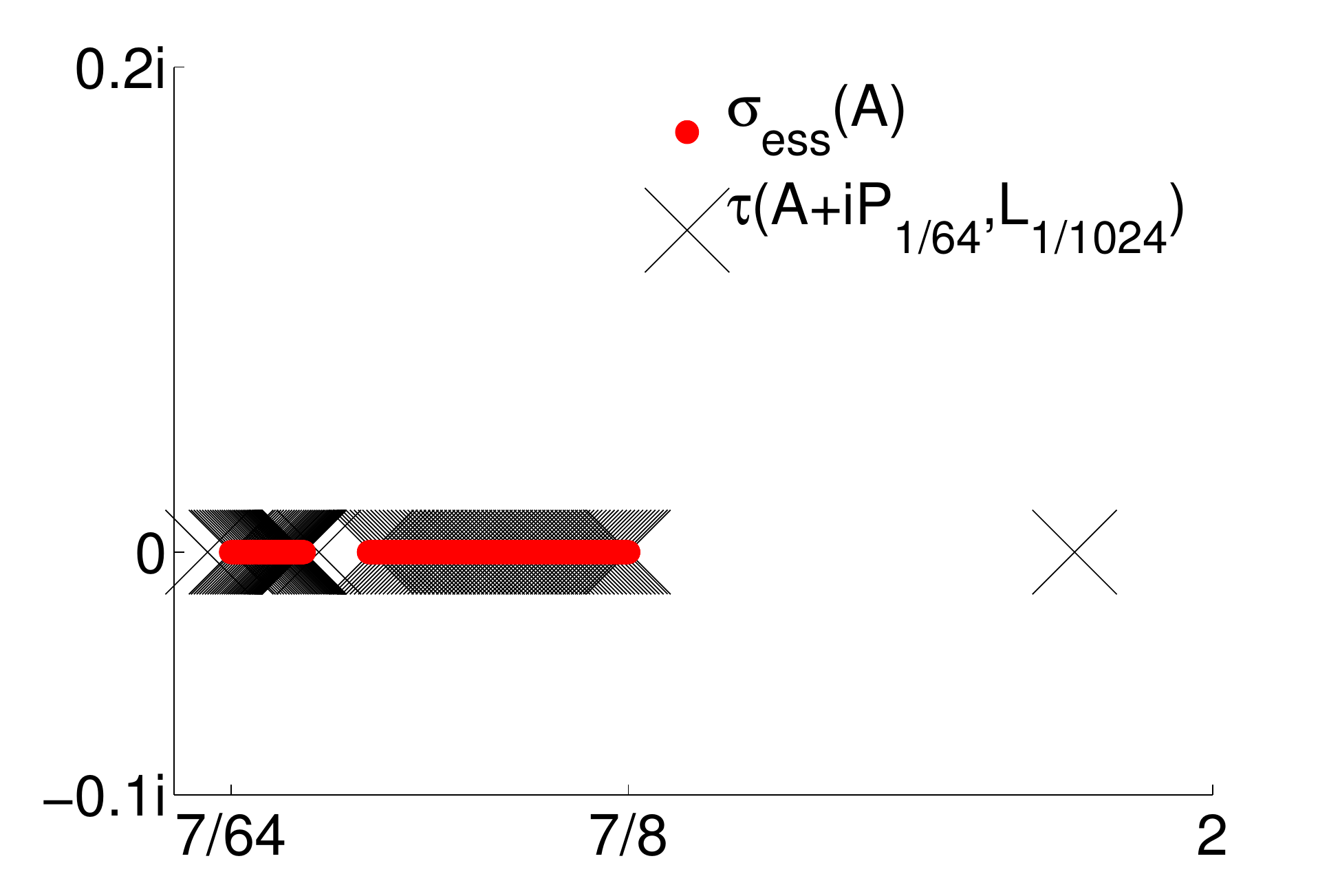}
\caption{On the top row, we see the Galerkin method approximation for $\sigma(A)$ from Example \ref{magneto}. There are many Galerkin eigenvalues in the gap in the essential spectrum and many more just above the essential spectrum; in these regions we should be suspicious of spectral pollution. The second row shows the perturbation method approximation for $\sigma(A)$ from Example \ref{magneto}; the essential spectrum is approximated, as are two eigenvalues, one in the gap and one just above the essential spectrum. The perturbation method has identified the spectral pollution in the gap and above the essential spectrum.}
\end{figure}

\end{example}

\begin{example}\label{schro1}
With $\mathcal{H}=L^2(\mathbb{R})$ we consider the Schr\"odinger operator
\[
Au = -u'' + \Big(\cos x - e^{-x^2}\Big)u.
\]
The essential spectrum of $A$ has a band structure. The first three intervals of essential spectrum are approximately 
\[
[-0.37849,-0.34767],\quad [0.5948,0.918058]\quad\textrm{and}\quad [1.29317,2.28516].
\]
The second order relative spectrum has been applied to this operator, see \cite{bole}, where  the following approximate eigenvalues were identified 
\[\lambda_1\approx -0.40961,\quad\lambda_2\approx 0.37763,\quad\textrm{and}\quad\lambda_3\approx 1.18216.\]
We note that $\lambda_1$ is below the essential spectrum,
$\lambda_2$ is in the first gap in the essential spectrum, and $\lambda_3$
is in the second gap. We apply the perturbation method with the trial spaces $\mathcal{L}_{(X,Y)}$ which is a $Y$-dimensional space of piecewise linear trial functions on the the interval $[-X,X]$ which vanish at the boundary, and $P_{(X,Y)}$ is the orthogonal projection onto $\mathcal{L}_{(X,Y)}$. The left-hand side of Figure 9 shows the perturbation method has clearly identified the first two bands of essential spectrum and the eigenvalues $\lambda_1$ below the essential spectrum, $\lambda_2$ in the first gap, and $\lambda_3$ in the second gap. 
\begin{figure}[h!]
\centering
\includegraphics[scale=.3]{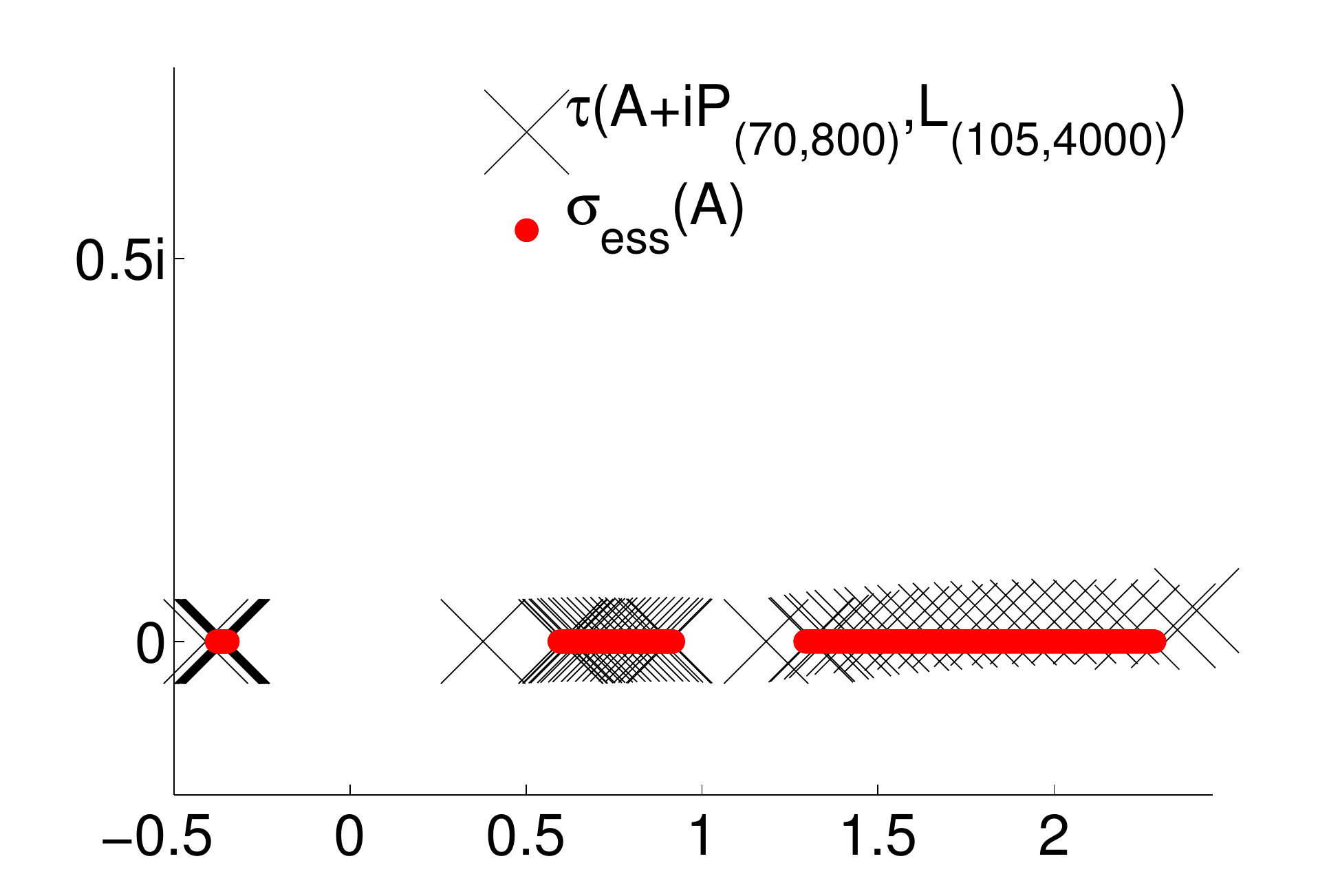}\includegraphics[scale=.3]{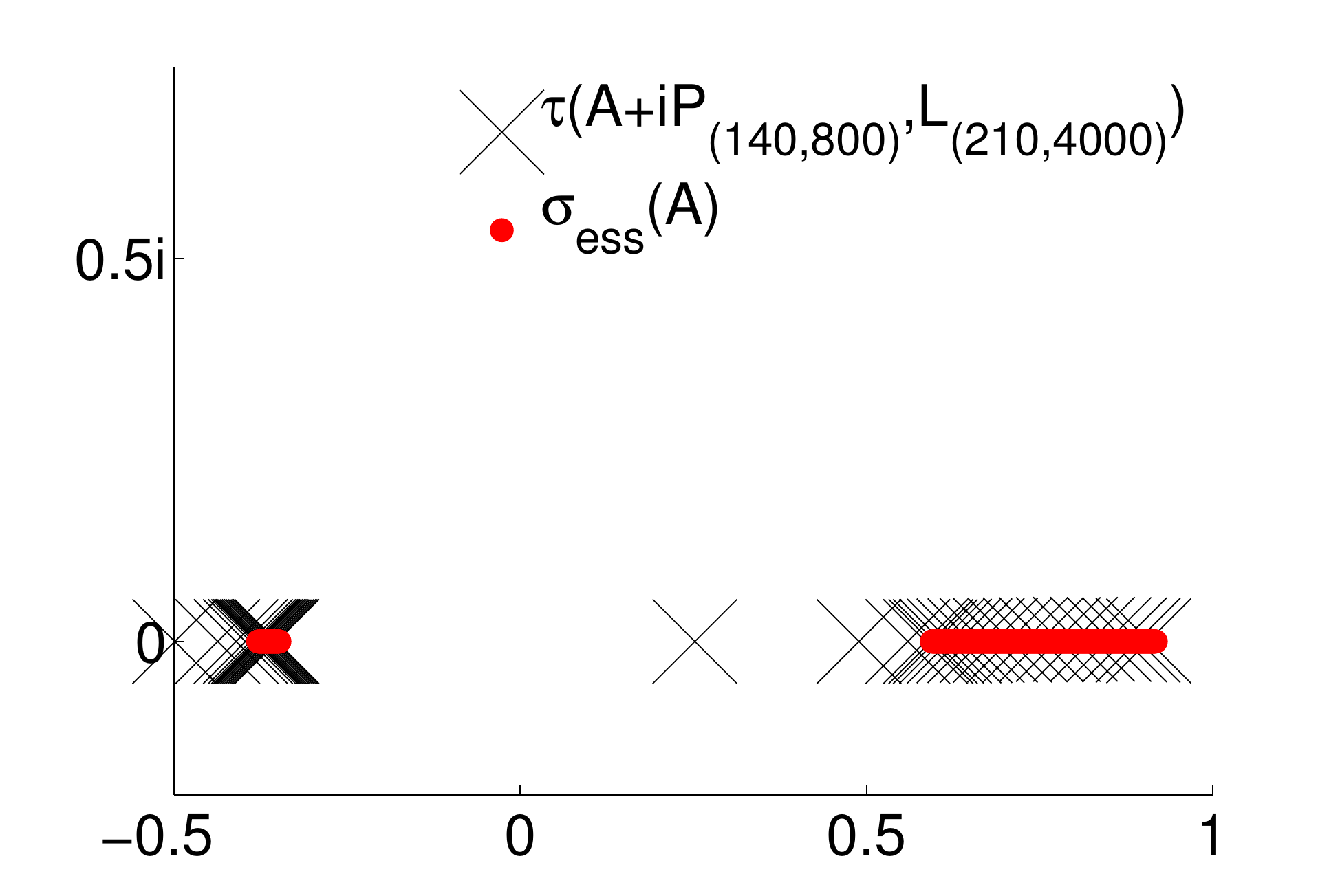}
\caption{The left-hand side shows the perturbation method approximation approximation for $\sigma(A)$ from Example \ref{schro1}. The first two bands of essential spectrum are approximated, as are the eigenvalues $\lambda_1$, $\lambda_2$ and $\lambda_3$. The right-hand side shows the perturbation method approximation approximation for $\sigma(A)$ from Example \ref{schro2}. The first two bands of essential spectrum are approximated, as are the first three eigenvalues in the first gap in the essential spectrum.}
\end{figure}

\end{example}


\begin{example}\label{schro2}
With $\mathcal{H}=L^2\big((0,\infty)\big)$ we consider the Schr\"odinger operator
\[
Au = -u'' + \left(\sin x - \frac{40}{1+x^2}\right)u,\quad u(0)=0.
\]
This example has been also been considered in \cite{mar2}. The first three bands of essential spectrum are the same as in the previous example. However, this time there are infinitely many eigenvalues in the gaps which accumulate at the lower end point of the bands with their spacing becoming exponentially small; see \cite{sch}. We apply the perturbation method with the trial spaces $\mathcal{L}_{(X,Y)}$ which is a $Y$-dimensional space of piecewise linear trial functions on the interval $[0,X]$ which vanish at the boundary. The operator $P_{(X,Y)}$ is the orthogonal projection onto trial space $\mathcal{L}_{(X,Y)}$. The right-hand side of Figure 9 shows that the perturbation method has approximated three eigenvalues in the first gap of the essential spectrum.
\end{example}

We should stress the ease with which the above calculations are conducted. The perturbation method does not require trial spaces from the operator domain, thus we have been able to use the FEM spaces of piecewise linear trial functions. The quadratic methods cannot be applied with these trial spaces. Our final example is outside much of the theory so far developed for the perturbation method, this is because the operator concerned is indefinite. However, the numerical results suggest that the perturbation method can be extended to the indefinite case. The second order relative spectra has been applied to this example and the code made available online; see \cite{bb} and \cite{nlevp}, respectively. We use this code to apply the Galerkin method, the perturbation method, the second order relative spectrum, the Davies \& Plum method and Zimmermann \& Mertins method.
\begin{example}\label{diracex}
With $\mathcal{H}=\big[L^2\big((0,\infty)\big)\big]^2$ we consider the Dirac operator
\begin{displaymath}
A=\left(\begin{array}{cc}
        I - \frac{1}{2x}~&~ -\frac{d}{dx}-\frac{1}{x}\\\\
	\frac{d}{dx} -\frac{1}{x}~&~-I - \frac{1}{2x}
      \end{array}\right).
\end{displaymath}
We have $\sigma_{\ess}(A)= (-\infty,-1]\cup[1,\infty)$ and the interval $(-1,1)$ contains the eigenvalues
\[
\sigma_{\dis}(A)= \left(1+\frac{1}{4(j-1+\sqrt{3/4})^2}\right)^{-1/2}\quad j=1,2,\dots.
\]
We use trial spaces generated by Hermite functions of odd order; see \cite{bb} for further details. There is no spectral pollution incurred by the Galerkin method in this example, therefore we can also compare the perturbation method with the Galerkin method. The second order relative spectrum is known to converge to the discrete spectrum; see \cite{bo}. From this method we obtain a sequence of complex numbers with $z_n\to\lambda_1$, where $n$ is the dimension of the trial space. The sequence of real parts $(\Re z_n)$ we take as our approximation for $\lambda_1$. The Davies \& Plum and Zimmermann \& Mertins methods, which are equivalent, provide a sequence of intervals containing $\lambda_1$ we take the mid-point of these intervals, which we denote by $w_n$, to be our approximation of $\lambda_1$. Our numerical results suggest the following convergence rates
\begin{align*}
&\dist\big(\lambda_1 +i,\sigma(A+iP_{n/2},P_n)\big)=\mathcal{O}(n^{-0.9}),~
\dist\big(\lambda_1,\sigma(A,P_n)\big)=\mathcal{O}(n^{-0.9}),\\
&\vert\lambda_1 - z_n\vert=\mathcal{O}(n^{-0.2}),~\vert\lambda_1 - \Re z_n\vert=\mathcal{O}(n^{-0.7})\textrm{~and~}\vert\lambda_1 - w_n\vert=\mathcal{O}(n^{-0.2}).
\end{align*}
Again we see the performance of the perturbation method is essentially the same as the Galerkin method. We also note the relatively poor performance of the quadratic methods. The latter is not entirely surprising as the known convergence rates for quadratic methods are measured in terms of $\delta_{A}(\mathcal{L}(\{\lambda\}),\mathcal{L}_n)$, i.e., the distance of the eigenspace to the trial space with respect to the graph norm; see \cite[Lemma 2]{bost1} and \cite[Section 6]{me3}.
  
\end{example}
\section{Conclusions and further research}
Our theoretical results are, for the most part, focused on the perturbation and approximation of the discrete spectrum. However, the examples indicate that our new perturbation method also captures the essential spectrum. This should be further investigated. For the approximation of eigenvalues, the rapid convergence assured by theorems \ref{QQ} \& \ref{eigconv} mean that, in terms of accuracy and convergence, we can expect the perturbation method to significantly outperform the quadratic methods. The fact that the former may be applied with trial spaces from the form domain is another significant advantage. Recently a second pollution-free and non-quadratic technique has emerged; see \cite{me4}. Currently the latter has the disadvantage of requiring \'a priori information about gaps in the essential spectrum, however, it does have the advantage of a self-adjoint algorithm. In terms of accuracy and convergence, there appears to be little separating these two non-quadratic techniques; see \cite[examples 5.2 \& 5.3]{me4}. Which technique is preferable will likely depend on the particular situation and availability of \'a priori information; this should be the subject of further study.

\section{Acknowledgements}
The authors are grateful to Marco Marletta and Eugene Shargorodsky for many useful discussions. Michael Strauss is also grateful for the support of the Wales Institute of Mathematical and Computational Sciences and the Leverhulme Trust grant: RPG-167.

\void{
\appendix
\section{Second order relative spectra and convergence}
The method of second order relative spectra has been extensively studied over the past 15 years. Interest in this method has been stimulated by the fact that it provides intervals which intersect the spectrum. With some \'a priori information the method can also provide enclosures for eigenvalues. The technique is known to converge to the discrete spectrum; see \cite{bo}. It was also thought, by many, to converge to the essential spectrum, however, this has recently been shown to be false, in general; see \cite{shar2}. For the discrete spectrum, we briefly examine the quality of the approximation and of the enclosures provided by this method. We also provide a new proof, which is based on classical spectral approximation theory, of the convergence rate to elements from $\sigma_{\dis}(A)$. 
\begin{definition}
Let $A$ be a self-adjoint operator acting on a Hilbert space $\mathcal{H}$. The second order spectrum of $A$ relative to a subspace $\mathcal{L}\subset\Dom(A)$, denoted $\Spec_2(A,\mathcal{L})$, consists of those $z\in\mathbb{C}$ for which there exists a $0\ne u\in\mathcal{L}$ such that
\[
\langle(A-z)u,(A-\overline{z})v\rangle=0\quad\forall v\in\mathcal{L}.
\]
\end{definition}
To apply the second order relative spectrum we need trial spaces which belong to the operator domain, rather than the preferred form domain. We must also assemble a matrix with entries of the form $\langle Au_i,Au_j\rangle$, which is also awkward. However, the method does have some nice properties: if $z\in\Spec_2(A,L)$ then
\begin{equation}\label{spec21}
\sigma(A)\cap\big[\Re z - \vert\Im z\vert,\Re z+\vert\Im z\vert\big]\ne\varnothing,
\end{equation}
if $(a,b)\cap\sigma(A)=\{\lambda\}$ and $a<\Re z<b$, then
\begin{equation}\label{spec22}
\lambda\in\left[\Re z - \frac{\vert\Im z\vert^2}{b-\Re z},\Re z+\frac{\vert\Im z\vert^2}{\Re z-a}\right];
\end{equation}
see \cite[corollaries 3.4 \& 4.2]{shar} and \cite[Remark 2.3]{me2}, respectively. We saw in Example \ref{diracex} that we obtain a sequence $z_n\in\Spec_2(A,\mathcal{L}_n)$ with $z_n\to\lambda_1$, let us compare the approximation of $\lambda_1$ by $\Re z_n$ to the size of the enclosures \eqref{spec21} and \eqref{spec22}; for the latter we may choose $a=-1$ and $b=\lambda_2$. We find that
\begin{align*}
\vert\lambda_1-\Re &z_n\vert=\mathcal{O}(n^{-0.7}),\quad2\vert\Im z_n\vert=\mathcal{O}(n^{-0.2}),\\
&\frac{\vert\Im z_n\vert^2}{b-\Re z_n}+\frac{\vert\Im z_n\vert^2}{\Re z_n-a}=\mathcal{O}(n^{-0.4}).
\end{align*}
which suggests that the enclosures obtained from the second order relative spectrum are very poor when compared to the approximation provided by $\Re z_n$. The latter, in turn, is poor when compared to the approximation provided by the perturbation method; see Example \ref{diracex}. In applications, by using the second order relative spectrum (or any other quadratic method) we can obtain intervals which intersect the spectrum, however, the actual approximate eigenvalues obtained are significantly compromised by using quadratic methods.
\subsection{The convergence of $\Spec_2(A,\mathcal{L}_n)$}
Using the well established convergence theory for the Galerkin method we will prove convergence properties for the second order relative spectra; see also \cite[Theorem 4.9, Theorem 6.1 \& Corollary 6.2]{me3}.
Unless stated otherwise we assume that $A$ is a bounded self-adjoint operator.
Consider the block matrix
\[
T:=\begin{pmatrix} 2A & -A^2 \\ I & 0 \end{pmatrix}:\mathcal{H}\oplus\mathcal{H}\to\mathcal{H}\oplus\mathcal{H}.
\]
\begin{lemma}{\cite[Lemma 3.1]{me3}}
$\sigma(T)=\sigma(A)$. If $\lambda\in\sigma_{\dis}(A)$ has multiplicity $m$, then  $\lambda$ is an eigenvalue of $T$ with algebraic multiplicity $2m$, geometric multiplicity $m$, ascent $2$.
\end{lemma}
\begin{lemma}{\cite[Lemma 3.2]{me3}}
$\Spec_2(A,\mathcal{L})=\sigma(T,\mathcal{L}\oplus\mathcal{L})$.
\end{lemma}
Let $(P_n)$ be a sequence of finite-rank orthogonal projections which converge strongly to the identity operator. The range of $P_n$ is denoted $\mathcal{L}_n$.  
\begin{lemma}{\cite[Theorem 4.4]{me3} see also \cite[proof of Theorem 1]{bo}}\label{mebo}
Let $\lambda\in\sigma_{\dis}(A)$ with $\dist(\lambda,\sigma(A)\backslash\{\lambda\})>r$. There exists a constant $c_r>0$ and an $N\in\mathbb{N}$, such that
\begin{equation}\label{nproof}
\Vert P_n(A-z)(A-z)P_nu\Vert\ge c_r\Vert P_nu\Vert
\end{equation}
for all $u\in\mathcal{H}$, $\vert\lambda-z\vert=r$, and $n\ge N$.
\end{lemma}
\begin{lemma}\label{nb}
Let $\lambda\in\sigma_{\dis}(A)$ with $\dist(\lambda,\sigma(A)\backslash\{\lambda\})>r$. There exists a constant $d_r>0$ and an $N\in\mathbb{N}$, such that
\begin{equation}\label{nproof2}
\left\Vert\begin{pmatrix} 2P_nAP_n - zP_n & -P_nA^2P_n \\ P_n & -zP_n \end{pmatrix}\begin{pmatrix} u \\ v \end{pmatrix}\right\Vert\ge d_r\left\Vert\begin{pmatrix} P_nu \\ P_nv \end{pmatrix}\right\Vert
\end{equation}
for all $u,v\in\mathcal{H}$, $\vert\lambda-z\vert=r$, and $n\ge N$.
\end{lemma}
\begin{proof}
Suppose the assertion is false. Then there exists a subsequence $n_j$, a sequence $(z_{n_j})$ with $\vert\lambda - z_{n_j}\vert=r$, and vectors $u_{n_j},v_{n_j}\in\mathcal{L}_{n_j}$ with $\Vert u_{n_j}\Vert^2+\Vert v_{n_j}\Vert^2=1$, such that
\[\left\Vert\begin{pmatrix} 2P_{n_j}A -z_{n_j} & -P_{n_j}A^2\\ I & -z_{n_j} \end{pmatrix}\begin{pmatrix} u_{n_j} \\ v_{n_j} \end{pmatrix}\right\Vert\to 0
\]
Without loss of generality we suppress the second subscript. We have
\begin{align}
2P_{n}Au_n -z_nu_n -P_{n}A^2v_n &\to 0\\
u_n - z_nv_n &\to 0.\label{two}
\end{align}
Then for some sequence of reals $0\le s_n\to 0$ and a sequence of normalised vectors $(w_n)$, we have
\[
u_n - z_nv_n = s_nw_n.
\]
Then
\[
-P_n(A-z)(A-z)v_n + s_n(2P_nA-z_n)w_n = 2P_{n}Au_n-zu_n -P_{n}A^2v_n\to0.
\]
Lemma \ref{mebo} implies that $v_n\to 0$. Then \eqref{two} implies that $u_n\to 0$. The result follows from the contradiction.
\end{proof}
Let us fix a $0\ne \lambda\in\sigma_{\dis}(A)$ (the case where $\lambda=0$ may be treated similarly by introducing a shift) and an $r<\dist(\lambda,\sigma(A)\backslash\{\lambda\})$ such that the circle $\vert\lambda -z\vert=r$ does not enclose zero. Denote by $\mathcal{M}$ the spectral subspace associated to $\lambda\in\sigma(T)$ and by $\mathcal{M}_n$ the spectral subspace associated to the operator
\begin{equation}\label{top}
\begin{pmatrix} 2P_{n}A & -P_{n}A^2 \\ I & 0 \end{pmatrix}:\mathcal{H}\oplus\mathcal{H}\to\mathcal{L}_n\oplus\mathcal{L}_n
\end{equation}
and the those eigenvalues enclosed by the circle $\vert\lambda - z\vert=r$.
\begin{lemma}{\cite[Theorem 4.6]{me3}}\label{sdim}
$\dim(\mathcal{M})=\dim(\mathcal{M}_n)$ for all sufficiently large $n$.
\end{lemma}
In view of Lemma \ref{nb} and Lemma \ref{sdim}, the operator \eqref{top} satisfies the definition of  \emph{strongly stable convergence} to $T$ in a neighbourhood of $\lambda$; see \cite[Chapter 5]{chat}. The following theorem is now a straightforward consequence of \cite[Theorem 6.11]{chat}.
\begin{theorem}\label{conn}
Let $z_n\in\Spec_2(A,\mathcal{L}_n)$ with $z_n\to\lambda$, then 
\[
\vert z_n-\lambda\vert=\mathcal{O}(\delta(\mathcal{L}(\{\lambda\}),\mathcal{L}_n))\quad\text{and}\quad
\vert\Re z_n-\lambda\vert=\mathcal{O}(\delta(\mathcal{L}(\{\lambda\}),\mathcal{L}_n)^2).\]
\end{theorem}
We now assume that $A$ is an unbounded self-adjoint operator. As above, $(P_n)$ denotes a sequence of finite-rank orthogonal projections each with range $\mathcal{L}_n$. We shall assume that
\[
\forall u\in\Dom(A)\quad\exists u_n\in\cL_n:\quad\Vert u-u_n\Vert_{A}\to0.
\]
The following theorem is now an immediate consequence of Theorem \ref{conn} and \cite[Lemma 2.6]{bost}.
\begin{theorem}\label{thmcon}
Let $z_n\in\Spec_2(A,\mathcal{L}_n)$ with $z_n\to\lambda$, then 
\[\vert z_n-\lambda\vert=\mathcal{O}(\delta_{A}(\mathcal{L}(\{\lambda\}),\mathcal{L}_n))\quad\text{and}\quad
\vert\Re z_n-\lambda\vert=\mathcal{O}(\delta_{A}(\mathcal{L}(\{\lambda\}),\mathcal{L}_n)^2).\]
\end{theorem}
The convergence rates in Theorem \ref{thmcon} are measured in terms of the graph norm which is why the method converges poorly; the convergence achieved by the Galerkin and perturbation methods is measured in terms of the norm associated to the quadratic form.
}

\end{document}